\documentclass[12pt]{amsart}
\usepackage{graphicx} 
\usepackage{caption}
\usepackage{subcaption}
\usepackage{url} 
\usepackage{hyperref}
\usepackage{amsfonts,amsthm,latexsym,amsmath,amssymb,amscd,amsmath,epsf} 
\usepackage{tikz}
\usepackage{verbatim}
\usetikzlibrary{snakes,arrows,shapes}

\theoremstyle{plain}
\newtheorem{theorem}{\bf Theorem}[section]
\newtheorem{lemma}[theorem]{\bf Lemma}
\newtheorem{proposition}[theorem]{\bf Proposition}
\newtheorem{corollary}[theorem]{\bf Corollary}

\theoremstyle{definition}
\newtheorem{definition}[theorem]{Definition}

\theoremstyle{remark}
\newtheorem{remark}[theorem]{\bf Remark}

\numberwithin{equation}{section}

\makeindex

\newcommand{\ZZ}{\mathbb{Z}}
\newcommand{\NN}{\mathbb{N}}

\newcommand{\E}{\mathcal{E}}
\newcommand{\F}{\mathcal{F}}

\newcommand{\T}{\mathcal{T}}
\newcommand{\BT}{\mathcal{BT}}

\newcommand{\BB}{\mathcal{B}}
\newcommand{\MM}{\mathcal{M}}

\newcommand{\sym}{\mathfrak{S}}
\newcommand{\lie}{\mathcal{L}ie}

\newcommand{\comb}{\textsf{Comb}^2}
\newcommand{\lyn}{\textsf{Lyn}^2}
\newcommand{\liu}{\textsf{Liu}^2}

\def\newop#1{\expandafter\def\csname #1\endcsname{\mathop{\rm #1}\nolimits}}

\newop{diag}
\newop{Tor}
\newop{supp}
\newop{per}
\newop{rank}
\newop{ker}
\newop{cl}
\newop{step}
\newop{col}
\newop{red}
\newop{blue}
\newop{length}
\newop{sgn}
\newop{val}
\newop{c}
\newop{root}
\newop{inv}
\newop{im}
\newop{couple}
\newop{free}
\newop{asc}
\newop{des}
\newop{red}
\newop{ch}
\newop{gr}

\title{On the (co)homology of the poset of weighted partitions}

\author[R. S. Gonz\'alez D'Le\'on]{Rafael S. Gonz\'alez D'Le\'on$^*$}
\address{Department of Mathematics, University of Miami, Coral Gables, FL 33124}
\email{dleon@math.miami.edu}
\thanks{$^*$Supported by NSF Grant  DMS 1202755}

\author[M. L. Wachs]{Michelle L. Wachs$^{**}$}
\address{Department of Mathematics, University of Miami, Coral Gables, FL 33124}
\email{wachs@math.miami.edu}
\thanks   {$^{**}$This work was partially supported by a grant from the Simons Foundation (\#267236 to Michelle Wachs) and by NSF Grants DMS 0902323 and DMS 1202755.}

\begin{document}
\allowdisplaybreaks


\begin{abstract}
We consider the poset of weighted partitions $\Pi_n^w$,  introduced by Dotsenko and Khoroshkin in 
their study of a certain pair of dual operads.   The maximal intervals of  $\Pi_n^w$ provide a 
generalization of the lattice $\Pi_n$ of  partitions,  which we show possesses many of the
well-known properties of $\Pi_n$.   
In particular,  we prove 
these intervals are EL-shellable, we show that the M\"obius invariant of each maximal interval is given up to sign by 
the number of  rooted trees on  node set $\{1,2,\dots,n\}$ having a fixed number of descents,
 we
find combinatorial bases for homology and 
cohomology, and we give an explicit sign twisted $\mathfrak{S}_n$-module isomorphism from 
cohomology to  
the multilinear component of the free Lie algebra with two compatible brackets.  We  also show that
the characteristic polynomial of $\Pi_n^w$ has 
a nice factorization analogous to that of  $\Pi_n$.\end{abstract}

\date{November 20, 2013; revised May 16, 2014; final version March 27, 2015}

\maketitle

\tableofcontents
\section{Introduction}\label{section:introduction}
We recall some combinatorial, topological and representation theoretic properties of the lattice
$\Pi_n$ of  partitions of the set  $[n]:=\{1,2,\dots,n\}$
ordered by refinement.\footnote{The poset terminology used here is defined in
Section~\ref{section:otheralgebraicinvariants}.}   The M\"obius invariant of $\Pi_n$ is given by 
$$\mu_{\Pi_n}(\hat 0, \hat 1) = (-1)^{n-1} (n-1)!,$$ 
 and  the characteristic polynomial by $$
\chi_{\Pi_n}(x) =  (x-1)(x-2) \dots (x-n+1)$$ (see \cite[Example 3.10.4]{Stanley2012}). 
 It was proved by Bj\"orner 
\cite{Bjorner1980}, using an edge labeling of Stanley \cite{Stan1974}, that $\Pi_n$ is  
EL-shellable;  consequently the order complex $\Delta(\overline{\Pi}_n)$  of the proper
part $\overline{\Pi}_n$ of the partition lattice $\Pi_n$ has the homotopy type of
a wedge of  $(n-1)!$ spheres of dimension $n-3$.  Various nice bases for the homology and cohomology
of the partition lattice have been introduced and studied; see \cite{Wachs1998} for a discussion of
these bases.

The symmetric group $\sym_n$ acts naturally on $\Pi_n$ and this action induces isomorphic
representations 
of $\sym_n$ on the unique nonvanishing reduced simplicial homology 
$\tilde H_{n-3}(\overline{\Pi}_n)$ of  the order complex $\Delta(\overline{\Pi}_n)$  and  on the 
unique nonvanishing simplicial cohomology  $\tilde H^{n-3}(\overline{\Pi}_n)$.
Joyal \cite{Joyal1986} observed that a formula
of Stanley and Hanlon (see  \cite{Stanley1982}) for the  character of this representation is a sign
twisted version 
of an  earlier formula of  Brandt \cite{Brandt1944} for the character of the representation of  
$\sym_n$  on the multilinear component  $\lie(n)$ of the free Lie algebra  on $n$ generators. Hence 
 the following $\sym_n$-module isomorphism  holds,
\begin{equation} 
\label{intro lie}  \tilde H_{n-3}(\overline{\Pi}_n) \simeq_{\sym_n} \lie(n) \otimes \sgn_n,
\end{equation}  
where  $\sgn_n$ is the sign representation of $\sym_n$. 
Joyal \cite{Joyal1986} gave a proof of the isomorphism using his theory of species.  
 The first purely 
combinatorial proof was obtained by Barcelo \cite{Barcelo1990} who provided  a bijection between
known 
bases for the two $\sym_n$-modules  (Bj\"orner's NBC basis for  $\tilde H_{n-3}(\overline{\Pi}_n)$ and the
Lyndon basis for $\lie(n)$) and analyzed the representation matrices for these bases.  Later Wachs
\cite{Wachs1998} gave a more general combinatorial proof 
by providing a natural bijection between generating sets of $\tilde H^{n-3}(\overline{\Pi}_n)$ and 
$\lie(n)$, which  revealed the strong connection between the two $\sym_n$-modules.

In this paper we explore analogous properties for  a weighted version of $\Pi_n$,    introduced by
Dotsenko and Khoroshkin 
\cite{DotsenkoKhoroshkin2007} in their study of  Koszulness of certain quadratic binary operads.   
A weighted partition of $[n]$ is a set $\{B_1^{v_1},B_2^{v_2},...,B_t^{v_t}\}$
where $\{B_1,B_2,...,B_t\}$ is a partition of $[n]$ and $v_i \in \{0,1,2,...,|B_i|-1\}$ for all $i$.
The {\it poset of weighted partitions} $\Pi_n^{w}$  is the set of weighted partitions of $[n]$ with 
order relation given by 
$\{A_1^{w_1},A_2^{w_2},...,A_s^{w_t}\}\le\{B_1^{v_1}, B_2^{v_2},...,B_t^{v_t}\}$ if the following
conditions hold:
\begin{itemize}
 \item $\{A_1,A_2,...,A_s\} \le \{B_1,B_2,...,B_t\}$ in $\Pi_n$
 \item if $B_k=A_{i_1}\cup A_{i_2}\cup ... \cup A_{i_l} $ then 
 $v_k-(w_{i_1} + w_{i_2} + ... + w_{i_l})\in \{0,1,...,l-1\}$.
\end{itemize}
Equivalently, we can define the covering relation  by 
$$\{A_1^{w_1},A_2^{w_2},...,A_s^{w_s}\}\lessdot \{B_1^{v_1}, B_2^{v_2},...,B_t^{v_t}\}$$ if the
following conditions hold:
\begin{itemize}
 \item $\{A_1,A_2,\dots,A_s\} \lessdot \{B_1,B_2,\dots,B_t\}$ in $\Pi_n$
 \item if $B_k=A_{i}\cup A_{j}$, where $i \ne j$, then $v_k-(w_{i} + w_{j}) \in \{0,1\}$
 \item if $B_k = A_i$ then $v_k = w_i$.
 \end{itemize}
In Figure~\ref{fign3k2}  below the set brackets and commas have been omitted.

\begin{figure}[h]

\begin{center} 
\begin{tikzpicture}[line join=bevel,scale=1]

\tikzstyle{every node}=[inner sep=0pt, scale=0.8, minimum width=4pt]
\node (n1232) at (3,4) {$123^{2}$};
  \node (n13020) at (-3,2) {$13^{ 0}| 2^{ 0}$};
  \node (n102030) at (0,0)  {$1^{0}| 2^{0}| 3^{0}$};
  \node (n1231) at (0,4) {$123^{1}$};
  \node (n12030) at (-5,2) {$12^{0}| 3^{0}$};
  \node (n13120) at (3,2)  {$13^{1}| 2^{0}$};
  \node (n1230) at (-3,4){$123^ {0}$};
  \node (n10230) at (-1,2)  {$1^{0}| 23^{0}$};
  \node (n12130) at (1,2)  {$12^{ 1}| 3^{0}$};
  \node (n10231) at (5,2) {$1^{0}| 23^ {1}$};

  \draw (n1231) -- (n10230) ;
  \draw [] (n13020) -- (n102030);
  \draw [] (n1232) -- (n13120);
  \draw [] (n1231)-- (n13020);
  \draw [] (n10230)--(n102030);
  \draw [] (n1230) -- (n10230);
  \draw [] (n1231) -- (n13120);
  \draw [] (n12030)-- (n102030);
  \draw [] (n1231) --(n12130);
  \draw [] (n1232) -- (n12130);
  \draw [] (n13120) --(n102030);
  \draw [] (n1231) --(n10231);
  \draw [] (n1230) -- (n13020);
  \draw [] (n1230)  -- (n12030);
  \draw [] (n12130) --  (n102030);
  \draw [] (n1232)  --  (n10231);
  \draw [] (n10231)  --  (n102030);
  \draw [] (n1231) -- (n12030);

\end{tikzpicture}
\end{center}
\caption[]{Weighted partition poset for $n=3$}\label{fign3k2}
\end{figure}
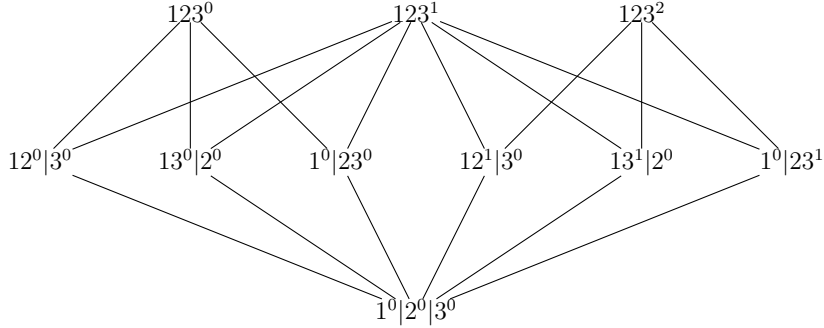

The poset $\Pi_n^w$ has a minimum element $$\hat 0:= \{\{1\}^0,\{2\}^0,\dots, \{n\}^0\}$$ and  $n$
maximal elements 
$$\{[n]^0\}, \, \{[n]^1\}, \dots,  \{[n]^{n-1}\}.$$  We write each maximal element $\{[n]^i\}$ as
$[n]^i$. Note  that  for all $i$, the maximal  intervals 
$[\hat 0, [n]^i]$ and $[\hat 0, [n]^{n-1-i}]$ are isomorphic to each other, and the two maximal 
intervals $[\hat 0, [n]^0]$ and $[\hat 0, [n]^{n-1}]$ are isomorphic to $\Pi_n$.  

The basic properties of $\Pi_n$ mentioned above have nice weighted analogs for the intervals $[\hat 0, [n]^i]$.  For 
instance,  the $\sym_n$-module isomorphism (\ref{intro lie}) can be generalized. Let $\lie_2(n)$ be the  multilinear 
component of the free Lie  
algebra  on $n$ generators with two compatible 
 brackets (defined in Section~\ref{subsection:genlie2})   and let $\lie_2(n,i)$ be  the component of $\lie_2(n)$ 
 generated by bracketed permutations with $i$ brackets of one type and $n-1-i$ brackets of the other type. The 
 symmetric group acts naturally on each $\lie_2(n,i)$ and on each open interval $(\hat 0, [n]^i)$.   It follows from  
 operad theoretic results of
Vallette \cite{Vallette2007} and Dotsenko-Khoroshkin \cite{DotsenkoKhoroshkin2010} that the following $\sym_n$-
module isomorphism holds:
\begin{equation}\label{gradelie} \tilde H_{n-3}((\hat 0, [n]^i))
\simeq_{\sym_n}  \lie_2(n,i)  \otimes \sgn_n.
\end{equation}  Note that this reduces to  (\ref{intro 
 lie}) when $i=0$ or $i=n-1$.  The character of each $\sym_n$-module $ \lie_2(n,i)$ was computed by Dotsenko 
 and 
 Khoroshkin \cite{DotsenkoKhoroshkin2007}.

 In \cite{Liu2010}  Liu proves a conjecture of Feigin that $\dim \lie_2(n) = n^{n-1}$ by constructing
 a combinatorial basis for $\lie_2(n)$  indexed by rooted
trees on node set $[n]$.   An operad theoretic proof of Feigin's conjecture  was obtained by Dotsenko and 
 Khoroshkin \cite{DotsenkoKhoroshkin2007}, but with a gap  pointed out in \cite{Strohmayer2008} and
corrected in 
 \cite{DotsenkoKhoroshkin2010}.
In fact, Liu and  Dotsenko-Khoroshkin  obtain the following refinement of Feigin's conjecture 
 \begin{equation} \label{introprodlie} \sum_{i=0}^{n-1}\dim \lie_2(n,i) t^i=  
\prod_{j=1}^{n-1} ((n-j) + jt).\end{equation}   
Since, as was proved by Drake \cite{Drake2008},  the  right hand side of (\ref{introprodlie})
is equal to the generating function for rooted trees on node set $[n]$ according to the number of 
descents of the tree, it follows that for each $i$,
the dimension of $\lie_2(n,i)$ equals  the number of rooted trees on node set
$[n]$ with $i$ descents.  (Drake's result is a refinement of the well-known result that the number
of trees on node set $[n]$ is $n^{n-1}$.) 

In
this paper we give an alternative proof of (\ref{gradelie}) by presenting an explicit  bijection between natural
 generating sets of  $\tilde H^{n-3}((\hat 0, [n]^i))$ and $
\lie_2(n,i)$, which reveals the connection between these  modules and
generalizes the  bijection that Wachs 
 \cite{Wachs1998} used to prove (\ref{intro lie}).  With  (\ref{gradelie}), we   take a  different path to proving the Liu 
 and 
 Dotsenko-Khoroshkin
formula (\ref{introprodlie}), one that employs  poset theoretic techniques.    

We prove that the augmented poset of weighted partitions 
$$\widehat{\Pi^w_n} := \Pi^w_n \cup \{ \hat 1\}$$   is EL-shellable by providing an interesting 
weighted analog of
the Bj\"orner-Stanley EL-labeling of 
$\Pi_n$.   In fact our labeling restricts to the Bj\"orner-Stanley EL-labeling on the intervals
$[\hat 0, [n]^0]$ and $[\hat 0, [n]^{n-1}]$.   A 
consequence of  shellability is that $\widehat{\Pi^w_n}$ is Cohen-Macaulay, which implies a result
of  Dotsenko and 
Khoroshkin \cite{DotsenkoKhoroshkin2010},
obtained through operad theory, that all  maximal intervals $[\hat 0, [n]^i]$ of $\Pi_n^w$ are
Cohen-Macaulay.  (Two prior attempts \cite{DotsenkoKhoroshkin2007,Strohmayer2008} to establish
Cohen-Macaulayness of $[\hat 0, [n]^i]$ are discussed in Remark~\ref{stro}.)  The ascent-free chains
 of
our EL-labeling provide a generalization of the Lyndon basis for cohomology of $\overline{\Pi}_n$ (i.e.
the basis for cohomology that corresponds to the classical Lyndon basis for $\lie(n)$).

 Direct
computation of the M\"obius function of $\Pi_n^w$, which  exploits the recursive nature of $\Pi^w_n$
and makes use of the compositional formula,  shows that 
$$(-1)^{n-1}\sum_{i=0}^{n-1}\mu_{\Pi_{n}^{w}}(\hat{0},[n]^i)t^{i}$$ equals the right hand side of (\ref{introprodlie}).  
From this computation
 and  the fact that $\widehat{\Pi_n^w}$ is EL-shellable (and thus the maximal intervals of $\Pi_n^w$
are Cohen-Macaulay), we conclude that
 \begin{equation} \label{introprod} \sum_{i=0}^{n-1} \rank \tilde H_{n-3}((\hat 0, [n]^i)) t^i=  
\prod_{j=1}^{n-1} ((n-j) + jt).\end{equation} The Liu and 
 Dotsenko-Khoroshkin
formula (\ref{introprodlie}) is  a consequence of this and  (\ref{gradelie}).

By (\ref{introprod}) and Drake's result mentioned above,  the rank of $ \tilde H_{n-3}((\hat 0,
[n]^i))$ is equal 
 to the number of rooted trees on $[n]$ with $i$ descents.  
  We construct a nice combinatorial basis for   
 $\tilde H_{n-3}((\hat 0, [n]^i))$ consisting of fundamental cycles  indexed by such rooted trees,
which generalizes 
 Bj\"orner's NBC basis for  $\tilde H_{n-3}(\overline{\Pi}_n)$.  Our proof  that these fundamental cycles form a basis 
 relies on Liu's \cite{Liu2010}
 generalization for $\lie_2(n,i)$ of the classical Lyndon basis for $\lie(n)$ and our bijective proof
of (\ref{gradelie}). Indeed, 
 our bijection enables us to transfer bases for 
$\lie_2(n,i)$ to bases for  $\tilde H^{n-3}((\hat 0, [n]^i))$ and vice verse.  
We first  transfer Liu's generalization of the Lyndon basis to $\tilde
H^{n-3}((\hat 0, [n]^i))$  and then use the natural pairing between homology and cohomology to prove
that our proposed homology basis is indeed a basis. (We also obtain an alternative proof that Liu's
generalization of the Lyndon basis is a basis along the way.)    By transferring the basis for
$\tilde H^{n-3}((\hat 0, [n]^i))$ that comes from the ascent-free chains of our
EL-labeling to $\lie_2(n,i)$, we  obtain a different generalization of the Lyndon basis that has a
somewhat simpler description than that of Liu's generalized Lyndon basis.

 The paper is organized as follows:
 In Section~\ref{section:otheralgebraicinvariants} we derive basic properties of the weighted
partition lattice, which include the formula for the 
 M\"obius function of $\Pi^w_n$ mentioned above.
 We also show that  
the M\"obius invariant of  the augmented poset of weighted partitions $\widehat{\Pi^w_n}$ is given 
by  $$\mu_{\widehat{\Pi^w_n}}(\hat 0,\hat 1)= (-1)^n(n-1)^{n-1}$$ and  the
characteristic polynomial  factors nicely as $$\chi_{\Pi^w_n}(x)= (x-n)^{n-1}.$$  The Whitney numbers of the
first and second  kind are also discussed.
 
 Section~\ref{section:topology} contains our results on  EL-shellability of the augmented poset of
weighted partitions and  its topological consequences. 
 
  In Section~\ref{section:conlie} we give a presentation of the cohomology of the maximal open
intervals $(\hat 0, [n]^i)$  in terms 
  of maximal chains associated with labeled bicolored binary trees.  This presentation enables us to
use a natural 
  bijection between  generating sets of  $ \tilde H^{n-3}((\hat 0, [n]^i))$ and
 $\lie_2(n,i)$ to establish the  $\sym_n$-module isomorphism (\ref{gradelie}).  
  Bases for cohomology and  for homology of $(\hat 0, [n]^i)$ are discussed in
Section~\ref{section:combinatorialbases}.   We also construct  bases for cohomology of the full poset
$\Pi_n^w \setminus \{\hat 0 \}$.
   
By extending the technique of Section~\ref{section:conlie}, we prove in
Section~\ref{section:whitney} that Whitney homology of $\Pi_n^w$  tensored with the sign
representation is isomorphic to the multilinear component of  the exterior algebra of the doubly
bracketed free Lie algebra on $n$ generators.   In Section~\ref{section:futurework} we mention
related results that will appear in  forthcoming papers.

\section{Basic properties}\label{section:otheralgebraicinvariants}

 For poset terminology not defined here see \cite{Stanley2012}, \cite{Wachs2007}. For $u  \le v$ in
a poset $P$,  the open interval $\{w \in P : u < w < v\}$ is denoted  by $(u,v)$ and the closed
interval $\{w \in P : u \le w \le v\}$ by $[u,v]$. 
A poset is said to be {\em bounded} if it has a minimum element $\hat 0$ and a maximum element 
$\hat 1$.  For a bounded poset $P$, we define the {\em proper part} of $P$ as $\overline {P}:=
P\setminus\{\hat 0,\hat 1\}$.    A poset is said to 
be {\em pure} (or ranked)  if all its maximal chains have the same length, where the length of a 
chain $s_0<s_1 < \dots < s_n$ is $n$.  The {\em length} $l(P)$ of a poset $P$ is the length of 
its longest chain.  For a poset $P$ with a minimum element $\hat 0$, the rank function $\rho:P \to
\NN$ is defined  by $\rho(s) = l([\hat 0, s])$.  The rank generating function $\F_P(x)$ is defined
by 
$\F_P(x) = \sum_{u \in P} x^{\rho(u)}$.

\subsection{The rank generating function} \label{ranksec}

It is easy to see that  the weighted partition poset $\Pi_n^w$ is pure of length $n-1$ and has 
minimum element $\hat 0 =\{\{1\}^0,\dots,\{n\}^0\}$.   For each $\alpha \in \Pi_n^w$, we have
$\rho(\alpha)=n-|\alpha|$.

\begin{proposition}  \label{rankprop} For all $n \ge 1$,  the rank generating function is given by 
\[\F_{ \Pi_n^{w}}(x)=\sum_{k=0}^{n-1}\binom{n}{k}(n-k)^k x^k.\]
\end{proposition}

\begin{proof}
Let $R_n(k)=\{ \alpha \in \Pi_n^{w}|\, \rho(\alpha)=k\}$.
We need to show 
 \begin{equation}\label{rankeq} |R_n(k)|=\binom{n}{n-k}(n-k)^k. \end{equation} 
 A weighted partition in $ R_n(k)$ can be viewed as a partition of $[n]$ into $n-k$ blocks,    with
one  element of each block   marked (or distinguished).  To choose such a partition, we first choose
the $n-k$ marked elements.  There are $\binom n {n-k}$ ways to choose these elements and place them
in $n-k$ distinct blocks.  To each of the remaining $k$ elements we allocate one of these $n-k$ 
blocks.  We can do this in $(n-k)^k$ ways.  Hence (\ref{rankeq}) holds. 
\end{proof}

\subsection{The M\"obius function} \label{subsection:mobius}

For $\alpha = \{A_1^{w_1},\dots,A_k^{w_k}\}  \in \Pi_n^w$, let $w(\alpha)=\sum_{i=1}^k w_i$.
 The following observations will be used to compute the M\"obius function 
 of the weighted partition
poset.  
\begin{proposition}\label{proposition:upperlowerideals}
 For all  $\alpha = \{A_1^{w_1},\dots,A_k^{w_k}\} \in \Pi_n^{w}$, 
 \begin{enumerate}
 \item $[\alpha,\hat 1]$ and $\widehat{\Pi_k^{w}}$ are isomorphic posets,
 \item $[\alpha, [n]^i]$ and $[\hat 0, [|\alpha|]^{i-w(\alpha)}]$ are isomorphic posets for
$w(\alpha) \le i\le n-1$,
 \item $[\hat{0},\alpha]$ and  $[\hat{0},[|A_1|]^{w_1}] \times \cdots \times [\hat 0,
[|A_k|]^{w_k}]$ 
are isomorphic posets.
\end{enumerate}
\end{proposition}

For a bounded poset $P$, let  $\mu_P$ denote its M\"obius function.
We will use the recursive definition of the M\"obius function and the compositional formula to
derive the following result.

\begin{proposition}\label{proposition:muweightedsumtrees}  For all $n \ge 1$,
\begin{equation} \label{3mobius}
\sum_{i=0}^{n-1}\mu_{\Pi_{n}^{w}}(\hat{0},[n]^i)t^{i}
=(-1)^{n-1}\prod_{i=1}^{n-1}((n-i)+it).
\end{equation}
Consequently,\[
\sum_{i=0}^{n-1}\mu_{\Pi_{n}^{w}}(\hat{0},[n]^i)=(-1)^{n-1}n^{n-1}
.\]
\end{proposition}

\begin{proof}
By the recursive definition of  the M\"obius function we have that
\[
\sum_{i=0}^{n-1}t^{i}
\sum_{\hat{0}\le\alpha\le[n]^i}\mu_{\Pi_{n}^{w}}(\alpha,[n]^i)
=\delta_{n,1}.
\]
Proposition~\ref{proposition:upperlowerideals} implies
$\mu_{\Pi_{n}^{w}}(\alpha,[n]^i) = \mu_{\Pi_{|\alpha|}^{w}}(\hat{0},[|\alpha|]^j)$, where $j=
i-w(\alpha)$.  Note also that $\hat 0 \le \alpha \le [n]^i$ if and only if $w(\alpha) \le i$ and $i -w(\alpha) \le |\alpha| -1$.
Hence,
\begin{align*}
 \delta_{n,1}
	    &= 	\sum_{\alpha\in \Pi_{n}^{w}}
		t^{w(\alpha)}
		\sum_{i=w(\alpha)}^{w(\alpha)+|\alpha|-1}
		\mu_{\Pi_{n}^{w}}(\alpha,[n]^i)t^{i-w(\alpha)}\\
	    &= 	\sum_{\alpha \in \Pi_{n}^{w}}
		t^{w(\alpha)}
		\sum_{j=0}^{|\alpha|-1}
		\mu_{\Pi_{|\alpha|}^{w}}(\hat{0},[|\alpha|]^j)t^{j}\\
	    &= 	\sum_{\pi \in \Pi_{n}}
		\left(\prod_{B \in \pi }(t^{|B|-1}+t^{|B|-2}+\cdots+1)\right)
		\sum_{j=0}^{|\pi|-1}
		\mu_{\Pi_{|\pi|}^{w}}(\hat{0},[|\pi|]^j)t^{j}\\\
            &= 	\sum_{\pi \in \Pi_{n}}
		\left(\prod_{B \in \pi }\dfrac{t^{|B|}-1}{t-1}\right)
		\sum_{j=0}^{|\pi|-1}
		\mu_{\Pi_{|\pi|}^{w}}(\hat{0},[|\pi|]^j)t^{j}.
 \end{align*}
This implies by the compositional formula (see \cite[Theorem~5.1.4]{Stanley1999}) that
\[
U(x)=\sum_{n \ge 1}\dfrac{t^{n}-1}{t-1}\frac{x^n}{n!}=\dfrac{e^{tx}-e^{x}}{t-1}
\]
and
\[
W(x)=\sum_{n \ge 1}\sum_{j=0}^{n-1}
		\mu_{\Pi_{n}^{w}}(\hat{0},[n]^j)t^{j}\frac{x^n}{n!}
\]
are compositional inverses.  

It follows from  \cite[Theorem 5.1]{GesselSeo2004}  that the compositional inverse of 
$U(x)$ is given by 
$$\sum_{n \ge 1}(-1)^{n-1}\prod_{i=1}^{n-1}((n-i)+it)\frac{x^n}{n!}.$$ (See \cite[Eq.
(10)]{Drake2008}.)
This yields (\ref{3mobius}).
\end{proof}

Let $T$ be a rooted tree on node set $[n]$.  A {\it descent} of $T$ is a node $x$ that has a smaller
label than its parent $p_T(x)$.   
We call the edge $\{x,p_T(x)\} $ a {\it descent edge}.  
We denote by $\T_{n,i}$  the set of 
 rooted trees on node set $[n]$ with exactly $i$ descents. 
In \cite{Drake2008} Drake proves that
\begin{equation}  \label{equation:drake} \sum_{i = 0}^{n-1} |\T_{n,i}| t^i =
\prod_{i=1}^{n-1}((n-i)+it).\end{equation}
The following result is a consequence of this and 
 Proposition~\ref{proposition:muweightedsumtrees}.

\begin{corollary}\label{corollary:mutrees} For all $n \ge 1$ and $i \in \{0,1,\dots,n-1\}$, 
\[ \mu_{\Pi_{n}^{w}}(\hat{0},[n]^{i})=(-1)^{n-1}|\T_{n,i}|.\]
\end{corollary}

  We can use Proposition~\ref{proposition:upperlowerideals} and 
Corollary~\ref{corollary:mutrees} to compute the M\"obius function on other intervals. 
 A {rooted
forest} 
on node set $[n]$ is a set of rooted trees whose node sets form a partition of $[n]$.
We associate a weighted partition $\alpha(F)$ with each   rooted forest $F= \{T_1,\dots,T_k\}$ on
node set $[n]$, by letting  $\alpha(F) = \{A_1^{w_1},\dots, A_k^{w_k}\}$ where $A_i$ is the node set
 of $T_i$  and $w_i$ is the number of descents  of $T_i$.  
 For lower intervals we obtain the following generalization of Corollary~\ref{corollary:mutrees}.

\begin{corollary} \label{corfor} For all $\alpha \in \Pi_n^w$,
$$\mu_{\Pi_n^w}(\hat 0, \alpha) = (-1)^{n-|\alpha|} |\{ F \in \mathcal F_n : \alpha(F) =
\alpha\}|,$$
where $\mathcal F_n$ is the set of rooted forests on node set $[n]$.
\end{corollary}

Next we consider the full poset $\widehat{\Pi_n^{w}}$.  To compute its M\"obius invariant we will
make use of Abel's identity (see \cite[Ex. 5.31 c]{Stanley1999}), 
\begin{equation}\label{proposition:abelidentity}
(x+y)^n=\sum_{k=0}^{n}\binom{n}{k}x(x-kz)^{k-1}(y+kz)^{n-k}.
\end{equation}

\begin{proposition}\label{proposition:mobiushat}
\[\mu_{\widehat{\Pi_n^{w}}}(\hat{0},\hat{1})=(-1)^n(n-1)^{n-1}.\]
\end{proposition}
\begin{proof}
We proceed by induction on $n$.
If $n=1$ then $$\mu_{\widehat{\Pi_1^{w}}}(\hat{0},\hat{1})=-1=(-1)^1(1-1)^{1-1}$$ since
$\widehat{\Pi_1^{w}}$ is the chain of length 1.

Let $n \ge 1$ and let $\alpha \in \Pi_n^{w} \setminus \{\hat{0}\}$.  Since the interval $[\alpha,
\hat{1}]$ in $\widehat{\Pi_n^{w}}$
is  isomorphic to  $\widehat{\Pi_{|\alpha|}^{w}}$ (cf.
Proposition~\ref{proposition:upperlowerideals}),  we  can assume by induction that 
\[
\mu_{\widehat{\Pi_{n}^{w}}}(\alpha,\hat{1})=(-1)^{|\alpha|}(|\alpha|-1)^{|\alpha|-1}.
\] 
Hence by the recursive definition of the M\"obius function we have,
\begin{eqnarray} \nonumber
\mu_{\widehat{\Pi_{n}^{w}}}(\hat{0},\hat{1}) &=& - \sum_{\alpha \in \widehat{\Pi_n^{w}}\setminus
\hat 0} \mu_{\widehat{\Pi_{n}^{w}}}(\alpha,\hat{1})
\\ \nonumber &=& -1 - \sum_{k=1}^{n-1} \sum_{\substack{\alpha \in \Pi_n^{w}\\ |\alpha| = k}
}\mu_{\widehat{\Pi_{n}^{w}}}(\alpha,\hat{1})
\\ \nonumber &=& -1 - \sum_{k=1}^{n-1}\sum_{\substack{\alpha \in \Pi_n^{w}\\|\alpha|=k}}
(-1)^{k}(k-1)^{k-1}
\\ \nonumber &=& -1 - \sum_{k=1}^{n-1} \binom{n}{	k}k^{n-k}(-1)^{k}(k-1)^{k-1}
\hspace{.1in}\mbox{ (by (\ref{rankeq}))}
\\ \label{mobeq} &=& -1 + \sum_{k=0}^{n} \binom{n}{	k}k^{n-k}(1-k)^{k-1} - (1-n)^{n-1}.
\end{eqnarray}
By setting $x=1,y=0,z=1$ in  Abel's identity (\ref{proposition:abelidentity}), we get 
$$
1=\sum_{k=0}^n\binom{n}{k}(1-k)^{k-1}k^{n-k}.$$
Substituting this into (\ref{mobeq}) yields the result.
\end{proof}

\begin{remark}In Section~\ref{subsection:characteristicpolynomial} we compute the characteristic
polynomial of $\Pi_n^w$ and use it to give a second proof of
Proposition~\ref{proposition:mobiushat}.
\end{remark}

\subsection{The characteristic polynomial}\label{subsection:characteristicpolynomial}

Recall that the characteristic polynomial of $\Pi_n$ factors nicely.  
We prove that the same is true for $\Pi_n^w$.

\begin{theorem}\label{proposition:characteristicpolynomial}
For all $n \ge 1$, the characteristic polynomial of $\Pi_n^w$ is given by 
\[\chi_{\Pi_n^{w}}(x):=\sum_{\alpha \in
\Pi_n^w}\mu_{\Pi_n^w}(\hat{0},\alpha)x^{n-1-\rho(\alpha)}=(x-n)^{n-1}. \]
\end{theorem}

We will need  the following result.
\begin{proposition}[see {\cite[Proposition 5.3.2]{Stanley1999}}]\label{proposition:numberofforests}
Let $\F_n^k$ be the number of rooted forests on node set $[n]$ with $k$ rooted trees.  Then
\[
 |\F_n^k|=\binom{n-1}{k-1}n^{n-k}.
\]
\end{proposition}

\begin{proof}[Proof of Theorem~\ref{proposition:characteristicpolynomial}]
We have
 \begin{align*}
  \chi_{\Pi_n^{w}}(x)
&= \sum_{\alpha \in \Pi_n^{w}}  \mu(\hat{0},\alpha)x^{|\alpha|-1}\\
&= \sum_{k=1}^n \sum_{\substack{\alpha \in \Pi_n^{w}\\|\alpha|=k}}  \mu(\hat{0},\alpha)x^{k-1}\\
&=      \sum_{k=1}^n (-1)^{n-k}|\mathcal F_n^k| x^{k-1} \hspace{.1in} \mbox{ (by
Corollary~\ref{corfor})} \\
&= \sum_{k=1}^n (-1)^{n-k} \binom{n-1}{k-1}n^{n-k} x^{k-1} \hspace{.1in} \mbox{ (by Proposition
\ref{proposition:numberofforests})} \\
&= \sum_{k=0}^{n-1}  \binom{n-1}{k}(-n)^{n-1-k} x^{k}\\
&= (x-n)^{n-1}.
 \end{align*}

\end{proof}

Theorem~\ref{proposition:characteristicpolynomial} yields an easier way to calculate
$\mu_{\widehat{\Pi_n^w}}(\hat 0, \hat 1)$.

\begin{proof}[Second proof of Proposition \ref{proposition:mobiushat} ]
 By the recursive definition of M\"obius function,
\begin{eqnarray*}
 \mu_{\widehat{\Pi_n^{w}}}(\hat{0},\hat{1})&=& - \sum_{\alpha \in \Pi_n^{w} } \mu(\hat{0},\alpha)\\
&=& - \chi_{\Pi_n^{w}}(1)\\
&=&-(1-n)^{n-1}\\
&=&(-1)^n(n-1)^{n-1}.
\end{eqnarray*}

\end{proof}

\subsection{Whitney numbers and uniformity}
Let $P$ be a pure poset of length $n$ with minimum element $\hat 0$. Recall that  the \emph{Whitney
number of the first kind} $w_k(P)$ is the coefficient of $x^{n-k}$ in the characteristic polynomial 
$\chi_P(x)$ and the \emph{Whitney number of the second kind} $W_k(P)$ is the coefficient of $x^{k}$
in the rank generating
function $\F_P(x)$; see \cite{Stanley2012}.  It follows from  
Theorem~\ref{proposition:characteristicpolynomial} and 
Proposition~\ref{rankprop}, respectively,
 that 
\begin{eqnarray} \label{whiteq} w_k(\Pi_n^w) &=& (-1)^k\binom{n-1}{k}n^k \\ \nonumber W_k(\Pi_n^w)
&=&\binom{n}{k}(n-k)^k.   
\end{eqnarray}

For the partition lattice $\Pi_n$, the Whitney numbers of the first and second kind are the Stirling
numbers of the first and second kind.   It is well-known that the Stirling numbers  of the first
kind and second kind form inverse matrices, cf., \cite[Proposition 1.9.1 a]{Stanley2012}.  
This can be viewed as a consequence of a property of the partition lattice  called uniformity \cite[Ex.
3.130]{Stanley2012}.  We observe in this section that $\Pi_n^w$ is also uniform and
discuss a Whitney number consequence.

A pure   poset $P$ of length $l$ with minimum element $\hat 0$ and with rank function $\rho$, is said to 
be {\it uniform}  if
there is a family of posets $\{P_i : 0 \le i \le l\}$ such that for all 
$x\in P$, the  upper order ideal  $I_x:=\{y \in P : x \le y\}$ is isomorphic to $P_i$, where $i = l- \rho(x)$.
We refer to
$(P_0,\dots,P_l)$  as the associated {\em uniform sequence}.
It follows from Proposition~\ref{proposition:upperlowerideals} that $P=\Pi_n^w$ is uniform
with $P_i =\Pi_{i+1}^w$ for $i=0,\dots, n-1$.  We will use the following variant 
of {\cite[Exercise 3.130(a)]{Stanley2012}} whose proof is left to the reader.  (A weighted version of this is 
proved in \cite{Dleon2013a}.)

\begin{proposition} \label{uniprop}Let $P$ be a uniform
poset of length $l$, with associated uniform sequence $(P_0,\dots,P_l)$.  
Then the matrices  $[w_{i-j}(P_i)]_{0\le i,j \le l}$ and $[W_{i-j}(P_i)]_{0\le i,j \le l}$ are
inverses of each other.
\end{proposition}

From the uniformity of $\Pi_n^w$ and (\ref{whiteq}),  we have the following consequence of
Proposition~\ref{uniprop}.
\begin{corollary} The matrices $ A=[(-1)^{i-j} \binom {i-1} {j-1}i^{i-j} ]_{1\le i,j \le n}$ and 
$B=[\binom i {j} j^{i-j}]_{1\le i,j \le n}$ are inverses of each other.  
\end{corollary}

This result  is not new and an equivalent dual version 
(conjugated by the matrix $[(-1)^j \delta_{i,j}]_{1\le i,j \le n}$)
was already obtained by Sagan in \cite{Sagan1983}, 
also by using essentially Proposition~\ref{uniprop}, but with a completely different poset.  
So we can consider this to be a new proof of that result (see also \cite{JoniRota1981}).

Chapoton and Vallette \cite{ChapotonVallette2006} consider another poset  that is quite similar to
the poset of weighted partitions, namely the poset of pointed partitions.  A pointed partition of
$[n]$ is a partition of $[n]$ in which one element of each block is distinguished.  The covering
relation is given by 
$$\{(A_1,{a_1}),(A_2,{a_2}),...,(A_s,{a_s})\}\lessdot \{(B_1,{b_1}),
(B_2,{b_2}),...,(B_t,{b_t})\},$$
where $a_i$ is the distinguished element of $A_i$ and $b_i$ is the distinguished element of $B_i$
for each $i$, if the following conditions hold:
\begin{itemize}
\item $\{A_1,A_2,\dots,A_s\} \lessdot \{B_1,B_2,\dots,B_t\}$ in $\Pi_n$ \item if $B_k=A_{i}\cup
A_{j}$, where $i \ne j$, then $b_k \in \{a_i,a_j\}$
 \item if $B_k = A_i$ then $b_k = a_i$.
\end{itemize}
Let $\Pi_n^p$ be the poset of pointed partitions of $[n]$.  It is easy to see that there is a rank
preserving bijection between $\Pi_n^w$ and $ \Pi_n^p$.  It follows that both posets have the same
Whitney numbers of the second kind.   Since  both posets are uniform, it follows from
Proposition~\ref{uniprop} that both posets have the same Whitney numbers of the  first kind and thus the same 
characteristic polynomial.  
The following result of Chapoton and Vallette \cite{ChapotonVallette2006} is therefore equivalent to
Theorem~\ref{proposition:characteristicpolynomial}.

\begin{corollary}[Chapoton and Vallette \cite{ChapotonVallette2006}] For all $n \ge 1$, the
characteristic polynomial of $\Pi_n^p$ is given by 
\begin{equation}\label{equation:pointed} \chi_{\Pi_n^{p}}(x)=(x-n)^{n-1}. \end{equation}
 Consequently,
$$\mu_{\widehat{\Pi_n^p}}(\hat 0, \hat 1) = (-1)^n(n-1)^{n-1}.$$
\end{corollary}

One can  also compute the M\"obius function for all intervals of $\Pi_n^p$ from (\ref{equation:pointed}). 
 Indeed, since all $n$
maximal intervals are isomorphic to 
each other, the M\"obius invariant can be obtained from (\ref{equation:pointed}) by setting $x = 0$
and then dividing by $n$.  This yields 
for all $i$, $$(-1)^n \mu_{\widehat{\Pi_n^{p}}}(\hat 0, ([n],i)) =  n^{n-2},$$ which is
the number of trees on node set 
$[n]$.   The M\"obius function on other intervals  can be computed from this since all intervals of
$\Pi_n^p$ are 
isomorphic to products of maximal intervals of ``smaller"  posets of pointed partitions.

\section{Homotopy type of the poset of weighted partitions} \label{section:topology}

In this section we use EL-shellability to  determine the homotopy type of the intervals of
$\widehat{\Pi_n^w}$ and to show that $\widehat{\Pi_n^w}$ is Cohen-Macaulay, extending a result of
Dotsenko and Khoroshkin \cite{DotsenkoKhoroshkin2010}, in which operad theory is used to prove that
all intervals of $\Pi_n^w$ are Cohen-Macaulay.  Some prior attempts to establish shellability of the
maximal intervals are discussed in Remark~\ref{stro}.

\subsection{EL-shellability}\label{subsection:cohenmacaulayness}

After reviewing some basic facts from the theory of lexicographic
shellability  (cf.~\cite{Bjorner1980}, \cite{BjornerWachs1983}, \cite{BjornerWachs1996},
\cite{Wachs2007}), we will present our main results on lexicographic shellability of the poset of
weighted partitions.

An {\em edge labeling} of a bounded poset $P$ is a map $\lambda: \mathcal E(P) \to
\Lambda$, where
$\mathcal E(P)$ is the set of edges of the Hasse diagram of $P$, i.e., the covering
relations $x <\!\!\!\!\cdot \,\, y$ of $P$, and $\Lambda$ is some poset. Given an edge labeling
$\lambda:
\mathcal E(P) \to \Lambda$, one can associate  a label word $$\lambda(c) = \lambda(x_0, x_1)
\lambda(x_1, x_2) \cdots \lambda(x_{t-1}, x_{t})$$  with each maximal chain $c = (\hat 0 = x_0
<\!\!\!\!\cdot
\,\,x_1 <\!\!\!\!\cdot \,\,
\cdots<\!\!\!\!\cdot\,\, x_{t-1} <\!\!\!\!\cdot\,\, x_t= \hat 1)$.   We say that  $c
$ is  {\em increasing} if its label word $\lambda(c)$ is
{\em strictly}  increasing.   That is, $c$ is  increasing if 
$$ \lambda(x_0, x_1) <
\lambda(x_1, x_2)<  \cdots < \lambda(x_{t-1}, x_t).$$  We say that  $c
$ is  {\em ascent-free} (or decreasing, falling) if its label word $\lambda(c)$ has no ascents, i.e.
$  \lambda(x_i, x_{i+1}) \not<  \lambda(x_{i+1}, x_{i+2}) $, for all $i=0,\dots,t-2$. We can
partially order the
maximal chains lexicographically by using the lexicographic order on the corresponding label
words.  Any
edge labeling
$\lambda$ of
$P$ restricts to an edge labeling of each closed interval $[x,y]$ of $P$.  So we may refer
to increasing and ascent-free maximal chains of $[x,y]$, and lexicographic order of maximal
chains of
$[x,y]$.

 \begin{definition}  Let $P$ be a bounded poset. {\em  An edge-lexicographical
labeling} (EL-labeling, for short)  of
$P$ is an edge labeling such that in each closed
interval $[x,y]$ of $P$, there is a unique  increasing maximal chain, and this chain
lexicographically precedes all other maximal chains of $[x,y]$.  A poset that admits an EL-labeling
is said to be {\em EL-shellable}.
\end{definition}

Note that if $P$ is EL-shellable then so is every closed interval of $P$.

A classical EL-labeling for the partition lattice $\Pi_n$ is obtained as follows.  Let $\Lambda =
\{(i,j)\in [n-1] \times [n] : i <j\}$ with lexicographic order as the order relation on $\Lambda$. 
If $x\lessdot y $ in $\Pi_n$  then $y$ is obtained from $x$ by merging two blocks $A$ and $B$, 
where $\min A < \min B$.
 Let $\lambda(x,y) = (\min A, \min B)$.  This defines a map $\lambda:\mathcal E(\Pi_n) \to \Lambda$.
 By viewing $\Lambda$ as  the set of atoms of $\Pi_n$, one sees that this labeling is a special case
of an  edge labeling for geometric lattices, which first appeared in Stanley \cite{Stan1974} and was
one of
 Bj\"orner's \cite{Bjorner1980} initial examples of an EL-labeling.  
 
 We now generalize the Bj\"orner-Stanley EL-labeling of $\Pi_n$ to the weighted partition lattice. 
For each $a \in [n]$, let $\Gamma_a:= \{(a,b)^u :   a<b \le n+1, \,\, u \in \{0,1\} \}$. 
We
partially order $\Gamma_a$  by letting $(a,b)^u \le  (a,c)^v$ if $b\le  c$ and $u \le v$.   
Note that $\Gamma_a$ is isomorphic to the direct product of the chain $a+1< a+2 <\dots < n+1 $ and
the chain $0 < 1$.  Now define $\Lambda_n$ to be the 
ordinal sum
$\Lambda_n := \Gamma_1 \oplus  \Gamma_2  \oplus \cdots \oplus \Gamma_{n}$.  (See Figure
\ref{fig:lambdaposet}.)

 If $x\lessdot y $ in $\Pi^w_n$   then $y$ is obtained from $x$ by merging two blocks $A$ and $B$, 
where $\min A < \min B$, and  assigning weight $u + w_A + w_B$ to the resulting block $A \cup B$,
where $u \in \{0,1\}$, and $w_A$, $w_B$ are the respective weights of $A$ and $B$ in the weighted
partition $x$.  Let $$\lambda(x \lessdot y) = (\min A, \min B)^u.$$  This defines a map
$\lambda:\mathcal E(\Pi^w_n) \to \Lambda_n$.  We extend this map to $\lambda:\mathcal
E(\widehat{\Pi^w_n})\to \Lambda_n$ by letting ${\lambda}([n]^i \lessdot \hat{1})=(1,n+1)^0$, for all
$i=0,\dots,n-1$. (See Figure~\ref{fig:ellabelingposet}.)  Note that when $\lambda$ is restricted to
the intervals $[\hat 0, [n]^0]$ and $[\hat 0, [n]^{n-1}]$,  which are both isomorphic to $\Pi_n$,
the labeling reduces to the Bj\"orner-Stanley EL-labeling of $\Pi_n$.   
 
 \begin{theorem}\label{theorem:ellabelingposet}
 The labeling $\lambda:\E(\widehat{\Pi_{n}^w})\rightarrow \Lambda _n$ defined above is an
EL-labeling of $\widehat{\Pi_{n}^w}$.  
\end{theorem}

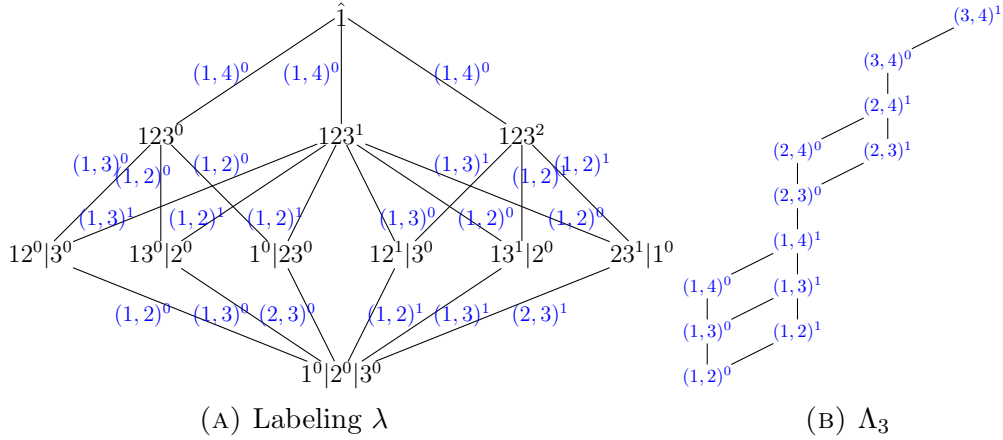
\begin{figure}
        \centering
        \begin{subfigure}[b]{0.6\textwidth}
                \centering
               \begin{tikzpicture}[line join=bevel,scale=0.8]

\tikzstyle{every node}=[inner sep=0pt, scale=0.8, minimum width=4pt]
\node (hat1) at (0,6) {$\hat{1}$};
\node (n1232) at (3,4) {$123^{2}$};
  \node (n13020) at (-3,2) {$13^{ 0}| 2^{ 0}$};
  \node (n102030) at (0,0)  {$1^{0}| 2^{0}| 3^{0}$};
  \node (n1231) at (0,4) {$123^{1}$};
  \node (n12030) at (-5,2) {$12^{0}| 3^{0}$};
  \node (n13120) at (3,2)  {$13^{1}| 2^{0}$};
  \node (n1230) at (-3,4){$123^ {0}$};
  \node (n10230) at (-1,2)  {$1^{0}| 23^{0}$};
  \node (n12130) at (1,2)  {$12^{ 1}| 3^{0}$};
  \node (n10231) at (5,2) {$23^ {1}| 1^{0}$};

  \draw (n1231) -- (n10230) ;
  \draw [] (n13020) -- (n102030);
  \draw [] (n1232) -- (n13120);
  \draw [] (n1231)-- (n13020);
  \draw [] (n10230)--(n102030);
  \draw [] (n1230) -- (n10230);
  \draw [] (n1231) -- (n13120);
  \draw [] (n12030)-- (n102030);
  \draw [] (n1231) --(n12130);
  \draw [] (n1232) -- (n12130);
  \draw [] (n13120) --(n102030);
  \draw [] (n1231) --(n10231);
  \draw [] (n1230) -- (n13020);
  \draw [] (n1230)  -- (n12030);
  \draw [] (n12130) --  (n102030);
  \draw [] (n1232)  --  (n10231);
  \draw [] (n10231)  --  (n102030);
  \draw [] (n1231) -- (n12030);	
	\draw (hat1) -- (n1230);
	\draw (hat1) -- (n1231);
	\draw (hat1) -- (n1232);

\tikzstyle{every node}= [scale=0.7]

\node  at (-3.3,1) {\color{blue}$(1,2)^0$};
\node  at (-2,1) {\color{blue}$(1,3)^0$};
\node  at (-0.9,1) {\color{blue}$(2,3)^0$};
\node  at (0.9,1) {\color{blue}$(1,2)^1$};
\node  at (2,1) {\color{blue}$(1,3)^1$};
\node  at (3.3,1) {\color{blue}$(2,3)^1$};

\node  at (-4,3.5) {\color{blue}$(1,3)^0$};
\node  at (-3.3,3.3) {\color{blue}$(1,2)^0$};
\node  at (-2,3.5) {\color{blue}$(1,2)^0$};

\node  at (-3.9,2.6) {\color{blue}$(1,3)^1$};
\node  at (-2.4,2.6) {\color{blue}$(1,2)^1$};
\node  at (-1.1,2.6) {\color{blue}$(1,2)^1$};
\node  at (1.1,2.6) {\color{blue}$(1,3)^0$};
\node  at (2.4,2.6) {\color{blue}$(1,2)^0$};
\node  at (3.9,2.6) {\color{blue}$(1,2)^0$};

\node  at (4,3.5) {\color{blue}$(1,2)^1$};
\node  at (3.3,3.3) {\color{blue}$(1,2)^1$};
\node  at (2,3.5) {\color{blue}$(1,3)^1$};

\node  at (-2,5) {\color{blue}$(1,4)^0$};
\node  at (-.5,5) {\color{blue}$(1,4)^0$};
\node  at (2,5) {\color{blue}$(1,4)^0$};

\end{tikzpicture}
                \caption{Labeling $\lambda$}
                \label{fig:ellabelingposet}
        \end{subfigure}%
        ~ 
        \begin{subfigure}[b]{0.55\textwidth}
                \centering
                \begin{tikzpicture}[scale=0.6]
 \tikzstyle{every node}=[inner sep=1pt, minimum width=14pt,scale=0.7, font=\footnotesize]
\draw (0,0) node (n120) {\color{blue}$(1,2)^0$};
\draw (0,1) node (n130) {\color{blue}$(1,3)^0$};
\draw (0,2) node (n140) {\color{blue}$(1,4)^0$};
\draw (2,1) node (n121) {\color{blue}$(1,2)^1$};
\draw (2,2) node (n131) {\color{blue}$(1,3)^1$};
\draw (2,3) node (n141) {\color{blue}$(1,4)^1$};

\draw (n141) -- (n140) ;
\draw (n131) -- (n130) ;
\draw (n121) -- (n120) ;
\draw (n140)-- (n130) -- (n120) ;

\draw (2,4) node (n230) {\color{blue}$(2,3)^0$};
\draw (2,5) node (n240) {\color{blue}$(2,4)^0$};
\draw (4,5) node (n231) {\color{blue}$(2,3)^1$};
\draw (4,6) node (n241) {\color{blue}$(2,4)^1$};
\draw (4,7) node (n340) {\color{blue}$(3,4)^0$};
\draw (6,8) node (n341) {\color{blue}$(3,4)^1$};
\draw (n241) -- (n240) ;
\draw (n231) -- (n230) ;

\draw (n240)-- (n230)  -- (n141) -- (n131) -- (n121);

\draw (n341) -- (n340) ;

\draw (n340) -- (n241) -- (n231);

\end{tikzpicture}
                \caption{$\Lambda _3$}
                \label{fig:lambdaposet}
        \end{subfigure}
\caption{EL-labeling of the poset $\widehat{\Pi_3^w}$}
\end{figure}

\begin{proof}
We need to show that in every closed interval of $\widehat{\Pi_{n}^w}$  there is a unique increasing
chain 
(from bottom to top), which  is also lexicographically first.  Let $\rho$ denote the rank function  of
$\widehat{\Pi_{n}^w}$.  
We divide the proof into $4$ cases:
\begin{enumerate}
 \item \emph{Intervals of the form $[\hat{0},[n]^r]$}. Since, from bottom to top, the last
 step of merging two blocks includes a block that contains 1,  all of the maximal
 chains have a final label of the form $(1,m)^u$, and so 
 any  increasing maximal chain has to have label word
 $(1,2)^{u_1}(1,3)^{u_2}\cdots(1,n)^{u_{n-1}}$ with $u_i=0$ for $i\le n-1-r$ and $u_i=1$ for $i>
n-1-r$.
 This label word is lexicographically first and the only chain with this label word is 
 (listing only the nonsingleton blocks)
 \[
  \hat{0}\lessdot 12^{u_1}\lessdot 123^{u_1+u_2} \lessdot \cdots \lessdot 123\cdots n ^r.
 \] 

 \item \emph{Intervals of the form $[\hat{0},\alpha]$ for $\rho(\alpha)<n-1$}.  Let
$A_1^{u_1},\dots, A_k^{u_k}$ be the weighted blocks of $\alpha$, where $\min A _i < \min A_j$ if $i
<j$.  For each $i$, let $m_i = \min A_i$.
 By the previous case, in each of the posets $[\hat{0},A_i^{u_i}]$ there is only one increasing
manner 
 of merging the blocks, and the labels of the increasing chain belong to the label set
$\Gamma_{m_i}$.  The increasing chain is also lexicographically first.  Consider the maximal chain
of $[\hat{0},\alpha]$ obtained
 by first  merging  the blocks of the increasing chain in $[\hat{0},A_1^{u_1}]$, then the ones in
 the increasing chain in $[\hat{0},A_2^{u_2}]$, and so on.
The constructed chain is still increasing
 since the labels in $\Gamma_{m_i}$ are less than the labels in $\Gamma_{m_{i+1}}$ for each
$i=1,\dots,k-1$.  It is not difficult to see that this is  the only increasing chain of
$[\hat{0},\alpha]$ and that it is lexicographically first. 

\vspace{.1in} \item \emph{The interval $[\hat 0,\hat 1]$.}  An increasing chain $c$ of this interval
must be of the form $\c^\prime \cup \{\hat 1\}$, where  $c^\prime$ is the unique increasing chain of
  some interval $[\hat{0},[n]^r]$.  By Case 1, the label word of $c^\prime$ ends in $(1,n)^u$
for some $u$. 
 For $c$ to be increasing, $u$ must be $0$. But $u=0$ only in the interval
$[\hat{0},[n]^0]$.  Hence the unique increasing chain of $[\hat{0},[n]^0]$ concatenated with $\hat
1$ is the only increasing chain of $[\hat 0,\hat 1]$.  It is clearly lexicographically first. 
 
\vspace{.1in} \item \emph{Intervals of the form $[\alpha,\beta]$ for $\alpha \ne \hat 0$}. We extend
the definition of $\Pi_n^w$ to  $\Pi_S^w$, where $S$ is an arbitrary  finite set of positive integers, by 
considering
partitions of $S$ rather than $[n]$.  We also extend the definition of the labeling $\lambda$ to
$\widehat{\Pi_S^w}$. Now we can identify the interval $[\alpha, \hat 1]$ with $\widehat{\Pi_{S}^w}$,
where $S$ is the set of minimum elements of the blocks of $\alpha$, 
 by replacing each block $A$ of $\alpha$ by  its minimum element and subtracting the weight of $A$ 
from the weight of the block containing $A$ in each
weighted partition of $[\alpha, \hat 1]$. This isomorphism preserves the labeling and so
 the three previous cases show that there is a unique  increasing chain  in $[\alpha,\beta]$ that is
also lexicographically
 first.	
 \end{enumerate}
\end{proof}

\subsection{Topological consequences}

  When we attribute a topological property to a poset $P$, we are really attributing the property to
the order 
  complex $\Delta(P)$, which is defined to be the simplicial complex whose faces are the chains of
$P$.  For 
  instance, by $\tilde H_r(P;{\bf k})$ and $\tilde H^r(P;{\bf k})$ we mean, respectively, reduced
simplicial homology 
  and cohomology of the order complex $\Delta(P)$, taken over ${\bf k}$, where ${\bf k}$ is an
arbitrary field or the ring of integers $\ZZ$.  (We will usually omit the ${\bf k}$  and  write just
$\tilde H_r(P)$ and $\tilde H^r(P)$.)
  For a brief review of the homology and cohomology of posets, see the appendix
(Section~\ref{section:homologyposets}).

The fundamental link between lexicographic shellability and topology is
given  is the following result.  Recall that the proper part $\overline P$ of a bounded poset $P$ is 
defined by  
$\overline P:=P \setminus \{\hat 0, \hat 1\}$.  Hence, if $c$ is a maximal chain of $P$ then $\bar 
c$ denotes the maximal chain of $\overline P$ given by  $c \setminus \{\hat 0, \hat 1\}$.

\begin{theorem}[Bj\"orner and Wachs \cite{BjornerWachs1996}] \label{elth}
Let $\lambda$ be an EL-labeling  of a bounded   poset $P$. 
Then for all $x<y$ in $P$, 
\begin{enumerate}
\item the open interval $(x,y)$ is homotopy equivalent to a wedge of  spheres, where for each $r \in
\NN$ the number of spheres  of dimension $r$ is the number of ascent-free maximal chains of the
closed interval $[x,y]$ of length $r+2$. 
\item the set
$$\{\bar c : c \mbox{ is an ascent-free maximal chain of $[x,y]$ of length } r+2 \}$$
forms a basis for cohomology $\tilde H^{r}((x,y))$, for all $r$.
\end{enumerate}
 \end{theorem}

Since the M\"obius invariant of a  bounded poset $P$ equals the reduced Euler characteristic of the
order complex of $\overline{P}$, the Euler-Poincar\'e formula implies the following corollary.
\begin{corollary}Let $P$ be a pure  EL-shellable poset of length $n$.  Then
\begin{enumerate}
\item $\overline{P}$ has the homotopy type of a wedge of spheres all of dimension $n-2$, where the number
of spheres is  $|\mu_P(\hat 0, \hat 1)|$. 
\item $P$ is Cohen-Macaulay, which means that $\tilde H_i((x,y)) = 0$ for all $x <y$ in $P$ and $i <
l([x,y]) -2$. \end{enumerate}
\end{corollary}

In \cite{DotsenkoKhoroshkin2010} Dotsenko and Khoroshkin use operad theory to prove that all
intervals of $\Pi_n^w$ are Cohen-Macaulay.  The following extension of their result is a consequence
of Theorem~\ref{theorem:ellabelingposet}.

\begin{corollary}\label{corollary:cohenmacaulay}
 The poset $\widehat{\Pi_n^w}$ is Cohen-Macaulay.
\end{corollary}

Now  by Theorem~\ref{theorem:ellabelingposet}, Proposition~\ref{proposition:mobiushat} and
Corollary~\ref{corollary:mutrees}  we have,

\begin{theorem} \label{theorem:homotopy} For all $n \ge 1$, 
\begin{enumerate}
\item $\Pi_n^w \setminus \{\hat 0\}$ has the homotopy type of a wedge of $(n-1)^{n-1}$ spheres of
dimension $n-2$,
\item $(\hat 0, [n]^i)$ has the homotopy type of a wedge of $|\T_{n,i}|$ spheres of dimension $n-3$
for all $i \in \{0,1,\dots,n-1\}$.
\end{enumerate}
\end{theorem}

It follows from Theorem~\ref{theorem:homotopy} (and Proposition~\ref{prop:free} in the appendix)
that  top cohomology $\tilde H^{n-2}(\Pi_n^w\setminus \hat{0})$ and $\tilde H^{n-3}((\hat 0,
[n]^i))$ are  free ${\bf k}$-modules, which are isomorphic to the corresponding top homology
modules, that is 
$$\tilde H^{n-2}(\Pi_n^w\setminus \hat{0}) \simeq \tilde H_{n-2}(\Pi_n^w\setminus \hat{0})$$ and
$$\tilde H^{n-3}((\hat 0, [n]^i)) \simeq \tilde H_{n-3}((\hat 0, [n]^i))$$ for  $0 \le i \le n-1$.
Moreover, we have the following result.
\begin{corollary}\label{proposition:dimensionhat} For  $0\le i \le n-1$, 
\begin{eqnarray*} \rank \tilde H_{n-2}(\Pi_n^w\setminus \hat{0})&=&(n-1)^{n-1}
\\ \rank \tilde H_{n-3}((\hat{0},[n]^{i})) &=& |\T_{n,i}| 
\\  \rank  \bigoplus_{i=0}^{n-1} \tilde H_{n-3}((\hat 0, [n]^i)) &= & n^{n-1}.
\end{eqnarray*}
\end{corollary}

\begin{remark} \label{stro}{\rm In a prior attempt to  establish Cohen-Macaulayness of each maximal
interval $[\hat 0, [n]^i]$ of $\Pi_n^w$,
it is argued in \cite{DotsenkoKhoroshkin2007}  that  the intervals are  totally semimodular and
hence CL-shellable\footnote{CL-shellability is a property more general the EL-shellability, which
also implies Cohen-Macaulaynes; see \cite{BjornerWachs1983}, \cite{BjornerWachs1996} or
\cite{Wachs2007}}.  In  \cite{Strohmayer2008} it is noted that this is
not the case and a proposed recursive atom ordering\footnote{See \cite{BjornerWachs1983},
\cite{BjornerWachs1996} or \cite{Wachs2007} for the definition of  recursive atom ordering. The
property of admitting a recursive atom ordering is equivalent to that of being CL-shellable.} of
each maximal interval $[\hat 0, [n]^i]$ is given in order to establish CL-shellability.    In
\cite[Proof of Proposition 3.9]{Strohmayer2008} it is claimed that given any linear ordering $
\{i_1,j_1\} , \{i_2,j_2\},\cdots, \{i_m,j_m\}
$
 of the atoms of $\Pi_n$ (the singleton blocks have been omitted), the linear ordering 
\begin{equation}\label{eq:stroh}
\{i_1,j_1\}^0, \{i_1,j_1\}^1, \{i_2,j_2\}^0,\{i_2,j_2\}^1\,
\cdots \{i_m,j_m\}^0,\{i_m,j_m\}^1
\end{equation}
 satisfies the criteria for being a recursive atom ordering of $[\hat 0, [n]^i]$, where $1\le i \le
n-2$.
We note here that one of the requisite conditions in the definition of recursive atom ordering fails
to hold when 
 $n=4$ and $i=2$. Indeed, assume (without loss of generality) that    the first two 
atoms in the atom ordering of $[\hat 0, [4]^2]$ given in (\ref{eq:stroh})  are $\{1,2\}^0$ and
$\{1,2\}^1$.  Then the atoms of  the interval $ 
[\{1,2\}^1, [4]^2]$ that cover $\{1,2\}^0$ are $\{1,2,3\}^1$ and $\{1,2,4\}^1$.  So by the
definition of recursive atom ordering one of these covers must come first in any recursive atom
ordering of $ 
[\{1,2\}^1, [4]^2]$ and the other must come second.  But this contradicts the form of 
(\ref{eq:stroh}) applied to the interval $ 
[\{1,2\}^1, [4]^2]$ which requires the atom $\{1,2,3\}^2$ to  immediately follow the atom
$\{1,2,3\}^1$ and the atom $\{1,2,4\}^2$ to immediately follow the atom $\{1,2,4\}^1$.  The proof of
Proposition 3.9 of \cite{{Strohmayer2008}}  breaks down in the second from last paragraph.}
\end{remark}

\section{Connection with the doubly bracketed free Lie algebra}\label{section:conlie}
\subsection{The doubly bracketed free Lie algebra} \label{subsection:genlie2} 
In this section ${\bf k}$ denotes an arbitrary field. 
Recall that a \emph{Lie bracket} on a  vector
space $V$ is a bilinear binary product  
$[\cdot,\cdot]:V\times V \rightarrow V$ such that for all $x,y,z \in V$, 
\begin{align}
 [x,y] = - [y,x]\quad\quad& \text{(Antisymmetry)}\label{rln:lb1}\\
[x,[y,z]]+[z,[x,y]]+[y,[z,x]]=0 \quad\quad&\text{(Jacobi Identity).}\label{rln:lb2}
\end{align}
The \emph{free Lie algebra on $[n]$} (over the field ${\bf k}$) is the  ${\bf k}$-vector space
generated by the elements
of $[n]$ and all the possible bracketings involving these elements subject only to the relations
(\ref{rln:lb1}) and (\ref{rln:lb2}).
Let $\lie(n)$ denote the {\it multilinear} component of the free Lie algebra on $[n]$, i.e., 
the subspace generated by bracketings that contain each
element of $[n]$ exactly once.   For example $[[2,3],1]$ is an element of $\lie(3)$, while
$[[2,3],2]$ is not.

Now let $V$ be a vector space equipped with two Lie brackets
${\color{blue}[}\cdot,\cdot{\color{blue}]}$ 
and ${\color{red}\langle} \cdot, \cdot {\color{red}\rangle}$.
The brackets are said to be {\it compatible} if any linear combination of them is  a Lie bracket.  
As pointed out in \cite{DotsenkoKhoroshkin2007,Liu2010}, compatibility is equivalent to the {\em
mixed Jacobi}  condition: for all $x,y,z \in V$,

\begin{align}
{\color{blue}[}x,{\color{red}\langle} y, z
{\color{red}\rangle}{\color{blue}]}+{\color{blue}[}z,{\color{red}\langle} x,y
{\color{red}\rangle}{\color{blue}]}+{\color{blue}[}y,{\color{red}\langle} z,
x{\color{red}\rangle}{\color{blue}]}+{\color{red}\langle} x,{\color{blue}[}y,z{\color{blue}]}
{\color{red}\rangle} + 
{\color{red}\langle} z,{\color{blue}[}x,y{\color{blue}]} {\color{red}\rangle}+{\color{red}\langle}
y,{\color{blue}[}z,x{\color{blue}]}{\color{red}\rangle}=0.\quad 
\label{rln:lb3}
\end{align}
Let $\lie_2(n)$ denote the multilinear component of the free Lie algebra on $[n]$ with two
compatible brackets 
${\color{blue}[}\cdot,\cdot{\color{blue}]}$ and ${\color{red}\langle} \cdot, \cdot
{\color{red}\rangle}$, 
that is, the multilinear component of the  ${\bf k}$-vector space generated by (mixed) bracketings
of elements of $[n]$  
subject only to the five  relations given by (\ref{rln:lb1}) and (\ref{rln:lb2}), for each bracket,
and 
(\ref{rln:lb3}). We will call the bracketed words  that generate  $\lie_2(n)$ {\it bracketed
permutations}.

 It will be convenient to refer to the bracket ${\color{blue}[}\cdot,\cdot{\color{blue}]}$ as the
{\it blue}
 bracket and the bracket ${\color{red}\langle} \cdot, \cdot {\color{red}\rangle}$ as the {\it red}
bracket.
For each $i$, let $\lie_2(n,i)$ be the subspace of $\lie_2(n)$ generated by bracketed permutations
with exactly
$i$ red brackets and $n-1-i$ blue brackets.  

A permutation  $\tau \in \sym_n$ acts on the bracketed permutations by replacing each letter $i$ by
$\tau(i)$. 
For example $(1,2)\,\,{\color{red}\langle}{\color{blue}[}{\color{red}\langle} 3,5
{\color{red}\rangle}, {\color{blue}[}2,4{\color{blue}]}
{\color{blue}]},1{\color{red}\rangle}= {\color{red}\langle}{\color{blue}[}{\color{red}\langle} 3,5
{\color{red}\rangle}, {\color{blue}[}1,4{\color{blue}]}
{\color{blue}]},2{\color{red}\rangle}$.  Since this action  respects the five relations, it induces
a 
representation of $\sym_n$ on   $\lie_2(n)$.    Since this action also preserves the number of red
and blue brackets, 
we have the following decomposition into $\sym_n$-submodules: $\lie_2(n)=\oplus_{i=0}^{n-1}
\lie_2(n,i)$.  
Note that by replacing red brackets with blue brackets and vice verce, we get the  $\sym_n$-module
isomorphism, $$\lie_2(n,i) \simeq_{\sym_n}\lie_2(n,n-1-i)$$
for all $i$. Also note that  $$\lie_2(n,0)\simeq_{\sym_n} \lie_2(n,n-1) \simeq_{\sym_n} \lie(n).$$

A {\it bicolored binary tree} is a complete binary tree (i.e., every internal node has a left and a
right child) 
for which each internal node has been colored red or blue. 
For a bicolored binary tree $T$ with $n$ leaves and $\sigma\in \sym_n$, define the 
\emph{labeled bicolored binary tree} $(T,\sigma)$ 
to be the tree $T$ whose $j$th leaf from left to 
right has been labeled $\sigma(j)$. We denote by $\BT_n$  the set of  labeled bicolored binary
trees 
with $n$ leaves and by $\BT_{n,i}$ the set of  labeled bicolored binary trees with $n$ nodes and
$i$ 
red internal nodes.  

It will also be convenient to consider 
labeled bicolored trees whose label 
set is more general than $[n]$.  For a finite set $A$, let $\BT_A$ be the set of bicolored binary
trees whose leaves are labeled by a permutation of $A$ and $\BT_{A,i}$ be the subset of $\BT_A$
consisting of trees with $i$ red internal nodes.    If $(S,\alpha) \in  \BT_{A}$ and $(T,\beta) \in 
\BT_{B}$, where $A$ and $B$ are disjoint finite sets, and $\col \in \{\text{red, blue}\}$ then
$(S,\alpha) \substack{\col \\ \wedge} (S,\beta)$ denotes the tree in $\BT_{A\cup B}$ whose left
subtree is $(S,\alpha)$, right subtree is $(T,\beta)$, and root color is $\col $.

We can represent the bracketed permutations that generate $\lie_2(n)$ with labeled bicolored 
binary trees.  More precisely, \ let
$(T_1,\sigma_1)$ and $(T_2,\sigma_2)$ be the left and right labeled subtrees of the root $r$ of
$(T,\sigma)$.  Then define recursively

\begin{equation}\label{definition:treebracket}
[T,\sigma]= \left\{ 
  \begin{array}{l l}
    {\color{blue}[}[T_1,\sigma_1],[T_2,\sigma_2]{\color{blue}]} & \quad \text{if $r$ is blue and
$n>1$}\\
    {\color{red}\langle }[T_1,\sigma_1],[T_2,\sigma_2]{\color{red}\rangle} & \quad \text{if $r$ is
red and $n>1$}\\
    \sigma & \quad \text{if $n=1$.}\\
  \end{array} \right.
\end{equation}
Clearly  $(T,\sigma) \in \BT_{n,i}$ if and only if $[T,\sigma] $ is a bracketed permutation of
$\lie_2(n,i)$. 
See Figure \ref{fig:bicolorbinarytree}.

\begin{figure}
        \centering
        \begin{subfigure}[b]{0.6\textwidth}
                \centering
               \begin{tikzpicture}[thick,scale=0.8]
    \draw [circle,color=red] (1.7,4)  node (red){ Red};
    \draw [color=blue] (1.8,3.5)  node (blue){ Blue};
\tikzstyle{every node}=[fill, draw,inner sep=4pt, minimum width=1pt,scale=0.8]
    \draw [circle,color=red] (1,4)  node (r){};
    \draw [color=blue] (1,3.5)  node (b){};

\tikzstyle{every node}=[fill, draw,inner sep=4pt, minimum width=1pt,scale=0.8]
    \draw [circle,color=red] (1,4)  node (r){};
    \draw [circle,color=red] (6,4)  node (i1){};
    \draw [circle,color=red] (8.5,3)  node (i2){};
    \draw [circle,color=red] (7.5,2)  node (i3){};
    \draw [color=blue] (6.5,1)  node (i4){};
    \draw [color=blue] (4,3)  node (i5){};
    \draw [color=blue] (5,2)  node (i6){};
    \draw [circle,color=red] (3,2)  node (i7){};
    \draw [color=blue] (2,1)  node (i8){};
\tikzstyle{every node}=[inner sep=1pt, minimum width=14pt,scale=0.8]

    \draw (4.3,1)  node (l1){1};
    \draw (5.5,0)  node (l2){2};
    \draw (1,0)  node (l3){3};
    \draw (3,0)  node (l4){4};
    \draw (6,1)  node (l5){5};
    \draw (3.7,1)  node (l6){6};
    \draw (7.5,0)  node (l7){7};
    \draw (9.5,2)  node (l8){8};
    \draw (8.5,1)  node (l9){9};

    \draw (i1) --  (i2) ;
    \draw (i1) --  (i5) ;
    \draw (i2) --  (i3) ;
    \draw (i2) --  (l8) ;
    \draw (i3) --  (i4) ;
    \draw (i3) --  (l9) ;
    \draw (i4) --  (l7) ;
    \draw (i4) --  (l2) ;
    \draw (i5) --  (i6) ;
    
    \draw (i5) --  (i7) ;
    \draw (i6) --  (l5) ;
    \draw (i6) --  (l1) ;
    \draw (i7) --  (l6) ;
    \draw (i7) --  (i8) ;
    \draw (i8) --  (l3) ;
    \draw (i8) --  (l4) ;
\end{tikzpicture}
        \end{subfigure}%
        ~ 
        \begin{subfigure}[b]{0.55\textwidth}
                \centering
                 \begin{tikzpicture}[thick,scale=0.8,baseline=50]

\tikzstyle{every node}=[inner sep=1pt, minimum width=14pt,scale=1]

    \draw (0,4)  node
(v){\hspace{-.6in}${\color{red}\langle}{\color{blue}[}{\color{red}\langle}{\color{blue}[}3,4{\color{blue}]},6{
\color{red}\rangle}, 
{\color{blue}[}1,5{\color{blue}] ]},{\color{red}\langle
\langle}{\color{blue}[}2,7{\color{blue}]},9{\color{red}\rangle},8{\color{red}\rangle \rangle}$};

\end{tikzpicture}
        \end{subfigure}
  \caption{Example of a tree $(T,346152798) \in \BT_{9,4}$\,\, and \,\, $[T,346152798]\in
\lie_2(9,4)$}
  \label{fig:bicolorbinarytree}

  \end{figure}

\subsection{A generating set for $\tilde H^{n-3}((\hat 0, [n]^i))$}\label{section:genhom}

In this section the ring of coefficients ${\bf k}$  for cohomology is either $\ZZ$ or an arbitrary
field.

The top dimensional cohomology of a pure poset $P$, say of length $\ell$, has a particularly simple
description (see Appendix \ref{section:homologyposets}).  Let $\MM(P)$ denote the set of maximal
chains of $P$ and let $\MM^\prime(P)$ denote the set of chains of length $\ell-1$.  We view the
coboundary map $\delta$ as a map from the chain space of  $P$ to itself, which takes chains of
length $d$ to chains of length $d+1$ for all $d$.  Since the image of $\delta$ on the top chain
space (i.e. the space spanned by $\MM(P)$) is $0$, the kernel  is the entire top chain space.   
Hence top cohomology is the quotient of the space spanned by $\MM(P)$ by the image of the space
spanned by $\MM^\prime(P)$.  The image of $\MM^\prime(P)$ is what we call the coboundary relations. 
We thus have the following presentation of the top cohomology  $$\tilde H^{\ell}(P) = \langle \MM(P)
| \mbox{ coboundary relations} \rangle.$$

Recall that the {\it postorder listing} of the internal nodes of a binary tree $T$ is defined
recursively as follows:  first list the internal nodes of the left subtree  in postorder, then list
the internal nodes of the right subtree in postorder, and finally list the root.   
The postorder listing of the internal nodes of the binary tree of Figure~\ref{fig:bicolorbinarytree}
is illustrated in Figure \ref{fig:postorder}.

Given $k$ blocks $A_1^{w_1},A_2^{w_2}, \dots, A_k^{w_k}$ in a weighted partition $\alpha$ and 
$u \in \{0,\dots,k-1\}$, by $u$-{\em merge} these blocks we mean remove them from $\alpha$ and
replace them by the block $(\bigcup A_i )^{\sum w_i +u}$.   Given $\col \in \{\mbox{blue, red}\}$,
let 
$$u(\col) = \begin{cases} 0 &\mbox{ if $\col = $ blue} \\
 1 &\mbox{ if $\col = $ red}.  \end{cases}  $$    For  $(T,\sigma) \in \BT_{A,i}$, let
$\pi(T,\sigma)=A^i$.   

\begin{definition}\label{definition:treechain} 
For $(T,\sigma) \in \BT_n$ and $k \in [n-1]$, let $T_{k}=L_{k}\substack{\col_k \\ \wedge} R_{k}$ be
the subtree of $(T,\sigma)$ rooted at 
the $k$th node listed in postorder.  The chain $\c(T,\sigma)\in \MM(\Pi_n^{w})$   is the one
whose rank $k$ weighted partition is obtained from the rank $k-1$ weighted partition by
$u(\col_k)$-merging the blocks $\pi(L_{k})$ and
$\pi(R_{k})$. 
See Figure~\ref{fig:binarytreepostorderandchain}.
\end{definition}

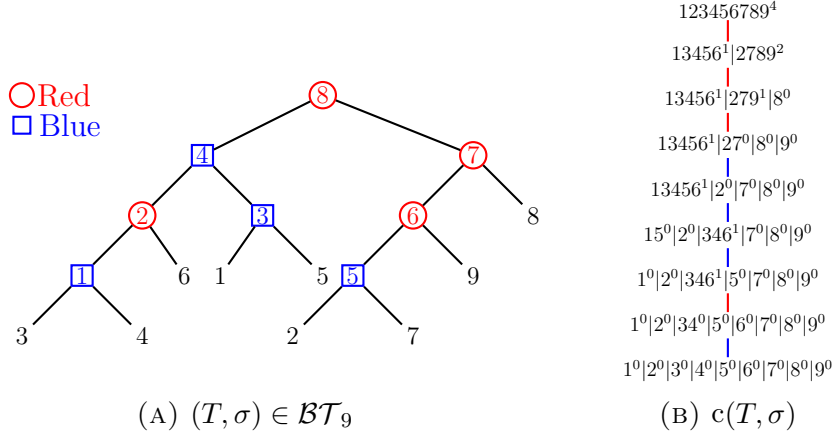
\begin{figure}
        \centering
        \begin{subfigure}[b]{0.5\textwidth}
                \centering
                   \begin{tikzpicture}[thick,scale=0.8]
\draw [circle,color=red] (1.7,4)  node (red){ Red};
    \draw [color=blue] (1.8,3.5)  node (blue){ Blue};
\tikzstyle{every node}=[draw,inner sep=4pt, minimum width=1pt,scale=0.8]
    \draw [circle,color=red] (1,4)  node (r){};
    \draw [color=blue] (1,3.5)  node (b){};

\tikzstyle{every node}=[draw,inner sep=1pt, minimum width=10pt,scale=0.8]

    \draw [circle,color=red] (6,4)  node (i1){8};
    \draw [circle,color=red] (8.5,3)  node (i2){7};
    \draw [circle,color=red] (7.5,2)  node (i3){6};
    \draw [color=blue] (6.5,1)  node (i4){5};
    \draw [color=blue] (4,3)  node (i5){4};
    \draw [color=blue] (5,2)  node (i6){3};
    \draw [circle,color=red] (3,2)  node (i7){2};
    \draw [color=blue] (2,1)  node (i8){1};
\tikzstyle{every node}=[inner sep=1pt, minimum width=14pt,scale=0.8]

    \draw (4.3,1)  node (l1){1};
    \draw (5.5,0)  node (l2){2};
    \draw (1,0)  node (l3){3};
    \draw (3,0)  node (l4){4};
    \draw (6,1)  node (l5){5};
    \draw (3.7,1)  node (l6){6};
    \draw (7.5,0)  node (l7){7};
    \draw (9.5,2)  node (l8){8};
    \draw (8.5,1)  node (l9){9};

    \draw (i1) --  (i2) ;
    \draw (i1) --  (i5) ;
    \draw (i2) --  (i3) ;
    \draw (i2) --  (l8) ;
    \draw (i3) --  (i4) ;
    \draw (i3) --  (l9) ;
    \draw (i4) --  (l7) ;
    \draw (i4) --  (l2) ;
    \draw (i5) --  (i6) ;
    
    \draw (i5) --  (i7) ;
    \draw (i6) --  (l5) ;
    \draw (i6) --  (l1) ;
    \draw (i7) --  (l6) ;
    \draw (i7) --  (i8) ;
    \draw (i8) --  (l3) ;
    \draw (i8) --  (l4) ;
\end{tikzpicture}

                \caption{$(T,\sigma) \in \BT_9$}
                \label{fig:postorder}
        \end{subfigure}%
        ~ 
        \begin{subfigure}[b]{0.5\textwidth}
                \centering
                 \begin{tikzpicture}[thick,scale=0.6]

 \tikzstyle{every node}=[inner sep=1pt, minimum width=10pt,scale=0.65]

    \draw (0,0)  node (c0){$1^0|2^0|3^0|4^0|5^0|6^0|7^0|8^0|9^0$};

    \draw (0,1)  node (c1){$1^0|2^0|34^0|5^0|6^0|7^0|8^0|9^0$};
    \draw [color=blue] (c0) -- (c1);

    \draw (0,2)  node (c2){$1^0|2^0|346^1|5^0|7^0|8^0|9^0$};
    \draw [color=red](c1) -- (c2);

    \draw (0,3)  node (c3){$15^0|2^0|346^1|7^0|8^0|9^0$};
    \draw [color=blue] (c2) -- (c3) ;

    \draw (0,4)  node (c4){$13456^1|2^0|7^0|8^0|9^0$};
    \draw [color=blue] (c3) -- (c4);

    \draw (0,5)  node (c5){$13456^1|27^0|8^0|9^0$};
    \draw [color=blue](c4) -- (c5);

    \draw (0,6)  node (c6){$13456^1|279^1|8^0$};
    \draw [color=red](c5) -- (c6);

    \draw (0,7)  node (c7){$13456^1|2789^2$};
    \draw [color=red] (c6) -- (c7);

    \draw (0,8)  node (c8){$123456789^4$};
    \draw [color=red](c7) -- (c8);

\end{tikzpicture}
 
                \caption{$\c(T,\sigma)$}
                 \label{fig:binarytreechain}
        \end{subfigure}
 \caption{Example of postorder (internal nodes) of the binary tree $T$ of Figure
\ref{fig:bicolorbinarytree} and the chain $\c(T,\sigma)$}
    \label{fig:binarytreepostorderandchain}

  \end{figure}

Not all maximal chains in $\MM(\Pi_n^w)$ can be described as $\c(T,\sigma)$. For some maximal chains
postordering of the internal
nodes is not enough to describe the process of merging the blocks.  We need a more flexible
construction in terms of linear
extensions (cf. \cite{Wachs1998}).  Let $v_1,\dots,v_n$ be the postorder listing of the internal
nodes of $T$.  A listing $v_{\tau(1)},v_{\tau(2)},...,v_{\tau(n-1)}$ of the internal nodes  such
that each node precedes its parent is said to be a {\em linear extension} of $T$.  We will say that
the permutation $\tau$ induces the linear extension.   In particular, the identity permutation
$\varepsilon$ induces postorder which is a linear extension. Denote by $\mathcal{E}(T)$  the set of 
permutations   that induce linear extensions of the internal nodes of  $T$.
So we extend the construction of $\c(T,\sigma)$ by letting $\c(T,\sigma,\tau)$  be the chain in
$\MM(\Pi_n^{w})$ 
whose rank $k$ weighted partition is obtained from the rank $k-1$ weighted partition by
$u(\col_{\tau(k)})$-merging the blocks
$\pi(L_{\tau(k)})$ and
$\pi(R_{\tau(k)})$, where $L_{i}\substack{\col_{i} \\ \wedge}
R_{i}$ 
is the subtree rooted at $v_{i}$. 
In particular,
$\c(T,\sigma)=\c(T,\sigma,\varepsilon)$. From each maximal chain we can easily construct a binary
tree and a linear extension that encodes the merging
instructions along the chain. So it follows that any maximal chain can be obtained in this form.

\begin{lemma}[{\cite[Lemma 5.1]{Wachs1998}}]\label{lemma:51}
 Let $T$ be a binary tree. Then
\begin{enumerate}
 \item $\varepsilon \in \mathcal{E}(T)$
 \item If $\tau \in \mathcal{E}(T)$ and $\tau(i)>\tau(i+1)$ then $\tau (i,i+1) \in \mathcal{E}(T)$,
\end{enumerate}
where $\tau (i,i+1)$ denotes the product of $\tau$ and the transposition $(i,i+1)$ in the symmetric group.
\end{lemma}
\begin{proof}
 Postorder $\varepsilon$ is a linear extension since in postorder we list children before 
parents. Now, 
$\tau(i)>\tau(i+1)$
 means that $v_{\tau(i+1)}$ is listed in postorder before  $v_{\tau(i)}$, and so $v_{\tau(i+1)}$ 
cannot be an
 ancestor of $v_{\tau(i)}$. This implies that $\tau (i,i+1)$ is also a linear extension. 
\end{proof}

The number of inversions of a permutation  $\tau \in \sym_n$ is defined by $\inv(\tau) := |\{ (i,j) :
1 \le i < j \le n, \,\, \tau(i) > \tau(j) \}|$ and the sign of $\tau$ is defined by
$\sgn(\tau) := (-1)^{\inv(\tau)}$.   
For $T \in \BT_{n,i}$, $\sigma \in \sym_n$, and $\tau \in 
\mathcal{E}(T)$, 
write $\bar c(T,\sigma,\tau)$ for  $\overline{c(T,\sigma,\tau)}:= c(T,\sigma,\tau) \setminus \{\hat 
0, [n]^i\}$ and  $\bar c(T,\sigma)$ for  $\overline{c(T,\sigma)}:= c(T,\sigma) \setminus \{\hat 0, 
[n]^i\}$.

\begin{lemma}[cf. {\cite[Lemma 5.2]{Wachs1998}}]\label{lemma:52}
 Let $T \in \BT_{n,i}$, $\sigma \in \sym_n$, $\tau \in \mathcal{E}(T)$. Then in 
$ \tilde H^{n-3}((\hat{0},[n]^i))$
\[
\bar{\c}(T,\sigma,\tau)=\sgn(\tau)\bar{\c}(T,\sigma).
\]
\end{lemma}
\begin{proof}
We proceed by induction on $\inv(\tau)$. If $\inv(\tau)=0$ then $\tau=\varepsilon$ and the result 
is trivial. 
If $\inv(\tau)\ge 1$,
then there is some descent $\tau(i)>\tau(i+1)$ and by Lemma \ref{lemma:51}, $\tau (i,i+1) \in 
\mathcal{E}(T)$. Since
$\inv(\tau (i,i+1))=\inv(\tau)-1$, by induction we have,
\[
\bar{\c}(T,\sigma,\tau (i,i+1))=\sgn(\tau(i,i+1))\bar{\c}(T,\sigma)=-\sgn(\tau)\bar{\c}(T,\sigma).
\]
We have to show then that
\[
\bar{\c}(T,\sigma,\tau)=-\bar{\c}(T,\sigma,\tau (i,i+1)).
\]

By the proof of Lemma \ref{lemma:51} we know that the internal nodes $v_{\tau(i)}$ and $v_{\tau(i+1)}$ 
are unrelated in $T$
and so $\pi(L_{\tau(i)})$, $\pi(R_{\tau(i)})$, $\pi(L_{\tau(i+1)})$ and 
$\pi(R_{\tau(i+1)})$ are pairwise disjoint sets
which are all blocks of the rank $i-1$ partition in both $\bar{\c}(T,\sigma,\tau)$ and 
$\bar{\c}(T,\sigma,\tau(i,i+1))$. 
The blocks $\pi(L_{\tau(i)} \wedge R_{\tau(i)})$ and $\pi(L_{\tau(i+1)} \wedge 
R_{\tau(i+1)})$ are blocks 
of the rank $i+1$ partition in both $\bar{\c}(T,\sigma,\tau)$ and $\bar{\c}(T,\sigma,\tau(i,i+1))$. Hence 
the 
maximal chains 
$\bar{\c}(T,\sigma,\tau)$ and $\bar{\c}(T,\sigma,\tau(i,i+1))$ only differ at rank $i$. So if we 
denote by $c$ either
of these maximal chains with the rank $i$ partition removed we get, using equation 
(\ref{equation:cohomologyboundary}), a cohomology relation given by
\[
 \delta(c)=(-1)^{i}(\bar{\c}(T,\sigma,\tau) + \bar{\c}(T,\sigma,\tau(i,i+1)))
\]
as desired.
\end{proof}

We conclude that in cohomology any maximal chain $c \in\MM(\Pi_n^w)$ is cohomology equivalent to a
chain of the form $\c(T,\sigma)$,
more precisely, in cohomology $\bar c=\pm \bar\c(T,\sigma)$.

We will make further use of  the elementary cohomology relations that are obtained by setting the 
coboundary
(given in (\ref{equation:cohomologyboundary})) of a codimension 1 
chain in $(\hat{0},[n]^i)$ equal to 0.  There are three types of codimension 1 chains, which 
correspond to the three types of intervals of length 2 (see Figure~\ref{fig:cohomologyrelations}). 
 Indeed, if $\bar c$ is a codimension 1 chain of $(\hat{0},[n]^i)$ then $c =\bar c \cup \{\hat 0,
[n]^i\}$ 
is unrefinable except between one pair of adjacent elements $x < y$, where $[x,y]$ is an interval 
of length 2.  If the open interval $(x,y) = \{z_1,\dots,z_k\}$ then it follows from 
(\ref{equation:cohomologyboundary}) that
$$\delta(\bar c) = \pm (\bar c \cup \{z_1\} + \dots + \bar c \cup \{z_k\} ).$$ By setting 
$\delta(\bar c) = 0$ we obtain the elementary cohomology relation  
$$(\bar c \cup \{z_1\}) + \dots + (\bar c \cup \{z_k\} )= 0.$$

\begin{enumerate}
\item[{\bf Type I}:] Two pairs of distinct blocks  of $x$ are merged to get $y$.   The open interval
$(x,y)$ equals $\{z_1,z_2\}$, where $z_1$ is obtained by $u_1$-merging the first pair of blocks and $z_2$
is obtained by $u_2$-merging the second pair of blocks for some $u_1,u_2 \in \{0,1\}$.  Hence the Type~I elementary cohomology relation
is
$$\bar c \cup \{z_1\} = - ( \bar c \cup \{z_2\}).$$ 
\item[{\bf Type II}:] Three distinct blocks  of $x$ are $2u$-merged to get $y$, where $u \in \{0,1\}$.  The open interval
$(x,y)$ equals $\{z_1,z_2,z_3\}$, where each weighted partition  $z_i$ is obtained from $x$ by
$u$-merging two of the three blocks.    Hence the Type~II elementary cohomology relation is
$$(\bar c \cup \{z_1\}) + ( \bar c \cup \{z_2\}) + ( \bar c \cup \{z_3\} )= 0.$$ 
\item[{\bf Type III}:] Three distinct blocks  of $x$ are $1$-merged to get $y$. The open interval
$(x,y)$ equals $\{z_1,z_2,z_3,z_4,z_5,z_6\}$, where each weighted partition  $z_i$ is obtained from
$x$ by either $0$-merging or $1$-merging two of the three blocks.    Hence the Type~III elementary
cohomology relation is
$$(\bar c \cup \{z_1\} )+  (\bar c \cup \{z_2\} )+  (\bar c \cup \{z_3\} ) +  (\bar c \cup \{z_4\})
+  (\bar c \cup \{z_5\} )+  (\bar c \cup \{z_6\}) = 0.$$ 
 \end{enumerate}

\begin{figure}
\begin{subfigure}[b]{\textwidth}
\centering
\begin{tikzpicture}[line join=bevel,scale=0.8]
\begin{scope}
  \tikzstyle{every node}=[inner sep=0pt, scale=0.65, minimum width=4pt]
  \node (v1-2-3-4) at (0,0)  {$A^{a}| B^{b}| C^{c}| D^{d}$};
  \node (v12-34) at (0,4)  {$AB^ {a+b+u_1}|CD^ {c+d+u_2}$};
  \node (v12-3-4) at (-2,2)  {$AB^{a+b+u_1}| C^{c}| D^{d}$};
  \node (v1-2-34) at (2,2)  {$A^{a}| B^{b}| CD^{c+d+u_2}$};
  \draw [] (v12-34) -- (v12-3-4);
  \draw [] (v12-34) -- (v1-2-34);
  \draw [] (v1-2-34) -- (v1-2-3-4); 
  \draw [] (v12-3-4) -- (v1-2-3-4);
 \end{scope}
\end{tikzpicture}
\caption{Type I}
\label{fig:type1}
\end{subfigure}
\vspace{0.2in}

\begin{subfigure}[b]{\textwidth}
\centering
\begin{tikzpicture}[line join=bevel,scale=0.8]
\begin{scope}
  \tikzstyle{every node}=[inner sep=0pt, scale=0.65, minimum width=4pt]
  \node (v1-2-3) at (0,0)  {$A^{a}| B^{b}| C^{c}$};
  \node (v123) at (0,4)  {$ABC^ {a+b+c+2u}$};
  \node (v12-3) at (-3,2)  {$AB^{a+b+u}| C^{c}$};
  \node (v1-23) at (3,2)  {$A^{a}| BC^{b+c+u} $};
 \node (v13-2) at (0,2)  {$AC^{a+c+u}| B^{b}$};
  \draw [] (v123) -- (v12-3);
  \draw [] (v123) -- (v1-23);
 \draw [] (v123) -- (v13-2);
  \draw [] (v1-23) -- (v1-2-3); 
  \draw [] (v12-3) -- (v1-2-3);
  \draw [] (v13-2) -- (v1-2-3);
 \end{scope}
\end{tikzpicture}
\caption{Type II}
\label{fig:type2}
\end{subfigure}
\vspace{0.2in}

\begin{subfigure}[b]{\textwidth}
\centering
\begin{tikzpicture}[line join=bevel,scale=0.8]
\tikzstyle{every node}=[ inner sep=0pt, scale=0.65, minimum width=4pt]
  \node (v1-2-3) at (0,0)  {$A^{a}| B^{b}| C^{c}$};
  \node (v123) at (0,4)  {$ABC^{a+b+c+1}$};
  \node (v12a-3) at (-7,2)  {$AB^{a+b}| C^{c}$};
  \node (v13a-2) at (-4.5,2)  {$AC^{a+c}| B^{b}$};
  \node (v1-23a) at (-1.5,2)  {$A^{a}| BC^{b+c}$};   
  \node (v12b-3) at (1.5,2)  {$AB^{a+b+1}| C^{c}$};
  \node (v13b-2) at (4.5,2)  {$AC^{a+c+1}| B^{b}$};
  \node (v1-23b) at (7,2)  {$A^{a}|BC^ {b+c+1}$};

  \draw [] (v123)-- (v12a-3);
  \draw [] (v123) -- (v1-23a);
  \draw [] (v123)--(v12b-3);
  \draw [] (v123) --(v13a-2);
  \draw [] (v123) --(v13b-2);
  \draw [] (v123) --(v1-23b);
  \draw [] (v1-23a)-- (v1-2-3); 
  \draw [] (v12a-3) --(v1-2-3);
  \draw [] (v13a-2)--(v1-2-3);
  \draw [] (v1-23b) -- (v1-2-3);
  \draw [] (v12b-3) --(v1-2-3);
  \draw [] (v13b-2) --(v1-2-3);
\end{tikzpicture}
  \caption{Type III}
  \label{fig:type3}
\end{subfigure}

\caption[]{Intervals of length 2}\label{fig:cohomologyrelations}
\end{figure}
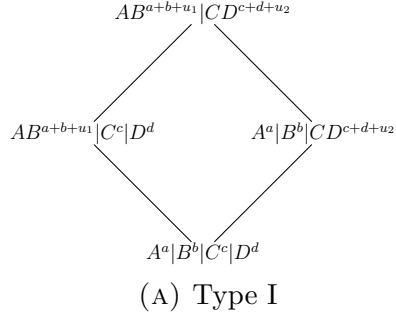
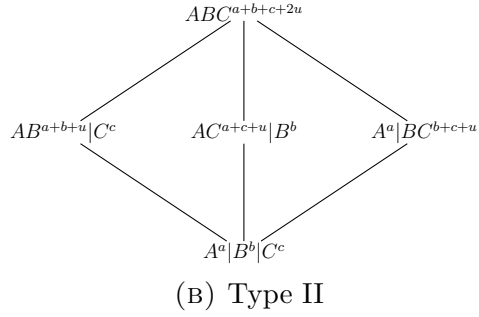
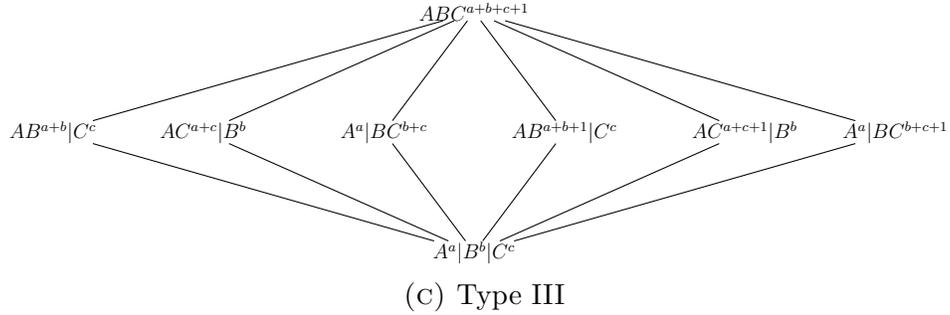

Let $I(\Upsilon)$ denote the set of internal nodes of the labeled bicolored binary tree $\Upsilon$.
Recall that $\Upsilon_1\substack{\col \\ \wedge}\Upsilon_2$ denotes the labeled bicolored binary
tree whose left subtree is $\Upsilon_1$, right subtree is $\Upsilon_2$ and root color  is  col,
where col $\in \{\mbox{blue, red}\}$. If $\Upsilon$ is a labeled bicolored binary tree then  
$\alpha(\Upsilon)\beta$ denotes a labeled bicolored binary tree with $\Upsilon$ as a subtree.  The
following result generalizes  \cite[Theorem 5.3]{Wachs1998}.

\begin{theorem}\label{proposition:binarybasishomology}
 The set $\{ \bar{\c}(T,\sigma) :  (T,\sigma) \in \BT_{n,i}\}$ is a generating set for $\tilde
H^{n-3}((\hat 0, [n]^i))$,
subject only to the relations
  \begin{equation}\bar{\c}(\alpha(\Upsilon_1\substack{\col \\ \wedge \\
\,}\Upsilon_2)\beta)=(-1)^{|I(\Upsilon_1)||I(\Upsilon_2)|}\bar{\c}(\alpha(\Upsilon_2\substack{\col
\\ \wedge \\ \,}\Upsilon_1)\beta),\label{relation:1h}\end{equation}

\begin{eqnarray}\label{relation:3h} \\ \nonumber \bar{\c}(\alpha(\Upsilon_1\substack{\col \\ \wedge
\\ \,}(\Upsilon_2\substack{\col \\ \wedge \\ \,\\ }\Upsilon_3))\beta)
&+& (-1)^{|I(\Upsilon_3)|}\bar{\c}(\alpha((\Upsilon_1\substack{\col \\ \wedge \\
\,}\Upsilon_2)\substack{\col \\ \wedge \\ \,}\Upsilon_3)\beta)\\ 
  &+& (-1)^{|I(\Upsilon_1)||I(\Upsilon_2)|}\bar{\c}(\alpha(\Upsilon_2\substack{\col \\ \wedge \\
\,}(\Upsilon_1\substack{\col \\ \wedge \\ \,}\Upsilon_3))\beta),
\nonumber  \\ &=& 0,\nonumber
\end{eqnarray}
 where $\col \in \{\blue, \red\}$, and
\begin{eqnarray} \label{relation:5h} 
 \bar{\c}(\alpha(\Upsilon_1\substack{\red \\ \wedge}(\Upsilon_2\substack{\blue \\
\wedge}\Upsilon_3))\beta)
 &+& \bar{\c}(\alpha(\Upsilon_1\substack{\blue \\ \wedge}(\Upsilon_2\substack{\red \\
\wedge}\Upsilon_3))\beta)\\ \nonumber
+ \,\, \,\, (-1)^{|I(\Upsilon_3)|} {\Big {(}} \bar{\c}(\alpha((\Upsilon_1\substack{\red \\
\wedge}\Upsilon_2)\substack{\blue \\ \wedge}\Upsilon_3)\beta)
&+&  \bar{\c}(\alpha((\Upsilon_1\substack{\blue \\ \wedge}\Upsilon_2)\substack{\red \\
\wedge}\Upsilon_3)\beta) \Big) \\ \nonumber
+\,\,\,\, (-1)^{|I(\Upsilon_1)||I(\Upsilon_2)|}\Big(\bar{\c}(\alpha(\Upsilon_2\substack{\red \\
\wedge}(\Upsilon_1\substack{\blue \\ \wedge}\Upsilon_3))\beta)
&+& \bar{\c}(\alpha(\Upsilon_2\substack{\blue \\ \wedge}(\Upsilon_1\substack{\red \\
\wedge}\Upsilon_3))\beta)\Big) \\ \nonumber
&=& 0.
\end{eqnarray}

\end{theorem}
\begin{proof}
It is an immediate consequence of Lemma \ref{lemma:52}  that $\{ \bar{\c}(\Upsilon)| \Upsilon \in
\BT_{n,i}\}$ generates $H^{n-3}((\hat{0},[n]^i))$.

Relation (\ref{relation:1h}):  This is also a consequence of Lemma~\ref{lemma:52}.  Indeed, first
note that 
$$c(\alpha(\Upsilon_2 \substack{\col \\ \wedge} \Upsilon_1)\beta) = c(\alpha(\Upsilon_1
\substack{\col \\ \wedge} \Upsilon_2)\beta,\tau),$$
where $\tau$ is the permutation that induces the linear extension that is just like postorder 
except that the internal nodes of 
$\Upsilon_2 $ are listed before those of $\Upsilon_1$. Since $\inv(\tau) = |I(\Upsilon_1)| |
I(\Upsilon_2)|$, relation (\ref{relation:1h}) follows from Lemma~\ref{lemma:52}.  (Note that since
Lemma~\ref{lemma:52} is a consequence only of the Type~I cohomology relation, one can view
(\ref{relation:1h}) as a consequence only of the Type~I cohomology relation.)

Relation (\ref{relation:3h}): Note that the following relation is a Type~II elementary
cohomology relation:
\begin{eqnarray*} 
 \bar{\c}(\alpha(\Upsilon_1\substack{\col \\ \wedge \\ \,}(\Upsilon_2\substack{\col \\ \wedge \\
\,\\ }
 \Upsilon_3))\beta) &+ &  \bar{\c}(\alpha((\Upsilon_1\substack{\col \\ \wedge \\
\,}\Upsilon_2)\substack{\col \\ \wedge \\ \,\\ }
 \Upsilon_3)\beta, \tau_1) 
 \\ &+ & \bar{\c}(\alpha(\Upsilon_2\substack{\col \\ \wedge \\ \,}(\Upsilon_1\substack{\col \\
\wedge \\ \,\\ }
 \Upsilon_3))\beta,\tau_2) = 0,
\end{eqnarray*}
where $\tau_1$ is the permutation that induces the linear extension that is like postorder but that
lists  the internal nodes of $\Upsilon_3$
before listing the root of $\Upsilon_1\wedge \Upsilon_2$, and  
$\tau_2$ is the permutation that induces the  linear extension that is like postorder but lists  the
internal nodes of $\Upsilon_1$
before listing the internal nodes of $\Upsilon_2$. So then $\inv(\tau_1)=|I(\Upsilon_3)|$ and
$\inv(\tau_2)=|I(\Upsilon_1)||I(\Upsilon_2)|$, and
using  Lemma \ref{lemma:52} we obtain relation (\ref{relation:3h}).

Relation (\ref{relation:5h}): Note that the following relation is a Type~III elementary cohomology
relation:

\begin{eqnarray*}
 \bar{\c}(\alpha(\Upsilon_1\substack{\red \\ \wedge}(\Upsilon_2\substack{\blue \\
\wedge}\Upsilon_3))\beta)
 &+& \bar{\c}(\alpha(\Upsilon_1\substack{\blue \\ \wedge}(\Upsilon_2\substack{\red \\
\wedge}\Upsilon_3))\beta)\\
+ \,\, \,\,\bar{\c}(\alpha((\Upsilon_1\substack{\red \\ \wedge}\Upsilon_2)\substack{\blue \\
\wedge}\Upsilon_3)\beta,\tau_1)
&+& \bar{\c}(\alpha((\Upsilon_1\substack{\blue \\ \wedge}\Upsilon_2)\substack{\red \\
\wedge}\Upsilon_3)\beta,\tau_1) \\
+\,\,\,\, \bar{\c}(\alpha(\Upsilon_2\substack{\red \\ \wedge}(\Upsilon_1\substack{\blue \\
\wedge}\Upsilon_3))\beta,\tau_2)
&+& \bar{\c}(\alpha(\Upsilon_2\substack{\blue \\ \wedge}(\Upsilon_1\substack{\red \\
\wedge}\Upsilon_3))\beta,\tau_2) \\
&=& 0,
\end{eqnarray*}
where as in the previous case, $\tau_1$ is the permutation that induces the linear extension that is
like postorder but that lists  the internal nodes of $\Upsilon_3$
before listing the root of $\Upsilon_1\wedge \Upsilon_2$, and  
$\tau_2$ is the permutation that induces the  linear extension that is like postorder but lists  the
internal nodes of $\Upsilon_1$
before listing the internal nodes of $\Upsilon_2$. So then $\inv(\tau_1)=|I(\Upsilon_3)|$ and
$\inv(\tau_2)=|I(\Upsilon_1)||I(\Upsilon_2)|$, and
using  Lemma \ref{lemma:52} we obtain relation (\ref{relation:5h}).

To complete the proof, we need to show that these relations generate all the cohomology relations.
In other words, we  need to show that $\tilde H^{n-3}((\hat 0, [n]^i)) = M/R$, where  $M$ is the
free ${\bf k}$-module with basis $\{ \bar{\c}(T,\sigma) : (T,\sigma) \in \BT_{n,i}\}$ and $R$ is the
submodule spanned by elements given in the relations (\ref{relation:1h}), (\ref{relation:3h}),
(\ref{relation:5h}).   We have already shown that $\rank \tilde H^{n-3}((\hat 0, [n]^i)) \le \rank M/R$. 
To complete the proof we need to establish the reverse inequality.
This is postponed to 
Section~\ref{section:combbasis}.  We will prove there, that a certain set $S$ of maximal chains
of $(\hat 0, [n]^i)$ whose cardinality equals $\rank \tilde H^{n-3}((\hat 0, [n]^i))$ generates $M/R$ by showing that
there is a straightening algorithm, which 
 using only the relations (\ref{relation:1h}),(\ref{relation:3h}),(\ref{relation:5h}), enables us
to express every 
 generator $\bar c(T,\sigma)$   as a linear combination of the elements of $S$. 
 It 
 follows that $\rank M/R \le |S| = \rank \tilde H^{n-3}((\hat 0, [n]^i))$. 
 See Remark~\ref{remark:finalstep}.
\end{proof}

\subsection{The isomorphism}\label{section:isomorphism}  In this section homology and cohomology are
taken over an arbitrary field ${\bf k}$, as is $\lie_2(n,i)$.

 The symmetric group $\sym_n$ acts  naturally  on $\Pi_n^w$.  Indeed,  let $\sigma\in \sym_n$ act on
the weighted blocks of $\pi\in \Pi_n^w$ by replacing each element $x$ of each weighted block of
$\pi$ with $\sigma(x)$.  Since the maximal elements  of $\Pi_n^w$ are fixed by each $\sigma\in
\sym_n$ and the order is preserved, each open interval $( \hat 0, [n]^i)$ is a $\sym_n$-poset. 
Hence by (\ref{eq:isocohom}) we have the $\sym_n$-module isomorphism,
$$\tilde H_{n-3}((\hat 0, [n]^i)) \simeq_{\sym_n} \tilde H^{n-3}((\hat 0, [n]^i)).$$  
The symmetric group $\sym_n$ also acts naturally on $\lie_2(n)$.  Indeed, let $\sigma \in \sym_n$
act by replacing letter $x$ of a bracketed permutation with $\sigma(x)$.  Since this action
preserves the number of brackets of each type, $\lie_2(n,i)$ is an $\sym_n$-module for each $i$.
In this section we obtain an explicit sign-twisted isomorphism between the $\sym_n$-modules $\tilde
H^{n-3}((\hat 0, [n]^i))$ and $\lie_2(n,i)$.

Define the {\em sign} of a binary tree $T$ recursively by 
$$\sgn(T) = \begin{cases} 1 &\mbox{ if } I(T) = \emptyset \\ (-1)^{|I(T_2)|}\sgn(T_1)\sgn(T_2) &
\mbox{ if } T = T_1 \land T_2 \end{cases}
$$
where $I(T)$ is the set of internal nodes of the binary tree $T$.
The sign of a bicolored   binary tree is defined to be the sign of the binary tree obtained by
removing the colors.

\begin{theorem}\label{theorem:liehomisomorphism}
 For each $i \in \{0,1,\dots,n-1\}$,  there is an $\sym_n$-module isomorphism
$\phi:\lie_2(n,i)\rightarrow \tilde H^{n-3}((\hat 0, [n]^i))\otimes \sgn_{n}$ determined by
\[
\phi([T,\sigma])=\sgn(\sigma)\sgn(T)\bar{\c}(T,\sigma),\]
for all $ (T,\sigma) \in \BT_{n,i}$.
\end{theorem}

Before proving the theorem we make a few preliminary observations.  The following lemma, which is
implicit in \cite[Proof of Theorem 5.4]{Wachs1998}, is easy to prove. For a binary tree $T$, let 
$a(T)b$ denote  a binary tree with $T$ as a subtree. 
\begin{lemma}\label{lemma:signTformulas} For all binary trees $T_1,T_2, T_3$,
\begin{enumerate}
  \item $\sgn(a(T_1 \wedge T_2)b)=(-1)^{|I(T_1)|+|I(T_2)|}\sgn(a(T_2 \wedge
T_1)b)$
 \vspace{.1in}
  \item $\sgn(a((T_1 \wedge T_2)\wedge T_3)b)=(-1)^{|I(T_3)|+1}\sgn(a(T_1 \wedge
(T_2\wedge T_3))b)$
\vspace{.1in}  \item $\sgn(a(T_2 \wedge (T_1\wedge
T_3))b)\!=\!(-1)^{|I(T_1)|+|I(T_2)|}\sgn(a(T_1 \wedge (T_2\wedge T_3))b)$.
 \end{enumerate}
\end{lemma}

For a word $w$ denote by $l(w)$ the \emph{length} or number of letters in $w$. We also have the
following easy relation, which we state as a lemma.
\begin{lemma}\label{lemma:signpermutationrelation}
 For $uw_1w_2v \in \sym_n$, where $u\,,w_1\,,w_2\,,v$ are subwords, 
\[
\sgn(uw_1w_2v)=(-1)^{l(w_1)l(w_2)}\sgn(uw_2w_1v).
\]
\end{lemma}

 We give a presentation of $\lie_2(n,i)$ in terms of labeled bicolored binary trees and a  slightly
modified, but clearly equivalent, form of the relations  (\ref{rln:lb1}), (\ref{rln:lb2}) and
(\ref{rln:lb3}) in the following proposition.

\begin{proposition}\label{proposition:binarybasislie}
The set $\{ [T,\sigma] : (T,\sigma) \in \BT_{n,i}\}$ is a generating set for $\lie_2(n,i)$,
subject only to the relations

  \begin{equation}[\alpha(\Upsilon_1\substack{\col \\ \wedge \\ \,}\Upsilon_2)\beta]= -
[\alpha(\Upsilon_2\substack{\col \\ \wedge \\ \,}\Upsilon_1)\beta]  \label{relation:1}\end{equation}

\begin{eqnarray}\label{relation:3}  [\alpha(\Upsilon_1\substack{\col \\ \wedge \\
\,}(\Upsilon_2\substack{\col \\ \wedge \\ \,\\ }\Upsilon_3))\beta]
&-& [\alpha((\Upsilon_1\substack{\col \\ \wedge \\ \,}\Upsilon_2)\substack{\col \\ \wedge \\
\,}\Upsilon_3)\beta]\\ 
&-&  [\alpha(\Upsilon_2\substack{\col \\ \wedge \\ \,}(\Upsilon_1\substack{\col \\ \wedge \\
\,}\Upsilon_3))\beta]
\nonumber  \\ &=& 0\nonumber
\end{eqnarray}

\begin{eqnarray} \label{relation:5} 
[\alpha(\Upsilon_1\substack{\red \\ \wedge}(\Upsilon_2\substack{\blue \\ \wedge}\Upsilon_3))\beta]
 &+& [\alpha(\Upsilon_1\substack{\blue \\ \wedge}(\Upsilon_2\substack{\red \\
\wedge}\Upsilon_3))\beta]\\ \nonumber
- \,\, \,\, [\alpha((\Upsilon_1\substack{\red \\ \wedge}\Upsilon_2)\substack{\blue \\
\wedge}\Upsilon_3)\beta]
&-&  [\alpha((\Upsilon_1\substack{\blue \\ \wedge}\Upsilon_2)\substack{\red \\
\wedge}\Upsilon_3)\beta] \\ \nonumber
-\,\,\,\, [\alpha(\Upsilon_2\substack{\red \\ \wedge}(\Upsilon_1\substack{\blue \\
\wedge}\Upsilon_3))\beta]
&-& [\alpha(\Upsilon_2\substack{\blue \\ \wedge}(\Upsilon_1\substack{\red \\
\wedge}\Upsilon_3))\beta]\\ \nonumber
&=& 0.
\end{eqnarray}

\end{proposition}

\begin{proof}[Proof of Theorem~\ref{theorem:liehomisomorphism}]
 The map $\phi$ maps generators onto generators  and clearly respects the $\sym_n$ action. We will
prove
 that the map $\phi$ extends to a well defined  homomorphism by showing that the relations in
$\lie_2(n,i)$ of
 the generators in Proposition~\ref{proposition:binarybasislie} map onto to the relations in Theorem
\ref{proposition:binarybasishomology}. Since by Theorem~\ref{proposition:binarybasishomology} (whose
proof will be completed in Section~\ref{section:combbasis}),  the relations in
 Theorem~\ref{proposition:binarybasishomology} span all the relations in cohomology, this also
 implies that the map is an isomorphism.
 
For each $\Upsilon_j$ in the relations of Proposition~\ref{proposition:binarybasislie}, let $w_j$ and $T_j$ be such that $\Upsilon_j =(T_j,w_j)$.
Let $u$ be the  permutation
labeling the portion $a$ of the tree  corresponding to the preamble $\alpha$, and let $v$ be the  permutation
labeling the portion $b$ of the tree  corresponding to the tail $\beta$.
Using Lemmas \ref{lemma:signTformulas} and \ref{lemma:signpermutationrelation} we have
the following.

{\em Relation  (\ref{relation:1})}:  Let $\wedge \in \{\substack{\blue\\ \wedge}, \substack{\red\\
\wedge}\}$.  Then
\begin{align*}
 \phi([\alpha(\Upsilon_2 \wedge \Upsilon_1)\beta])&
    =\sgn(uw_2w_1v)\sgn(a(T_2 \wedge T_1)b)\bar{c}(\alpha(\Upsilon_2 \wedge
\Upsilon_1)\beta)\\ \\
    &=\sgn(uw_1w_2v)\sgn(a(T_1 \wedge T_2)b)\\
    &\hspace{.3in}\cdot (-1)^{l(w_1)l(w_2)+|I(T_1)|+|I(T_2)|}
\bar{c}(\alpha(\Upsilon_2 \wedge \Upsilon_1)\beta)\\ \\
     &=\sgn(uw_1w_2v) \sgn(a(T_1 \wedge T_2)b)\\ 
    &\hspace{.3in}\cdot
(-1)^{(|I(T_1)|+1)(|I(T_2)|+1)+|I(T_1)|+|I(T_2)|}\bar{c}
(\alpha(\Upsilon_2 \wedge \Upsilon_1)\beta)\\ \\
    &=\sgn(uw_1w_2v)\sgn(a(T_1 \wedge T_2)b)\\
    &\hspace{.3in}\cdot  (-1)^{|I(T_1)||I(T_2)|+1}\bar{c}(\alpha(\Upsilon_2 \wedge
\Upsilon_1)\beta).
  \end{align*}
  
 Hence,
 \begin{align*}
 \phi([\alpha(\Upsilon_1\wedge \Upsilon_2)\beta])&+\phi([\alpha(\Upsilon_2\wedge \Upsilon_1)\beta])
=\sgn(uw_1w_2v)\sgn(a(T_1 \wedge T_2)b)\\
&\cdot{\big (}\bar{c}(\alpha(\Upsilon_1 \wedge \Upsilon_2)\beta) -(-1)^{|I(\Upsilon_1)|
|I(\Upsilon_2)|}\bar{c}(\alpha(\Upsilon_2 \wedge \Upsilon_1)\beta){\big)}.
\end{align*}
We conclude that relation (\ref{relation:1}) maps to relation (\ref{relation:1h}).

{\em Relations (\ref{relation:3}) and (\ref{relation:5})}:  Let $\wedge, \tilde\wedge \in
\{\substack{\blue\\ \wedge}, \substack{\red\\ \wedge}\}$.  Then
  \begin{align*}
  \phi([\alpha((\Upsilon_1 \wedge \Upsilon_2)\tilde \wedge
\Upsilon_3)\beta])&=\sgn(uw_1w_2w_3v)\sgn(a((T_1 \wedge T_2)\wedge
T_3)b)\\
    &\hspace{.3in} \cdot \bar{c}(\alpha((\Upsilon_1 \wedge \Upsilon_2)\tilde\wedge \Upsilon_3)\beta)
    \\ \\
    &=\sgn(uw_1w_2w_3v) \sgn(a(T_1 \wedge (T_2 \wedge T_3))b)\\
    &\hspace{.3in}\cdot (-1)^{|I(T_3)|+1} \bar{c}(\alpha((\Upsilon_1 \wedge \Upsilon_2)\tilde\wedge
\Upsilon_3)\beta).
    \\ \\
  \phi([\alpha(\Upsilon_2 \wedge (\Upsilon_1 \tilde \wedge
\Upsilon_3))\beta])&=\sgn(uw_2w_1w_3v)\sgn(a(T_2 \wedge (T_1 \wedge
T_3))b)\\
    &\hspace{.3in}\cdot\ \bar{c}(\alpha(\Upsilon_2 \wedge (\Upsilon_1 \tilde \wedge
\Upsilon_3))\beta)
    \\ \\
    &=\sgn(uw_1w_2w_3v)\sgn(a(T_1 \wedge (T_2 \wedge T_3))b) \\
    &\hspace{.3in}\cdot (-1)^{l(w_1)l(w_2)+|I(T_1)|+|I(T_2)|}\bar{c}(\alpha(\Upsilon_2 \wedge
(\Upsilon_1 \tilde\wedge \Upsilon_3))\beta)
    \\ \\
    &=\sgn(uw_1w_2w_3v)  \sgn(a(T_1 \wedge (T_2 \wedge T_3))b)\\
    &\hspace{.3in}\cdot (-1)^{|I(T_1)||I(T_2)|+1}\bar{c}(\alpha(\Upsilon_2 \wedge (\Upsilon_1 \tilde
\wedge \Upsilon_3))\beta).\\
  \end{align*}
Hence,
  \begin{align} \label{mapeq} 
 \phi([\alpha(\Upsilon_1\wedge&(\Upsilon_2 \tilde \wedge \Upsilon_3))\beta])
-\phi([\alpha((\Upsilon_1\wedge \Upsilon_2)\tilde \wedge \Upsilon_3)\beta])
-\phi([\alpha(\Upsilon_2\wedge(\Upsilon_1\tilde \wedge \Upsilon_3))\beta]) 
\\   \nonumber =
&\sgn(uw_1w_2w_3v)\sgn(a(T_1 \wedge (T_2 \wedge T_3))b)
\\ \nonumber
& \cdot\Big(\bar{\c}(\alpha(\Upsilon_1 \wedge(\Upsilon_2 \tilde\wedge \Upsilon_3))\beta)
+(-1)^{|I(T_3)|}\bar{\c}(\alpha((\Upsilon_1 \wedge \Upsilon_2) \tilde\wedge \Upsilon_3)\beta)  \\
\nonumber
&\hspace{.2in}+ (-1)^{|I(\Upsilon_1)||I(\Upsilon_2)|}\bar{\c}(\alpha(\Upsilon_2 \wedge (\Upsilon_1 \tilde \wedge
\Upsilon_3))\beta)\Big).
\end{align}
By setting $\wedge = \tilde \wedge$ in (\ref{mapeq}) we conclude that relation (\ref{relation:3}) 
maps
to relation (\ref{relation:3h}).
By adding (\ref{mapeq}) with  $\wedge = \substack{\blue \\ \wedge \\ \,}$ and $\tilde \wedge = \substack{\red \\ \wedge \\ \,}$ to (\ref{mapeq}) with
$\wedge = \substack{\red \\ \wedge \\ \,}$ and $\tilde \wedge = \substack{\blue \\ \wedge \\ \,}$, we are also able to conclude that  relation
(\ref{relation:5}) maps to relation (\ref{relation:5h}).
\end{proof}
  
Theorem~\ref{theorem:liehomisomorphism} and Corollary~\ref{proposition:dimensionhat} yield the
following result.  
\begin{corollary}[Liu \cite{Liu2010}, Dotsenko and Khoroshkin \cite{DotsenkoKhoroshkin2010}] For  $0
\le i \le n-1$,  $\dim \lie_2(n,i) = |\mathcal T_{n,i}|$.
\end{corollary}

\section{Combinatorial bases}\label{section:combinatorialbases}

Throughout this section we take homology and cohomology over the integers or over an arbitrary field ${\bf
k}$. 
 We present three bases for cohomology and one for homology of each interval $(\hat 0, [n]^i)$.  Two of the three cohomology bases
correspond to known bases for $\lie_2(n,i)$ and one appears to be new.  The homology basis also
appears to be new. We also present two new bases for cohomology of the full weighted partition poset $\Pi_n^w \setminus \{\hat 0\}$.

We say that a labeled  binary tree  is {\it normalized} if  the leftmost leaf of each subtree has
the smallest label in the subtree.  Using cohomology relation (\ref{relation:1h}), we see that
$\tilde H^{n-3}((\hat 0, [n]^i))$ is generated by maximal chains of the form $\bar c(T,\sigma)$,
where  $(T,\sigma)$ is a normalized binary tree in  $\BT_{n,i}$.   The first two bases for $\tilde
H^{n-3}((\hat 0, [n]^i))$  presented here are  subsets of this set of maximal chains.

\subsection{A bicolored comb basis for $\tilde H^{n-3}((\hat 0, [n]^i))$ and
$\lie_2(n,i)$}\label{section:combbasis}

In this section we present a generalization  of a classical basis for $\tilde H^{n-3}(\overline{\Pi}_n)$
and a corresponding generalization of a classical basis  for $\mathcal Lie(n)$; the classical bases
are sometimes referred to as  comb bases (see \cite[Section 4]{Wachs1998}).   The generalization for
$\lie_2(n)$ is due to Bershtein, Dotsenko and Khoroshkin (see \cite{BershteinDotsenkoKhoroshkin2007}
and \cite[Theorem 4]{DotsenkoKhoroshkin2007}).

A {\it bicolored comb} is a  normalized bicolored binary tree that satisfies the following coloring
restriction:  for each internal node $x$ whose right child $y$ is not a leaf, $x$ is colored red and
 $y$ is colored blue.  Let $\comb_n$ be the set of bicolored combs in $\BT_n$ and let $\comb_{n,i}$
be the set of bicolored combs in $\BT_{n,i}$
The set of bicolored combs for $n=3$ is depicted in Figure \ref{fig:bicoloredcombs}.

\begin{figure}[h]
        \centering
         \begin{tikzpicture}[thick,scale=0.6]
\begin{scope}[xshift=0,yshift=-1cm]

\tikzstyle{every node}=[fill, draw,inner sep=2pt,scale=0.8]
    \draw [color=blue] (1,1)  node (i1){};
    \draw [color=blue] (2,2)  node (i2){};

\tikzstyle{every node}=[inner sep=1pt, minimum width=14pt,scale=0.7]

    \draw (0,0)  node (m){$1$};
    \draw (2,0)  node (l1){$2$};
    \draw (3,1)  node (l2){$3$};

    \draw (m) --  (i1) ;
    \draw (i1) --  (l1) ;
    \draw (i1) --  (i2) ;
    \draw (i2) --  (l2) ;
\end{scope}

\begin{scope}[xshift=3.5cm,yshift=0]
\tikzstyle{every node}=[fill, draw,inner sep=2pt,scale=0.8]
    \draw [circle,color=red] (1,1)  node (i1){};
    \draw [color=blue] (2,2)  node (i2){};

\tikzstyle{every node}=[inner sep=1pt, minimum width=14pt,scale=0.7]

    \draw (0,0)  node (m){$1$};
    \draw (2,0)  node (l1){$2$};
    \draw (3,1)  node (l2){$3$};

    \draw (m) --  (i1) ;
    \draw (i1) --  (l1) ;
    \draw (i1) --  (i2) ;
    \draw (i2) --  (l2) ;
\end{scope}

\begin{scope}[xshift=7cm,yshift=0]
\tikzstyle{every node}=[fill, draw,inner sep=2pt,scale=0.8]
    \draw [color=blue] (1,1)  node (i1){};
    \draw [circle,color=red] (2,2)  node (i2){};

\tikzstyle{every node}=[inner sep=1pt, minimum width=14pt,scale=0.7]

    \draw (0,0)  node (m){$1$};
    \draw (2,0)  node (l1){$2$};
    \draw (3,1)  node (l2){$3$};

    \draw (m) --  (i1) ;
    \draw (i1) --  (l1) ;
    \draw (i1) --  (i2) ;
    \draw (i2) --  (l2) ;
\end{scope}
\begin{scope}[xshift=10.5cm,yshift=-1cm]
\tikzstyle{every node}=[fill, draw,inner sep=2pt,scale=0.8]
    \draw [circle,color=red] (1,1)  node (i1){};
    \draw [circle,color=red] (2,2)  node (i2){};

\tikzstyle{every node}=[inner sep=1pt, minimum width=14pt,scale=0.7]

    \draw (0,0)  node (m){$1$};
    \draw (2,0)  node (l1){$2$};
    \draw (3,1)  node (l2){$3$};

    \draw (m) --  (i1) ;
    \draw (i1) --  (l1) ;
    \draw (i1) --  (i2) ;
    \draw (i2) --  (l2) ;
\end{scope}

\begin{scope}[xshift=0,yshift=-5cm]

\tikzstyle{every node}=[fill, draw,inner sep=2pt,scale=0.8]
    \draw [color=blue] (1,1)  node (i1){};
    \draw [color=blue] (2,2)  node (i2){};

\tikzstyle{every node}=[inner sep=1pt, minimum width=14pt,scale=0.7]

    \draw (0,0)  node (m){$1$};
    \draw (2,0)  node (l1){$3$};
    \draw (3,1)  node (l2){$2$};

    \draw (m) --  (i1) ;
    \draw (i1) --  (l1) ;
    \draw (i1) --  (i2) ;
    \draw (i2) --  (l2) ;
\end{scope}

\begin{scope}[xshift=3.5cm,yshift=-6cm]
\tikzstyle{every node}=[fill, draw,inner sep=2pt,scale=0.8]
    \draw [circle,color=red] (1,1)  node (i1){};
    \draw [color=blue] (2,2)  node (i2){};

\tikzstyle{every node}=[inner sep=1pt, minimum width=14pt,scale=0.7]

    \draw (0,0)  node (m){$1$};
    \draw (2,0)  node (l1){$3$};
    \draw (3,1)  node (l2){$2$};

    \draw (m) --  (i1) ;
    \draw (i1) --  (l1) ;
    \draw (i1) --  (i2) ;
    \draw (i2) --  (l2) ;
\end{scope}

\begin{scope}[xshift=7cm,yshift=-6cm]
\tikzstyle{every node}=[fill, draw,inner sep=2pt,scale=0.8]
    \draw [color=blue] (1,1)  node (i1){};
    \draw [circle,color=red] (2,2)  node (i2){};

\tikzstyle{every node}=[inner sep=1pt, minimum width=14pt,scale=0.7]

    \draw (0,0)  node (m){$1$};
    \draw (2,0)  node (l1){$3$};
    \draw (3,1)  node (l2){$2$};

    \draw (m) --  (i1) ;
    \draw (i1) --  (l1) ;
    \draw (i1) --  (i2) ;
    \draw (i2) --  (l2) ;
\end{scope}
\begin{scope}[xshift=10.5cm,yshift=-5cm]
\tikzstyle{every node}=[fill, draw,inner sep=2pt,scale=0.8]
    \draw [circle,color=red] (1,1)  node (i1){};
    \draw [circle,color=red] (2,2)  node (i2){};

\tikzstyle{every node}=[inner sep=1pt, minimum width=14pt,scale=0.7]

    \draw (0,0)  node (m){$1$};
    \draw (2,0)  node (l1){$3$};
    \draw (3,1)  node (l2){$2$};
    
    \draw (m) --  (i1) ;
    \draw (i1) --  (l1) ;
    \draw (i1) --  (i2) ;
    \draw (i2) --  (l2) ;
\end{scope}

\begin{scope}[xshift=4cm,yshift=-3cm]
\tikzstyle{every node}=[fill, draw,inner sep=2pt,scale=0.8]
    \draw [color=blue] (3,1)  node (i1){};
    \draw [circle,color=red] (2,2)  node (i2){};

\tikzstyle{every node}=[inner sep=1pt, minimum width=14pt,scale=0.7]

    \draw (2,0)  node (m){$2$};
    \draw (4,0)  node (l1){$3$};
    \draw (1,1)  node (l2){$1$};
    
    \draw (m) --  (i1) ;
    \draw (i1) --  (l1) ;
    \draw (i1) --  (i2) ;
    \draw (i2) --  (l2) ;
\end{scope}
\end{tikzpicture}
 \caption{Set of bicolored combs for $n=3$}
\label{fig:bicoloredcombs}
  \end{figure}
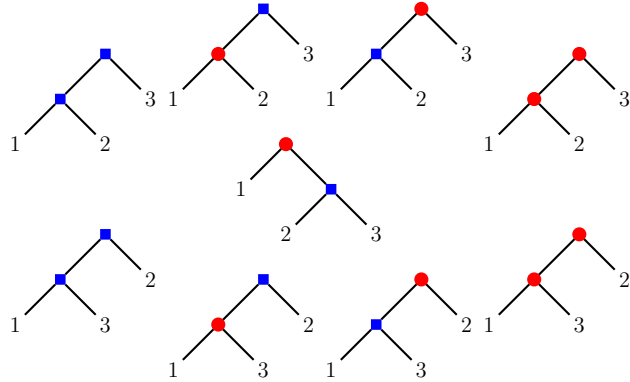

We refer to such trees as bicolored {\em combs} because the monochromatic ones are the usual left
combs in the sense of  \cite{Wachs1998};
indeed  if a  bicolored comb  is
monochromatic then the right child of every internal node is a leaf and   the left-most leaf label
of the tree is the smallest label.  
In this case we get the usual left comb, which has the form,
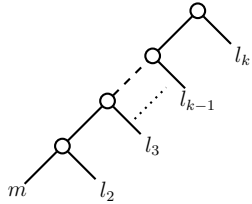
\begin{center} \begin{tikzpicture}[thick,scale=0.6]

\tikzstyle{every node}=[draw,scale=0.5]

    \draw [circle,radius=20pt,color=black] (1,1)  node (i1){};
    \draw [circle,color=black] (2,2)  node (i2){};
    \draw [circle,color=black] (4,4)  node (i4){};
    \draw [circle,color=black] (3,3)  node (i3){};

\tikzstyle{every node}=[inner sep=1pt, minimum width=14pt,scale=0.7]

    \draw (0,0)  node (m){$m$};
    \draw (2,0)  node (l1){$l_2$};
    \draw (3,1)  node (l2){$l_3$};
    \draw (4,2)  node (l3){$l_{k-1}$};
    \draw (5,3)  node (l4){$l_k$};

    \draw (m) --  (i1) ;
    \draw (i1) --  (l1) ;
    \draw (i2) --  (l2) ;
    \draw (i3) --  (l3) ;
    \draw (i4) --  (l4) ;
    \draw (i1) --  (i2) ;
    \draw [dashed, thick] (i2) --  (i3) ;
    \draw [dotted, thick] (2.6,1.6) --  (3.3,2.3) ;

    \draw (i3) --  (i4) ;

\end{tikzpicture}\end{center}
where $m$ and all  the $l_j$ are leaves, and $m$ is the smallest label usually 1.

Bershtein, Dotsenko and Khoroshkin 
{\cite[Lemma 5.2] {BershteinDotsenkoKhoroshkin2007}}  present the results that
$\{[T,\sigma] : (T,\sigma) \in \comb_{n,i}\}$ spans $\lie_2(n,i)$ and $|\comb_n| = n^{n-1}$. Since
it was already known from \cite{Liu2010} and \cite{DotsenkoKhoroshkin2007} that $\dim \lie_2(n) =
n^{n-1}$, they conclude that $\{[T,\sigma] : (T,\sigma) \in \comb_{n,i}\}$ is a basis for
$\lie_2(n,i)$.   For the sake of completeness we give a detailed proof that the corresponding set
$\{\bar c(T,\sigma) : (T,\sigma) \in \comb_{n,i}\}$ spans cohomology and we give an 
alternative proof of $|\comb_n|= n^{n-1}$.

\begin{proposition}\label{proposition:combbasisspans}
The set $\{\bar c(T,\sigma) : (T,\sigma) \in \comb_{n,i} \}$ spans $\tilde H^{n-3}((\hat 0,
[n]^i))$, for  all $0 \le i \le n-1$.
\end{proposition}

\begin{proof}  We prove this result by ``straightening'' via the relations in  
Theorem~\ref{proposition:binarybasishomology}.
Define the {\em weight} $w(T)$ of a bicolored binary tree $T$ to be 
$$w(T) = \sum_{x \in I(T)} r(x),$$
where $I(T)$ is the set of internal nodes  of  $T$ and $r(x)$ is the number of internal nodes in the right subtree of $x$.  We say a node $y$ of $T$ is  a {\em right descendent} of a node $x$ if $y$ can be reached from $x$ along a path of right edges.  Next we define   an  {\em inversion} of $T$ to be a pair of internal nodes $(x,y)$ of $T$ such that $x$ is blue and $y$ is a red right descendent of $x$.  Let $\inv(T)$ be the number of inversions of $T$.  The {\em weight-inversion pair} of $T$ is $(w(T),\inv(T))$.  We order these pairs lexicographically, that is we say  $(w(T),\inv(T))< (w(T^\prime),\inv(T^\prime))$ if either $w(T) < w(T^\prime)$ or $w(T)=w(T^\prime)$ and $\inv(T) < \inv(T^\prime)$.  For $\Upsilon = (T,\sigma) \in \mathcal BT_n$, let $w(\Upsilon) := w(T)$ and $\inv(\Upsilon) := \inv(T)$.  Also define the weight-inversion pair of $\Upsilon$ to be that of $T$.

It follows from (\ref{relation:1h}) that the chains of the form $\bar c(\Upsilon)$, where $\Upsilon$
is a normalized bicolored binary tree in $\BT_{n,i}$, span $\tilde H^{n-3}((\hat 0, [n]^i))$. 
Hence to prove the result we need only  show that if $\Upsilon \in \BT_{n,i}$ is a normalized
bicolored binary tree that is not a bicolored comb then $\bar c(\Upsilon)$ can be expressed as a
linear combination of chains of the form $\bar c(\Upsilon^\prime)$, where $\Upsilon^\prime$ is a
normalized bicolored binary tree in $\BT_{n,i}$ such that  $(w(\Upsilon^\prime),\inv(\Upsilon^\prime)) < (w(\Upsilon),\inv(\Upsilon))$ in lexicographic order.
It will then follow by induction on the weight-inversion pair that $\bar c(\Upsilon)$ can be expressed as a
linear combination of chains of the form $\bar c(\Upsilon^\prime)$, where $\Upsilon^\prime \in
\comb_{n,i}$.

Now let $\Upsilon \in \BT_{n,i}$ be a normalized  bicolored binary tree that is not a bicolored
comb.  Then $\Upsilon$ must have a subtree of one of the following   forms: $\Upsilon_1
\substack{\blue \\ \wedge} (\Upsilon_2 \substack{\blue \\ \wedge} \Upsilon_3)$, $\Upsilon_1
\substack{\red \\ \wedge} (\Upsilon_2 \substack{\red \\ \wedge} \Upsilon_3)$, or $\Upsilon_1
\substack{\blue \\ \wedge} (\Upsilon_2 \substack{\red \\ \wedge} \Upsilon_3)$. We will show that in
all three cases $\bar c(\Upsilon)$ can be expressed as a linear combination of chains with a smaller
weight-inversion pair.

\noindent{\bf Case 1:} $\Upsilon$ has a subtree of the form $\Upsilon_1 \substack{\blue \\ \wedge}
(\Upsilon_2 \substack{\blue \\ \wedge} \Upsilon_3)$.
We can therefore express $\Upsilon$ as $\alpha(\Upsilon_1 \substack{\blue \\ \wedge} (\Upsilon_2
\substack{\blue \\ \wedge} \Upsilon_3))\beta$. Using  relation (\ref{relation:3h}) (and  relation
(\ref{relation:1h})) we have that
   
\[ \bar c( \alpha(\Upsilon_1 \substack{\blue \\ \wedge} (\Upsilon_2 \substack{\blue \\ \wedge}
\Upsilon_3))\beta )= 
\pm\bar c(\alpha((\Upsilon_1 \substack{\blue \\ \wedge} \Upsilon_2) \substack{\blue \\ \wedge}
\Upsilon_3)\beta)
\pm \bar c(\alpha((\Upsilon_1 \substack{\blue \\ \wedge} \Upsilon_3)\substack{\blue \\ \wedge}
\Upsilon_2)\beta).
\]
(The  signs in the relations of Theorem~\ref{proposition:binarybasishomology} are not relevant here and have therefore been suppressed.)

It is easy to see that 
\begin{eqnarray*} w(\alpha((\Upsilon_1 \substack{\blue \\ \wedge} \Upsilon_2) \substack{\blue \\
\wedge} \Upsilon_3)\beta) &=& w( \alpha((\Upsilon_1 \substack{\blue \\ \wedge}
\Upsilon_3)\substack{\blue \\ \wedge} \Upsilon_2)\beta) \\ &=& w(\alpha(\Upsilon_1 \substack{\blue
\\ \wedge} (\Upsilon_2 \substack{\blue \\ \wedge} \Upsilon_3))\beta) -|I(\Upsilon_3)|-1.\end{eqnarray*}
Hence $\bar c(\Upsilon)$ can be expressed as a linear combination of chains of smaller weight,  and therefore of smaller weight-inversion pair.

\noindent{\bf Case 2:} $\Upsilon$ has a subtree of the form $\Upsilon_1 \substack{\red \\ \wedge}
(\Upsilon_2 \substack{\red \\ \wedge} \Upsilon_3)$.
An  argument analogous to that of Case 1 shows that $\bar c(\Upsilon)$ can be expressed as a linear combination of chains of smaller  weight-inversion pair.

\noindent{\bf Case 3:} $\Upsilon$ has a subtree of the form $\Upsilon_1 \substack{\blue \\ \wedge}
(\Upsilon_2 \substack{\red \\ \wedge} \Upsilon_3)$.
Using  relation (\ref{relation:5h}) (and  relation (\ref{relation:1h})) we have that

\begin{eqnarray*} \bar c(\alpha( \Upsilon_1 \substack{\blue \\ \wedge} (\Upsilon_2 \substack{\red\\
\wedge} \Upsilon_3) )\beta)&=& 
\pm \bar c(\alpha(\Upsilon_1 \substack{\red \\ \wedge} (\Upsilon_2\substack{\blue \\ \wedge}
\Upsilon_3))\beta) \\ & &
\pm \bar c(\alpha((\Upsilon_1 \substack{\blue \\ \wedge} \Upsilon_2)\substack{\red \\ \wedge}
\Upsilon_3)\beta) \\ & &
\pm \bar c(\alpha((\Upsilon_1 \substack{\red \\ \wedge} \Upsilon_2 )\substack{\blue \\ \wedge}
\Upsilon_3)\beta) \\ & &
\pm \bar c(\alpha((\Upsilon_1 \substack{\blue \\ \wedge} \Upsilon_3 )\substack{\red \\ \wedge}
\Upsilon_2)\beta) \\ & &
\pm \bar c(\alpha((\Upsilon_1 \substack{\red \\ \wedge} \Upsilon_3 )\substack{\blue \\ \wedge}
\Upsilon_2)\beta) .
\end{eqnarray*}

Just as in Case 1, all the labeled bicolored trees on the right hand side of the equation, except for the first, have weight smaller than that of $\alpha(\Upsilon_1 \substack{\blue \\ \wedge} (\Upsilon_2 \substack{\red\\
\wedge} \Upsilon_3) )\beta$.  The first labeled bicolored  tree $\alpha(\Upsilon_1 \substack{\red \\ \wedge} (\Upsilon_2\substack{\blue \\ \wedge}
\Upsilon_3))\beta$ has the same weight as that of $\alpha(\Upsilon_1 \substack{\blue \\ \wedge} (\Upsilon_2 \substack{\red\\
\wedge} \Upsilon_3) )\beta$.  However the inversion number is reduced, that is
$$\inv( \alpha(\Upsilon_1 \substack{\red \\ \wedge} (\Upsilon_2\substack{\blue \\ \wedge}
\Upsilon_3))\beta) = \inv(\alpha(\Upsilon_1 \substack{\blue \\ \wedge} (\Upsilon_2 \substack{\red\\
\wedge} \Upsilon_3) )\beta)-1.$$
Hence the weight-inversion pair for the first bicolored labeled tree is less than that of $\Upsilon:=\alpha(\Upsilon_1 \substack{\blue \\ \wedge} (\Upsilon_2 \substack{\red\\
\wedge} \Upsilon_3) )\beta$ just as it is for the other bicolored labeled trees on the right hand side of the equation.  We conclude that $\bar c(\Upsilon)$ can be expressed as a linear combination of chains of smaller weight-inversion pair.
\end{proof}

\begin{proposition}[Bershtein, Dotsenko and Khoroshkin
\cite {BershteinDotsenkoKhoroshkin2007}]
\label{proposition:combbasiscardinality} Let $n \ge 1$.  Then
 $|\comb_n|=n^{n-1}$.
\end{proposition}

\begin{proof} We present a different proof than that of \cite {BershteinDotsenkoKhoroshkin2007}.
Our proof is by induction on $n$. The cases $|\comb_1|=1$ and $|\comb_2|=2$ are trivially verified.
For $n \ge 3$ assume that
$|\comb_k|=k^{k-1}$ for any $k<n$.  
We claim that 
\begin{equation} \label{combcount}
|\comb_n|=(n-1)^{n-1}+\sum_{k=1}^{n-1}\binom{n-1}{k}(n-k)^{n-k-1}(k-1)^{k-1}.
\end{equation}
To prove the claim we show that the  term that precedes the summation counts blue-rooted bicolored
combs and the $k$th term of the sum counts red-rooted bicolored combs whose right subtree has $k$
leaves.  To construct a blue-rooted bicolored comb $T \in \comb_n$, we can
choose the right subtree, which is a leaf, in $n-1$ different ways, and the left subtree, which is a
bicolored comb,  in $(n-1)^{n-2}$
different ways, by induction.   Hence there are $(n-1)^{n-1}$ blue-rooted bicolored combs.  
To construct a red-rooted bicolored comb $T \in \comb_n$ whose right subtree has $k$ leaves, first
choose
$k$ labels for the right subtree in $\binom{n-1}{k}$ different ways. Then choose a right subtree
that uses these labels. Since the right subtree must be a blue-rooted bicolored comb, there are
$(k-1)^{k-1}$ ways to choose such a subtree by the previous case.  Now choose the left subtree,
which is a bicolored comb, in  $(n-k)^{n-k-1}$ different ways by induction.

By setting $x,z:=-1$, $y:=n$ and $n:=n-1$ in Abel's polynomial identity
(\ref{proposition:abelidentity}), we have 
\begin{eqnarray*}
(n-1)^{n-1}&=&-\sum_{k=0}^{n-1}\binom{n-1}{k}(n-k)^{n-k-1}(k-1)^{k-1}\\
&=& n^{n-1} -\sum_{k=1}^{n-1}\binom{n-1}{k}(n-k)^{n-k-1}(k-1)^{k-1}.
\end{eqnarray*}
It therefore follows from (\ref{combcount}) that $|\comb_n|=n^{n-1}$.
\end{proof}

\begin{theorem}\label{proposition:combbasiscohomology}
 The set $\{\bar{c}(T,\sigma) : (T,\sigma) \in \comb_{n,i}\}$ 
 is a basis for \newline
 $\tilde H^{n-3}((\hat 0, [n]^i))$.
\end{theorem} 
\begin{proof}It follows from Propositions~\ref{proposition:combbasisspans}
and~\ref{proposition:combbasiscardinality} that $\{\bar{c}(T,\sigma) : (T,\sigma) \in \comb_n\}$
spans $\oplus_{i=0}^{n-1} \tilde H^{n-3}((\hat 0, [n]^i))$  and is of cardinality $n^{n-1}$.  Since,
by Corollary \ref{proposition:dimensionhat}, $\rank \oplus_{i=0}^{n-1} \tilde H^{n-3}((\hat 0,
[n]^i)) = n^{n-1}$, the result holds.
\end{proof}

\begin{remark} \label{remark:finalstep} Since  the only relations used in the straightening
algorithm of Proposition~\ref{proposition:combbasisspans} are the relations of the presentation
given in 
Theorem~\ref{proposition:binarybasishomology},  it follows from
Theorem~\ref{proposition:combbasiscohomology} that these relations are  the {\em only} relations
needed to present $\tilde H^{n-3}((\hat 0, [n]^i))$.  Thus the final step of the proof of
Theorem~\ref{proposition:binarybasishomology} is now complete.  \end{remark}    

\begin{remark} Note that by switching left and right, small and large, blue and red, we get 8
different variations of bicolored comb bases.  
\end{remark}

\subsection{A bicolored Lyndon basis for $\tilde H^{n-3}((\hat 0, [n]^i))$ and $\lie_2(n,i)$}
\label{decsubsec}
In this section, we describe the ascent-free chains of the EL-labeling of  $[\hat 0,[n]^i]$ given in
Theorem~\ref{theorem:ellabelingposet}.  Recall from  Theorem~\ref{elth} that these yield a basis for
$H^{n-3}((\hat 0,[n]^i))$. By applying the isomorphism of Theorem~\ref{theorem:liehomisomorphism}, 
one gets a corresponding basis for $\lie_2(n,i)$, which  is the classical Lyndon basis for $\lie(n)$
when $i=0,n-1$. 

We begin by recalling the Lyndon basis for $\lie(n)$.  A Lyndon tree is a labeled binary tree
$(T,\sigma)$ such that for each internal node $x$ of $T$ the smallest leaf label of the subtree
$T_x$ rooted at $x$ is in the left subtree of $T_x$ and the second smallest label is in the right
subtree of $T_x$.    Let ${\textsf{Lyn}}_n$ be the set of Lyndon trees whose leaf labels form the
set $[n]$.  The set $\{[T,\sigma] : (T,\sigma) \in {\textsf{Lyn}}_n\}$  is the classical Lyndon
basis for $\lie(n)$.

For each internal node $x$ of a binary tree let $L(x)$ denote the left child of $x$ and $R(x)$
denote the right child.  For each node $x$ of a bicolored labeled binary tree $(T,\sigma)$ define
its  {\em valency} $v(x)$ to be the smallest leaf label of the subtree rooted at $x$.  A Lyndon
tree is depicted in Figure \ref{fig:lyndonandvalency} illustrating the valencies of the
internal nodes. The following alternative characterization of Lyndon tree is easy to
verify.

\begin{proposition} \label{prop:valen} Let $(T,\sigma)$ be a  labeled binary tree.  Then
$(T,\sigma)$ is a Lyndon tree if  and only if it is normalized and for every internal node $x$ of
$T$ we have 
\begin{equation} \label{eq:lynnode} v(R(L(x))> v(R(x)).\end{equation}
\end{proposition}

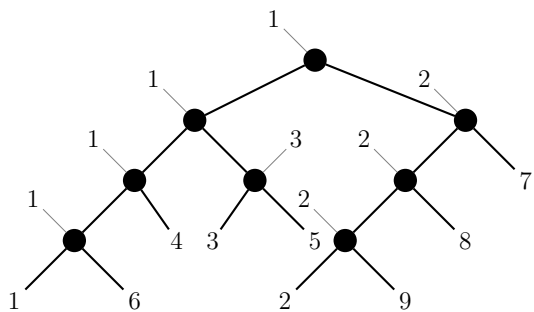
\begin{figure}[h]
        \centering
        \begin{tikzpicture}[thick,scale=0.8]

\tikzstyle{every node}=[circle,fill,draw,inner sep=1pt, minimum width=10pt,scale=0.8]

    \draw [circle,color=black] (6,4)  node (i1)[pin=above left:\color{black}1]{};
    \draw [circle,color=black] (8.5,3)  node (i2)[pin=above left:\color{black}2]{};
    \draw [circle,color=black] (7.5,2)  node (i3)[pin=above left:\color{black}2]{};
    \draw [color=black] (6.5,1)  node (i4)[pin=above left:\color{black}2]{};
    \draw [color=black] (4,3)  node (i5)[pin=above left:\color{black}1]{};
    \draw [color=black] (5,2)  node (i6)[pin=above right:\color{black}3]{};
    \draw [circle,color=black] (3,2)  node (i7)[pin=above left:\color{black}1]{};
    \draw [color=black] (2,1)  node (i8)[pin=above left:\color{black}1]{};
\tikzstyle{every node}=[inner sep=1pt, minimum width=14pt,scale=0.8]

    \draw (4.3,1)  node (l1){3};
    \draw (5.5,0)  node (l2){2};
    \draw (1,0)  node (l3){1};
    \draw (3,0)  node (l4){6};
    \draw (6,1)  node (l5){5};
    \draw (3.7,1)  node (l6){4};
    \draw (7.5,0)  node (l7){9};
    \draw (9.5,2)  node (l8){7};
    \draw (8.5,1)  node (l9){8};

    \draw (i1) --  (i2) ;
    \draw (i1) --  (i5) ;
    \draw (i2) --  (i3) ;
    \draw (i2) --  (l8) ;
    \draw (i3) --  (i4) ;
    \draw (i3) --  (l9) ;
    \draw (i4) --  (l7) ;
    \draw (i4) --  (l2) ;
    \draw (i5) --  (i6) ;
    
    \draw (i5) --  (i7) ;
    \draw (i6) --  (l5) ;
    \draw (i6) --  (l1) ;
    \draw (i7) --  (l6) ;
    \draw (i7) --  (i8) ;
    \draw (i8) --  (l3) ;
    \draw (i8) --  (l4) ;
\end{tikzpicture}
 \caption{Example of a Lyndon tree. The numbers above the  lines correspond to the valencies of the
internal nodes}
\label{fig:lyndonandvalency}
  \end{figure}

	We will say that an internal node $x$ of  a labeled binary tree $(T,\sigma)$ is a {\em
Lyndon node} if  (\ref{eq:lynnode}) holds. 
Hence Proposition~\ref{prop:valen} says that $(T,\sigma)$ is a Lyndon tree if and only if it is
normalized and all its internal nodes are Lyndon nodes. 

A {\it bicolored Lyndon tree} is a  normalized bicolored binary tree that satisfies the following
coloring restriction: for each internal node $x$ that is not a Lyndon node, $x$ is colored blue and 
its left child is colored red.
The set of bicolored Lyndon trees for $n=3$ is depicted in Figure \ref{fig:bicoloredlyndons}.

 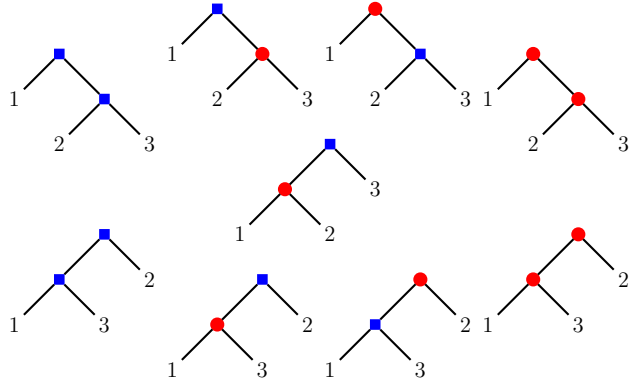
\begin{figure}[h]
        \centering
         \begin{tikzpicture}[thick,scale=0.6]
\begin{scope}[xshift=-1cm,yshift=-1cm]

\tikzstyle{every node}=[fill, draw,inner sep=2pt,scale=0.8]
    \draw [color=blue] (3,1)  node (i1){};
    \draw [color=blue] (2,2)  node (i2){};

\tikzstyle{every node}=[inner sep=1pt, minimum width=14pt,scale=0.7]

    \draw (2,0)  node (m){$2$};
    \draw (4,0)  node (l1){$3$};
    \draw (1,1)  node (l2){$1$};
    
    \draw (m) --  (i1) ;
    \draw (i1) --  (l1) ;
    \draw (i1) --  (i2) ;
    \draw (i2) --  (l2) ;
\end{scope}

\begin{scope}[xshift=2.5cm,yshift=0]

\tikzstyle{every node}=[fill, draw,inner sep=2pt,scale=0.8]
    \draw [circle,color=red] (3,1)  node (i1){};
    \draw [color=blue] (2,2)  node (i2){};

\tikzstyle{every node}=[inner sep=1pt, minimum width=14pt,scale=0.7]

    \draw (2,0)  node (m){$2$};
    \draw (4,0)  node (l1){$3$};
    \draw (1,1)  node (l2){$1$};
    
    \draw (m) --  (i1) ;
    \draw (i1) --  (l1) ;
    \draw (i1) --  (i2) ;
    \draw (i2) --  (l2) ;
\end{scope}

\begin{scope}[xshift=6cm,yshift=0]

\tikzstyle{every node}=[fill, draw,inner sep=2pt,scale=0.8]
    \draw [color=blue] (3,1)  node (i1){};
    \draw [circle,color=red] (2,2)  node (i2){};

\tikzstyle{every node}=[inner sep=1pt, minimum width=14pt,scale=0.7]

    \draw (2,0)  node (m){$2$};
    \draw (4,0)  node (l1){$3$};
    \draw (1,1)  node (l2){$1$};
    
    \draw (m) --  (i1) ;
    \draw (i1) --  (l1) ;
    \draw (i1) --  (i2) ;
    \draw (i2) --  (l2) ;
\end{scope}
\begin{scope}[xshift=9.5cm,yshift=-1cm]

\tikzstyle{every node}=[fill, draw,inner sep=2pt,scale=0.8]
    \draw [circle,color=red] (3,1)  node (i1){};
    \draw [circle,color=red] (2,2)  node (i2){};

\tikzstyle{every node}=[inner sep=1pt, minimum width=14pt,scale=0.7]

    \draw (2,0)  node (m){$2$};
    \draw (4,0)  node (l1){$3$};
    \draw (1,1)  node (l2){$1$};
    
    \draw (m) --  (i1) ;
    \draw (i1) --  (l1) ;
    \draw (i1) --  (i2) ;
    \draw (i2) --  (l2) ;
\end{scope}

\begin{scope}[xshift=0,yshift=-5cm]

\tikzstyle{every node}=[fill, draw,inner sep=2pt,scale=0.8]
    \draw [color=blue] (1,1)  node (i1){};
    \draw [color=blue] (2,2)  node (i2){};

\tikzstyle{every node}=[inner sep=1pt, minimum width=14pt,scale=0.7]

    \draw (0,0)  node (m){$1$};
    \draw (2,0)  node (l1){$3$};
    \draw (3,1)  node (l2){$2$};

    \draw (m) --  (i1) ;
    \draw (i1) --  (l1) ;
    \draw (i1) --  (i2) ;
    \draw (i2) --  (l2) ;
\end{scope}

\begin{scope}[xshift=3.5cm,yshift=-6cm]
\tikzstyle{every node}=[fill, draw,inner sep=2pt,scale=0.8]
    \draw [circle,color=red] (1,1)  node (i1){};
    \draw [color=blue] (2,2)  node (i2){};

\tikzstyle{every node}=[inner sep=1pt, minimum width=14pt,scale=0.7]

    \draw (0,0)  node (m){$1$};
    \draw (2,0)  node (l1){$3$};
    \draw (3,1)  node (l2){$2$};

    \draw (m) --  (i1) ;
    \draw (i1) --  (l1) ;
    \draw (i1) --  (i2) ;
    \draw (i2) --  (l2) ;
\end{scope}

\begin{scope}[xshift=7cm,yshift=-6cm]
\tikzstyle{every node}=[fill, draw,inner sep=2pt,scale=0.8]
    \draw [color=blue] (1,1)  node (i1){};
    \draw [circle,color=red] (2,2)  node (i2){};

\tikzstyle{every node}=[inner sep=1pt, minimum width=14pt,scale=0.7]

    \draw (0,0)  node (m){$1$};
    \draw (2,0)  node (l1){$3$};
    \draw (3,1)  node (l2){$2$};

    \draw (m) --  (i1) ;
    \draw (i1) --  (l1) ;
    \draw (i1) --  (i2) ;
    \draw (i2) --  (l2) ;
\end{scope}
\begin{scope}[xshift=10.5cm,yshift=-5cm]
\tikzstyle{every node}=[fill, draw,inner sep=2pt,scale=0.8]
    \draw [circle,color=red] (1,1)  node (i1){};
    \draw [circle,color=red] (2,2)  node (i2){};

\tikzstyle{every node}=[inner sep=1pt, minimum width=14pt,scale=0.7]

    \draw (0,0)  node (m){$1$};
    \draw (2,0)  node (l1){$3$};
    \draw (3,1)  node (l2){$2$};
    
    \draw (m) --  (i1) ;
    \draw (i1) --  (l1) ;
    \draw (i1) --  (i2) ;
    \draw (i2) --  (l2) ;
\end{scope}

\begin{scope}[xshift=5cm,yshift=-3cm]
\tikzstyle{every node}=[fill, draw,inner sep=2pt,scale=0.8]
    \draw [circle,color=red] (1,1)  node (i1){};
    \draw [color=blue] (2,2)  node (i2){};

\tikzstyle{every node}=[inner sep=1pt, minimum width=14pt,scale=0.7]

    \draw (0,0)  node (m){$1$};
    \draw (2,0)  node (l1){$2$};
    \draw (3,1)  node (l2){$3$};

    \draw (m) --  (i1) ;
    \draw (i1) --  (l1) ;
    \draw (i1) --  (i2) ;
    \draw (i2) --  (l2) ;

\end{scope}
\end{tikzpicture}
 \caption{Set of bicolored Lyndon trees for $n=3$}
\label{fig:bicoloredlyndons}
  \end{figure}

Clearly if a bicolored Lyndon tree is  monochromatic then all its nodes are Lyndon nodes.  Hence the
monochromatic ones are the classical Lyndon trees.  

Let $\lyn_{n,i}$ be the set of bicolored Lyndon trees in   $ \BT_{n,i}$.   We will show
 that the ascent-free chains of the EL-labeling of  $[\hat 0,[n]^i]$ 
given in Theorem~\ref{theorem:ellabelingposet} are of the form $c(T,\sigma,\tau)$, where $(T,\sigma)
\in \lyn_{n,i}$ 
and $\tau$ is a certain linear extension of the internal nodes of $T$, which we now describe.
It is easy to see that there is a unique linear extension of the internal notes of $(T,\sigma) \in 
\BT_{n,i}$ in which the valencies of the nodes weakly decrease.   Let $\tau_{T,\sigma}$ denote the
permutation that  induces this linear extension.

\begin{theorem}\label{thm:ascfreeEL}
The set $\{c(T,\sigma,\tau_{T,\sigma}) : \, (T,\sigma) \in \lyn_{n,i}\}$ is the set of ascent-free
maximal
chains of the EL-labeling of $[\hat 0, [n]^i]$ given in Theorem \ref{theorem:ellabelingposet}.
\end{theorem}

\begin{proof}

We begin by showing that $c:=c(T,\sigma,\tau)$ is ascent-free whenever $(T,\sigma) \in \lyn_{n,i}$
and $\tau= \tau_{T,\sigma}$ .  Let  $x_i$ be the $i$th internal node of $T$ in postorder.  Then by
the definition of $\tau:=\tau_{T,\sigma}$,
\begin{equation} \label{eq:val} v(x_{\tau(1)}) \ge v(x_{\tau(2)}) \ge \dots \ge 
v(x_{\tau(n-1)}),\end{equation} 
where $v$ is the valency.  For each $i$, the $i$th letter of  the label word $\lambda(c)$ is given
by 
$$\lambda_i(c) = (v(L(x_{\tau(i)})), v(R(x_{\tau(i)})))^{u_i}= (v(x_{\tau(i)}),
v(R(x_{\tau(i)})))^{u_i},$$ where  $u_i = 0$ if $x_{\tau(i)}$ is blue and is $1$ if $x_{\tau(i)}$ is
red.  Note that since $(T,\sigma)$ is normalized, $v(R(x_{\tau(i)})) \ne v(R(x_{\tau(i+1)}))$ for all $i \in [n-1]$.  Now suppose the word $\lambda(c)$ has an ascent  at $i$.  Then it follows from (\ref{eq:val})
that 
\begin{equation} \label{eq:ascent} v(x_{\tau(i)}) = v(x_{\tau(i+1)}),\,\,\, v(R(x_{\tau(i)}))<
v(R(x_{\tau(i+1)})),\,\mbox{ and } \, u_i \le u_{i+1} .\end{equation}  
The equality of valencies implies that $x_{\tau(i)}= L(x_{\tau(i+1)})$ since $(T,\sigma)$ is
normalized.  
Hence by (\ref{eq:ascent}),  $$ v(R(L(x_{\tau(i+1)})))<v(R(x_{\tau(i+1)})).$$
It follows that $x_{\tau(i+1)}$ is not a Lyndon node.  So by the coloring restriction on bicolored
Lyndon trees, 
$x_{\tau(i+1)}$ must be colored blue and its left child $x_{\tau(i)}$ must be colored red.  This
implies $u_i = 1$ and  
$u_{i+1} = 0$, which contradicts (\ref{eq:ascent}).  Hence the chain $c$ is ascent-free.

Conversely, assume $c$ is an ascent-free maximal chain of $[\hat 0, [n]^i]$.  Then $c =
c(T,\sigma,\tau)$ for some  bicolored labeled tree $(T,\sigma)$ and some permutation $\tau \in
\sym_{n-1}$.  We can assume without loss of generality that $(T,\sigma)$ is normalized.  Since $c$
is ascent-free, (\ref{eq:val}) holds.  This implies that $\tau$ is the unique permutation that
induces the valency-decreasing linear extension, namely $ \tau_{T,\sigma}$.  

If all internal nodes of $(T,\sigma)$ are Lyndon nodes we are done.  So  let $i\in [n-1]$ be such
that  $x_{\tau(i)}$ is not a Lyndon node.   That is $$  v(R(L(x_{\tau(i)})))< v(R(x_{\tau(i)})) .$$
  Since $(T,\sigma)$ is normalized and (\ref{eq:val}) holds,  $ L(x_{\tau(i)})= x_{\tau(i-1)}$. 
Hence,
$v(R(x_{\tau(i-1)}))<v(R(x_{\tau(i)})) $.  Since $(T,\sigma)$ is normalized we also have
$v(L(x_{\tau(i-1)})) = v(L(x_{\tau(i)}))$.  Hence to avoid an ascent at $i-1$ in $c$, we must color
$x_{\tau(i-1)}$ red and $x_{\tau(i)}$  blue, which is precisely what we need to conclude that
$(T,\sigma)$ is a bicolored Lyndon tree.
\end{proof}

From Theorem~\ref{elth}, Lemma~\ref{lemma:52} and Theorem~\ref{theorem:liehomisomorphism} we have
the following corollary.
\begin{corollary} \label{corollary:lyndonbasiscohomology} The set  $\{\bar c(T,\sigma) : \,
(T,\sigma) \in \lyn_{n,i}\}$ is a basis for $\tilde H^{n-3}((\hat 0, [n]^i))$ and the set
$\{[T,\sigma] : \, (T,\sigma) \in \lyn_{n,i}\}$ is a basis for $ \lie_2(n,i)$.
\end{corollary}

\begin{remark} Note that by switching left and right, small and large, blue and red, we get 8
different variations of bicolored Lyndon bases.
\end{remark}

\subsection{Liu's bicolored Lyndon basis} \label{sec:liu-lyn}  In this section we describe a
different generalization of the  Lyndon basis due to Liu \cite{Liu2010}. The basis we present
is actually a twisted version of the one in \cite{Liu2010} and has an easier description. The
two bases are related by a simple bijection. In Section~\ref{sec:treebasis} we will use this basis
to prove that a certain naturally constructed set of fundamental cycles is a basis for homology of
the interval $(\hat 0, [n]^i)$ .

We need to define a different valency from that of the previous section. This valency is referred to
in \cite{Liu2010} as the \emph{graphical root}. 
Recall that given an internal node $x$ of a binary tree,  $L(x)$ denotes the left child of $x$ and
$R(x)$ denotes the right child.  For each node $x$ of a bicolored labeled binary tree $(T,\sigma)$, 
define its  {\em valency} $v(x)$ recursively as follows: 
$$v(x) = \begin{cases} \mbox{label of } x &\mbox{if $x$ is a leaf} \\
\min\{v(L(x)),v(R(x))\} &\mbox{if $x$ is a blue internal node} \\
\max\{v(L(x)),v(R(x))\} &\mbox{if $x$ is a red internal node.}
\end{cases}$$

 A {\em Liu-Lyndon tree} is a bicolored labeled binary tree $(T,\sigma)$ such that for each 
internal node $x$ of $T$,
 \begin{enumerate}
 \item $v(L(x)) = v(x)$
  \item  if $x$ is $\blue$ and $L(x)$ is $\blue$ then $$v(R(L(x))) > v(R(x)) $$
 \item  if $x$ is $\red$ then $L(x)$ is $\red$ or is a leaf; in the former case, $$v(R(L(x)))
<v(R(x)) .$$
 \end{enumerate}
 
 Note that condition (1) is equivalent to the condition that $v(L(x)) < v(R(x))$ if $x$ is blue and
$v(L(x)) > v(R(x))$ if $x$ is red.  Note also that every subtree of a Liu-Lyndon tree is a
Liu-Lyndon tree.
 The set of Liu-Lyndon trees for $n=3$ is depicted in Figure \ref{fig:Liu-lyndon}. 
  
 \begin{figure}[h]
        \centering
         \begin{tikzpicture}[thick,scale=0.6]
\begin{scope}[xshift=-1cm,yshift=-1cm]

\tikzstyle{every node}=[fill, draw,inner sep=2pt,scale=0.8]
    \draw [color=blue] (3,1)  node (i1){};
    \draw [color=blue] (2,2)  node (i2){};

\tikzstyle{every node}=[inner sep=1pt, minimum width=14pt,scale=0.7]

    \draw (2,0)  node (m){$2$};
    \draw (4,0)  node (l1){$3$};
    \draw (1,1)  node (l2){$1$};
    
    \draw (m) --  (i1) ;
    \draw (i1) --  (l1) ;
    \draw (i1) --  (i2) ;
    \draw (i2) --  (l2) ;
\end{scope}

\begin{scope}[xshift=2.5cm,yshift=0]

\tikzstyle{every node}=[fill, draw,inner sep=2pt,scale=0.8]
    \draw [circle,color=red] (3,1)  node (i1){};
    \draw [color=blue] (2,2)  node (i2){};

\tikzstyle{every node}=[inner sep=1pt, minimum width=14pt,scale=0.7]

    \draw (2,0)  node (m){$3$};
    \draw (4,0)  node (l1){$2$};
    \draw (1,1)  node (l2){$1$};
    
    \draw (m) --  (i1) ;
    \draw (i1) --  (l1) ;
    \draw (i1) --  (i2) ;
    \draw (i2) --  (l2) ;
\end{scope}

\begin{scope}[xshift=6cm,yshift=0]

\tikzstyle{every node}=[fill, draw,inner sep=2pt,scale=0.8]
    \draw [color=blue] (3,1)  node (i1){};
    \draw [circle,color=red] (2,2)  node (i2){};

\tikzstyle{every node}=[inner sep=1pt, minimum width=14pt,scale=0.7]

    \draw (2,0)  node (m){$1$};
    \draw (4,0)  node (l1){$3$};
    \draw (1,1)  node (l2){$2$};
    
    \draw (m) --  (i1) ;
    \draw (i1) --  (l1) ;
    \draw (i1) --  (i2) ;
    \draw (i2) --  (l2) ;
\end{scope}
\begin{scope}[xshift=9.5cm,yshift=-1cm]

\tikzstyle{every node}=[fill, draw,inner sep=2pt,scale=0.8]
    \draw [circle,color=red] (3,1)  node (i1){};
    \draw [circle,color=red] (2,2)  node (i2){};

\tikzstyle{every node}=[inner sep=1pt, minimum width=14pt,scale=0.7]

    \draw (2,0)  node (m){$2$};
    \draw (4,0)  node (l1){$1$};
    \draw (1,1)  node (l2){$3$};
    
    \draw (m) --  (i1) ;
    \draw (i1) --  (l1) ;
    \draw (i1) --  (i2) ;
    \draw (i2) --  (l2) ;
\end{scope}

\begin{scope}[xshift=0,yshift=-5cm]

\tikzstyle{every node}=[fill, draw,inner sep=2pt,scale=0.8]
    \draw [color=blue] (1,1)  node (i1){};
    \draw [color=blue] (2,2)  node (i2){};

\tikzstyle{every node}=[inner sep=1pt, minimum width=14pt,scale=0.7]

    \draw (0,0)  node (m){$1$};
    \draw (2,0)  node (l1){$3$};
    \draw (3,1)  node (l2){$2$};

    \draw (m) --  (i1) ;
    \draw (i1) --  (l1) ;
    \draw (i1) --  (i2) ;
    \draw (i2) --  (l2) ;
\end{scope}

\begin{scope}[xshift=3.5cm,yshift=-6cm]
\tikzstyle{every node}=[fill, draw,inner sep=2pt,scale=0.8]
    \draw [circle,color=red] (1,1)  node (i1){};
    \draw [color=blue] (2,2)  node (i2){};

\tikzstyle{every node}=[inner sep=1pt, minimum width=14pt,scale=0.7]

    \draw (0,0)  node (m){$2$};
    \draw (2,0)  node (l1){$1$};
    \draw (3,1)  node (l2){$3$};

    \draw (m) --  (i1) ;
    \draw (i1) --  (l1) ;
    \draw (i1) --  (i2) ;
    \draw (i2) --  (l2) ;
\end{scope}

\begin{scope}[xshift=7cm,yshift=-6cm]
\tikzstyle{every node}=[fill, draw,inner sep=2pt,scale=0.8]
    \draw [circle,color=red] (1,1)  node (i1){};
    \draw [circle,color=red] (2,2)  node (i2){};

\tikzstyle{every node}=[inner sep=1pt, minimum width=14pt,scale=0.7]

    \draw (0,0)  node (m){$3$};
    \draw (2,0)  node (l1){$1$};
    \draw (3,1)  node (l2){$2$};
    
    \draw (m) --  (i1) ;
    \draw (i1) --  (l1) ;
    \draw (i1) --  (i2) ;
    \draw (i2) --  (l2) ;
\end{scope}
\begin{scope}[xshift=9.5cm,yshift=-5cm]

\tikzstyle{every node}=[fill, draw,inner sep=2pt,scale=0.8]
    \draw [circle,color=red] (3,1)  node (i1){};
    \draw [color=blue] (2,2)  node (i2){};

\tikzstyle{every node}=[inner sep=1pt, minimum width=14pt,scale=0.7]

    \draw (2,0)  node (m){$3$};
    \draw (4,0)  node (l1){$1$};
    \draw (1,1)  node (l2){$2$};
    
    \draw (m) --  (i1) ;
    \draw (i1) --  (l1) ;
    \draw (i1) --  (i2) ;
    \draw (i2) --  (l2) ;
\end{scope}

\begin{scope}[xshift=5cm,yshift=-3cm]
\tikzstyle{every node}=[fill, draw,inner sep=2pt,scale=0.8]
    \draw [color=blue] (3,1)  node (i1){};
    \draw [circle,color=red] (2,2)  node (i2){};

\tikzstyle{every node}=[inner sep=1pt, minimum width=14pt,scale=0.7]

    \draw (2,0)  node (m){$1$};
    \draw (4,0)  node (l1){$2$};
    \draw (1,1)  node (l2){$3$};
    
    \draw (m) --  (i1) ;
    \draw (i1) --  (l1) ;
    \draw (i1) --  (i2) ;
    \draw (i2) --  (l2) ;

\end{scope}
\end{tikzpicture}
 \caption{Set of  Liu-Lyndon trees for $n=3$}
\label{fig:Liu-lyndon}
  \end{figure}
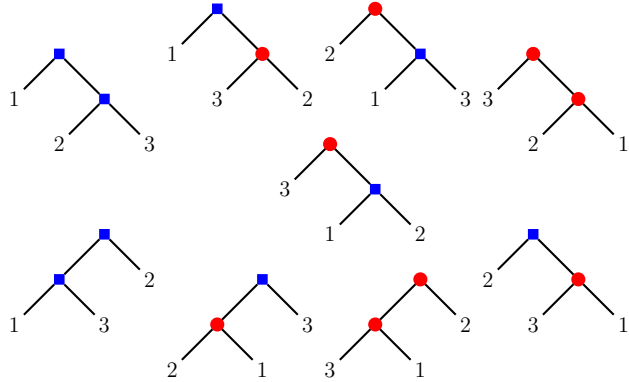
 
Let $\liu_{n,i} $ be the set of Liu-Lyndon trees in $\BT_{n,i}$.  When $i=0$, all internal nodes are
blue and it follows from the definition that $\liu_{n,0} $ is the set of Lyndon trees on $n$ leaves.
 When $i=n-1$, all internal nodes are red and it follows from the definition that $\liu_{n,n-1} $ 
consists of labeled binary trees  obtained from Lyndon  trees by replacing each  label $j$ by  label
$n-j$.

 In \cite{Liu2010} Liu proves that $\{ [T,\sigma] : (T,\sigma) \in \liu_{n,i} \} $ is a basis for
$\lie_{n,i}$ by using a perfect pairing between   $\lie_{n,i}$ and another module that she
constructs.
In the next section, we will use  the natural pairing between cohomology and homology of $(\hat 0,
[n]^i)$ to prove this result.  

We will need a bijection of Liu \cite{Liu2010}. Let $A$ be a finite subset of the positive integers
and let $0 \le i \le |A|-1$. Extend the definitions of $\mathcal T_{n,i}$ and $\liu_{n,i}$ by
letting $\mathcal T_{A,i}$  be the set of rooted trees on node set $A$ with $i$ descents and
$\liu_{A,i}$ be the set of Liu-Lyndon trees with leaf label set $A$ and $i$ red internal nodes. 
Define $\psi: \mathcal T_{A,i} \to \liu_{A,i}$ recursively as follows: if $|A|=1$, let $ \psi(T)$ be
the labeled binary tree whose single leaf is labeled with the sole element of $A$.  Now suppose $|A|
> 1$ and $r_T \in A$ is the root of $T$.  Let $x$ be the smallest child of $r_T$ that is larger than
$r_T$. If no such node exists let $x$ be the largest child of $r_T$. Let  $T_x$ be the subtree of $T$ rooted at $x$ 
and let
$T\setminus T_x$ be the subtree of $T$ obtained by removing $T_x$ from $T$.  Now let $$\psi(T) =
\psi(T\setminus T_x) \,\substack{\col \\ \wedge} \,\psi(T_x) ,$$ where $$\col = \begin{cases} \blue
& \mbox{ if } x > r_T \\ \red & \mbox{ if }x <r_T. \end{cases}$$

   It will be convenient to refer to descent edges of $T$  (i.e., edges $\{x,p_T(x)\}$, where $x <
p_T(x)$) as red edges, and nondescent edges (i.e., edges $\{x,p_T(x)\}$, where $x > p_T(x)$) as blue
edges. Hence
 $\psi$ takes  blue edges to  blue internal nodes and  red edges to  red internal nodes. 
Consequently $\psi(T) \in \BT_{A,i}$ if $T \in \mathcal T_{A,i}$.
By induction we see that the valuation of the root of $\psi(T)$ is equal to the root of $T$.   It
follows from this  that $\psi(T) \in \liu_{A,i}$.  It is not difficult to describe the inverse of
$\psi$ and thereby prove the following result.

\begin{proposition}[\cite{Liu2010}]\label{proposition:lyndontrees}
For all finite sets $A$ and $0 \le i \le |A|$, the map 
\[
\psi:\T_{A,i} \rightarrow \liu_{A,i}\]
is a well-defined bijection.
\end{proposition}

\begin{remark}\label{rem:equi} It follows from Corollary~\ref{proposition:dimensionhat}, 
Theorem~\ref{proposition:combbasiscohomology}, 
Corollary~\ref{corollary:lyndonbasiscohomology}  and Proposition~\ref{proposition:lyndontrees} that
$$|\T_{n,i}| = |\comb_{n,i}| = |\lyn_{n,i}| = |\liu_{n,i}| .$$  It would be desirable to find nice  bijections 
between the given sets like that of Proposition ~\ref{proposition:lyndontrees}.  
In \cite{Dleon2013b} Gonz\'alez D'Le\'on constructs such a bijection between  $ \comb_{n,i}$ and $
\lyn_{n,i}$.  We 
leave open the problem of finding a bijection between $\T_{n,i}$ and $ \comb_{n,i}$ or $ \lyn_{n,i}$.
\end{remark}

\subsection{The tree basis for homology} \label{sec:treebasis}
We now present a generalization of Bj\"orner's NBC basis for homology 
of $\overline{\Pi}_n$ (see \cite[Proposition~2.2]{Bjorner1982}).    
Recall that in Section~\ref{ranksec}, we associated a weighted partition $\alpha(F)$ with each 
forest 
$F= \{T_1,\dots,T_k\}$ on node set $[n]$, by letting  $$\alpha(F) = \{A_1^{w_1},\dots,
A_k^{w_k}\},$$ where $A_i$ is the node set  of $T_i$  and $w_i$ is the number of descents  of $T_i$.

 Let $T$ be a rooted tree on node set $[n]$.  For each subset 
$E$ of the edge set $E(T)$ of $T$, let $T_E$  be the subgraph of $T$
with node
set $[n]$ and edge set $E$.  Clearly $T_E$ is a forest on $[n]$. 
 We define $\Pi_T$ to be the induced subposet of 
 $\Pi_n^{w}$ on the set $\{\alpha(T_E) : E \in E(T)\}$.  See Figure~\ref{figpit} for an example
of $\Pi_T$. The poset $\Pi_T$ 
 is clearly isomorphic to the boolean algebra $\BB_{n-1}$.  Hence 
 $\Delta(\overline{\Pi_T})$ is the barycentric subdivision of the 
 boundary of the $(n-2)$-simplex.
We let $\rho_T$ denote
a fundamental cycle of the spherical complex $\Delta(\overline{\Pi_T})$, that is, a generator of the unique nonvanishing integral simplicial homology of $\Delta(\overline{\Pi_T})$.  Note that $\rho_T = \sum_{c \in \mathcal M(\Pi_T)} \pm \bar c$.

\begin{figure}[h]
        \centering
        \begin{subfigure}[b]{0.5\textwidth}
                \centering
                   \begin{tikzpicture}[thick,scale=0.8]
\tikzstyle{every node}=[circle, draw,inner sep=1pt, minimum width=14pt,scale=0.8]

    \draw (1,2)  node (v3){3};
    \draw (0,1) node (v4){4};
    \draw (2,1) node (v1){1};
    \draw (0,-0.5) node (v2){2};
    \draw [color=blue] (v3) --  (v4) ;
    \draw [color=red] (v3) --  (v1) ;
    \draw [color=red] (v4) --  (v2) ;

\end{tikzpicture}
                \caption{$T$}
        \end{subfigure}%
        ~ 
        \begin{subfigure}[b]{0.5\textwidth}
                \centering
                \begin{tikzpicture}[thick,scale=1]
\tikzstyle{every node}=[inner sep=1pt, minimum width=14pt,scale=0.6]
    \node (v1234) at (0,3){$1234^{2}$};
    \node (v13-24) at (-2,2){$13^{1}|24^{1}$};
    \node (v134-2) at (0,2){$134^{1}|2^{0}$};
    \node (v1-234) at (2,2){$1^{0}|234^{1}$};
    \node (v13-2-4) at (-2,1){$13^{1}|2^{0}|4^{0}$};
    \node (v1-24-3) at (0,1){$1^{0}|24^{1}|3^{0}$};
    \node (v1-2-34) at (2,1){$1^{0}|2^{0}|34^{0}$};
    \node (v1-2-3-4) at (0,0){$1^{0}|2^{0}|3^{0}|4^{0}$};

    \draw (v1234) -- (v13-24);
    \draw (v1234) -- (v134-2);
    \draw (v1234) -- (v1-234);
    \draw (v13-24) -- (v13-2-4);
    \draw (v13-24) -- (v1-24-3);
    \draw (v134-2) -- (v13-2-4);
    \draw (v134-2) -- (v1-2-34);
    \draw (v1-234) -- (v1-24-3);
    \draw (v1-234) -- (v1-2-34);
    \draw (v13-2-4) -- (v1-2-3-4);
    \draw (v1-24-3) -- (v1-2-3-4);
    \draw (v1-2-34) -- (v1-2-3-4);

\end{tikzpicture}
                \caption{$\Pi_T$}
        \end{subfigure}
 \caption{Example of a tree $T$ with two descent edges (red edges) and the corresponding poset
$\Pi_T$}
\label{figpit}
  \end{figure}
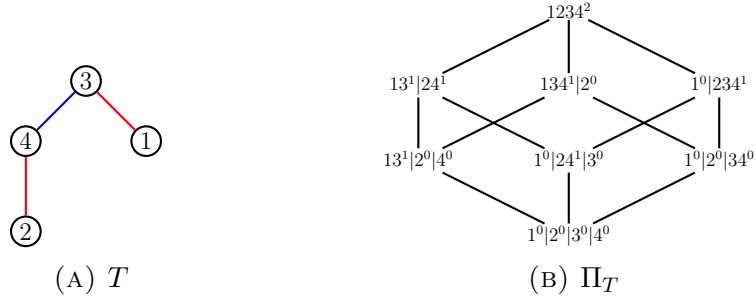

The set $\{\rho_T : T \in \mathcal T_{n,0} \}$ is   precisely 
the interpretation of the Bj\"orner NBC basis for homology of $\overline{\Pi}_n$ given in
\cite[Proposition~2.2]{Wachs1998}, and  the set $\{\rho_T : T \in \mathcal T_{n,n-1} \}$ is a
variation of this basis. Bj\"orner's NBC basis 
is dual to the Lyndon basis $\{\bar c(\Upsilon) :  \Upsilon \in  {\textsf{Lyn}_n}\}$ for cohomology of
$\overline{\Pi}_n$ (using the natural pairing between homology and cohomology).   While it is 
not true in general that $\{\rho_T : T \in \mathcal T_{n,i} \}$ is  dual  to any of the
generalizations of the bases given in the previous sections, 
we are able to prove that it is a basis by  pairing it with the Liu-Lyndon basis for cohomology.

\begin{theorem} \label{thm:treebasis} The set $\{\rho_T : T \in \mathcal T_{n,i} \}$ is a basis for
$\tilde H_{n-3}((\hat 0, [n]^i))$ and the set $\{\bar c(\Upsilon) : \Upsilon \in \liu_{n,i} \}$ is a
basis for $\tilde H^{n-3}((\hat 0, [n]^i))$.
\end{theorem}

Our main tool in proving this theorem is Proposition~\ref{proposition:dual} (of the Appendix), which
involves the bilinear form $\langle,\rangle$ defined in Appendix~\ref{section:homologyposets}.
 In order to apply 
Proposition~\ref{proposition:dual} we need   total orderings  of the sets $\mathcal T_{n,i}$ and
$\liu_{n,i}$.  Recall 
Liu's bijection $\psi: \mathcal T_{n,i} \to \liu_{n,i}$ given in
Proposition~\ref{proposition:lyndontrees}. We will show that any linear 
extension $\{T_1,T_2,\dots, T_{|\mathcal T_{n,i}|} \}$ of a certain partial   ordering  on $\mathcal
T_{n,i}$ provided by Liu \cite{Liu2010}  
yields a matrix $\langle \rho_{T_j}, \bar c(\psi(T_k) )\rangle_{1 \le j,k \le |\mathcal T_{n,i}|}$ 
that is upper-triangular with diagonal entries equal to $\pm 1$. Theorem~\ref{thm:treebasis}  will
then follow from Proposition~\ref{proposition:dual} and Theorem~\ref{theorem:homotopy}~(2).

We define Liu's partial ordering $\le_{\text{Liu}}$ of $\mathcal T_{A,i}$ recursively.  For $|A|\le
2$, the set $\mathcal T_{A,i}$ has only one element.  So assume that $|A| \ge 3$ and   that 
$\le_{\text{Liu}}$ has been defined for all $\mathcal T_{B,j}$ where $|B| < |A|$. Let $T, T^\prime
\in  \mathcal T_{A,i}$.   
We say that $T \preceq T^\prime$ if there exist edges $e$ of $T$ and $e^\prime$ of $T^\prime$ such
that the following conditions hold
\begin{itemize}
\item $e$ and $e^\prime $ have the same color,
\item $e^\prime$ contains the root of $T^\prime$,
\item $\alpha(T_{E(T)\setminus\{e\}}) = \alpha(T^\prime_{E(T^\prime)\setminus\{e^\prime\}}) $
\item $T_1 \le_{\text{Liu}} T_1^\prime$,
\item $T_2 \le_{\text{Liu}} T_2^\prime$,
\end{itemize}
where $T_1$ and $T_2$ are the connected components (trees) of  the forest  obtained by removing $e$
from $T$, and $T^\prime_1$ and $T^\prime_2$ are the corresponding connected components (trees) of 
the forest  obtained by removing $e^\prime$ from $T^\prime$.

Now define $\le_{\text{Liu}}$ to be the transitive closure of the relation $\preceq$ on $\mathcal
T_{A,i}$.  It follows from \cite[Lemma 8.12]{Liu2010} that this relation is the same as the relation
$\le_{\text{op}}$ that was defined in \cite[Definition 7.11]{Liu2010} and was proved to be a partial
order in  \cite[Lemma 7.13]{Liu2010}.  

\begin{lemma} \label{lem:liuorder} Let $T,T^\prime \in \mathcal T_{n,i}$ and let $\psi: \mathcal
T_{n,i} \to \liu_{n,i}$ be the bijection of Proposition~\ref{proposition:lyndontrees}.  If $
c(\psi(T^\prime)) \in \mathcal M(\Pi_T)$ then $T\le_{\text{Liu}} T^\prime$.
\end{lemma}

\begin{proof}  First note that if  $\Upsilon_1 \substack{\col \\ \wedge} \Upsilon_2$ is a bicolored
labeled binary tree such that  $c(\Upsilon_1 \substack{\col \\ \wedge} \Upsilon_2) $ is a maximal
chain in $\Pi_T$ then there is an edge $e$ of $T$ whose color equals $\col$ and whose removal  from
$T$ yields a forest whose connected components (trees) $T_1$ and $T_2$ satisfy: $c(\Upsilon_1)$  is
a maximal chain in $\Pi_{T_1}$ and $c(\Upsilon_2)$  is a maximal chain in $\Pi_{T_2} $.

Now recalling the definition of $\psi$, let $x$ be the child of the root $r_{T^\prime}$ of
$T^\prime$, for which 
$$\psi(T^\prime) = \psi(T^\prime\setminus T^\prime_x) \,\substack{\col \\ \wedge} \,\psi(T^\prime_x)
,$$
where $\col$ equals the color of the edge  $\{x,r_{T^\prime}\}$.  Let $e$ be the edge of $T$ whose
removal yields the subtrees $T_1$ and $T_2$ such that $c(\psi(T^\prime\setminus T^\prime_x)) \in
\mathcal M(\Pi_{T_1})$ and $c(\psi(T^\prime_x)) \in \mathcal M(\Pi_{T_2})$.  Then the color of $e$
is the same as that of  the edge $\{x,r_{T^\prime}\}$.  By induction we can assume that
$$T_1 \le_{\text{Liu}} T^\prime\setminus T^\prime_x \,\,\,\mbox{ and }\,\,\, T_2 \le_{\text{Liu}}
T^\prime_x .$$
Since $e$ and $e^\prime:= \{x,r_{T^\prime}\}$ satisfy the conditions of the definition of $\preceq$,
we have $T \preceq T^\prime$, which implies the result.
\end{proof} 

\begin{proof}[Proof of Theorem~\ref{thm:treebasis}]  Let $T_1,\dots,T_m$ be any linear extension of
$\le_{\text{Liu}}$ on $\mathcal T_{n,i}$, where $m = |\mathcal T_{n,i}|$.    It follows from
Lemma~\ref{lem:liuorder} that the matrix $$M:=\langle \rho_{T_j}, \bar c(\psi(T_k) )\rangle_{1 \le
j,k \le m}$$ 
 is upper-triangular, where $\langle , \rangle$ is the bilinear form defined in Appendix~\ref{section:homologyposets}. 
Since $c(\psi(T))$ is a maximal chain of  $\Pi_T$ for all $T\in \mathcal T_{n,i}$, the diagonal
entries of  $M$ are equal to $\pm 1$.  Hence $M$ is invertible over $\ZZ$ or any field.  The result
now follows from Propositions~\ref{proposition:lyndontrees} and~\ref{proposition:dual} and
Theorem~\ref{theorem:homotopy} (2).
\end{proof}

\begin{remark} Theorems~\ref{theorem:liehomisomorphism} and~\ref{thm:treebasis}  yield an
alternative proof of Liu's result that 
$\{ [T,\sigma] : (T,\sigma) \in \liu_{n,i} \} $ is a basis for $\lie_{n,i}$. \end{remark}

\subsection{Bases for cohomology of the full weighted partition poset}

In this section we use  bicolored combs and bicolored Lyndon trees  to construct   bases for 
$\tilde H^{n-2}(\Pi_n^w \setminus{\{\hat 0}\})$. 

For a chain $c$ in $\Pi_n^w$, let $$\breve c:= c \setminus \{\hat 0\}.$$  
The codimension 1 chains of  $\Pi_n^w \setminus \{\hat 0\}$ are of the form $\breve c$, where $c$ is
either
\begin{enumerate}
\item unrefinable in some maximal interval $[\hat 0, [n]^i]$ except between one pair of adjacent elements $x < y$, where $[x,y]$ is 
an interval of length 2 in $[\hat 0, [n]^i]$, or 
 \item unrefinable in $[\hat 0, x]$, where $x$   is a weighted partition of $[n]$ consisting of  exactly two blocks.
 \end{enumerate}
The former case yields the cohomology relations of Types I, II and III given in Section~\ref{section:genhom}, with 
$\bar c$ replaced by $\breve c$.  The latter case yields 
 the additional cohomology relation:
\begin{enumerate}
\item[{\bf Type IV}:] The  two blocks of $x$ are either $0$-merged to
get a single-block partition $z_1$ or $1$-merged to get a single-block partition $z_2$. The open
interval
$(x,\hat{1})$ is equal to $\{z_1,z_2\}$, see Figure \ref{fig:type4}. Hence the Type IV elementary
cohomology relation is
$$(\breve c \cup \{z_1\} )+  (\breve c \cup \{z_2\} ) = 0.$$
\end{enumerate}

\begin{figure}
\centering
\begin{tikzpicture}[line join=bevel,scale=0.8]
\begin{scope}
  \tikzstyle{every node}=[inner sep=0pt, scale=0.65, minimum width=4pt]
  \node (v1-2-3-4) at (0,0)  {$A^{a}| B^{b}$};
  \node (v12-34) at (0,4)  {$\hat{1}$};
  \node (v12-3-4) at (-2,2)  {$AB^{a+b}$};
  \node (v1-2-34) at (2,2)  {$AB^{a+b+1}$};
  \draw [] (v12-34) -- (v12-3-4);
  \draw [] (v12-34) -- (v1-2-34);
  \draw [] (v1-2-34) -- (v1-2-3-4); 
  \draw [] (v12-3-4) -- (v1-2-3-4);
 \end{scope}
\end{tikzpicture}
\caption{Type IV cohomology relation}
\label{fig:type4}
\end{figure}
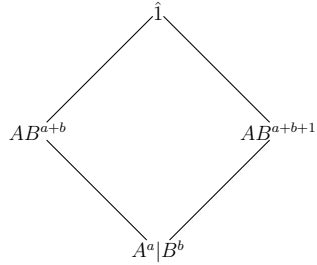

 The reader can verify, using the cohomology relations of Type~I (with $\bar c$ replaced by, $\breve c$), that the 
 proof of Lemma
\ref{lemma:52} goes through for $\tilde H^{n-2}(\Pi_n^w \setminus\{\hat 0\})$.   
Hence $\tilde H^{n-2}(\Pi_n^w \setminus\{\hat 0\})$ 
is generated by chains of the form $\breve c(\Upsilon)$ where $\Upsilon \in \BT_n$.  The  reader can also check, 
using the  relations of Types~I,~II, and~III, that the relations in 
Theorem~\ref{proposition:binarybasishomology}  hold (with $\bar c$ replaced by $\breve c$). It follows from the 
cohomology relation of Type~IV that 
\begin{equation} \label{eq:type4}
 \breve c(\Upsilon_1
\substack{\red \\ \wedge} \Upsilon_2) = - \breve c(\Upsilon_1
\substack{\blue \\ \wedge}\Upsilon_2),\end{equation}
  for all  $\Upsilon_1
\substack{\red \\ \wedge} \Upsilon_2 \in \BT_n$.

Recall $\comb_n = \bigcup_{i=0}^{n-1} \comb_{n,i}$ and let $\lyn_n = 
\bigcup_{i=0}^{n-1} \lyn_{n,i}$.
\begin{theorem}\label{proposition:combbasiscohomologyhat}
 The sets $$\{\breve c(T,\sigma) : \,(T,\sigma) \in \comb_n, \, \col(\root(T))=
\blue \}$$ and $$\{\breve c(T,\sigma) : \,(T,\sigma) \in \lyn_n, \, \col(\root(T)) = \red \}$$
are bases for 
 $\tilde H^{n-2}(\Pi_n^w \setminus\{\hat 0\})$. 
\end{theorem} 

\begin{proof}  

{\em The Comb Basis}:  We prove, by induction on the size $r(\Upsilon)$ of the right subtree of $\Upsilon$,  that if 
$\Upsilon$ is a normalized tree in  $\BT_n$ then $\breve c(\Upsilon)$ can be expressed as a linear combination of 
chains of the form $\breve c(\Upsilon^\prime)$, where $\Upsilon^\prime$ is a blue-rooted bicolored comb.   Since 
the relations in 
Theorem~\ref{proposition:binarybasishomology}  hold (with $\bar c$ replaced by $\breve c$),
we can use  the straightening algorithm in the proof of Proposition~\ref{proposition:combbasisspans} to express 
 $\breve c(\Upsilon)$  as a linear combination of chains of the form $\breve c(\Upsilon^\prime)$, where 
 $\Upsilon^\prime$ is a bicolored comb whose right subtree has size at most $r(\Upsilon)$.  
 If $\Upsilon^\prime$ is red-rooted we can use relation~(\ref{eq:type4}) to change the root color to blue.  The only 
 way that the modified blue-rooted $\Upsilon^\prime$  will fail to be a bicolored comb is if the right child of its root is 
 blue, in which case we can apply
 Case~1 of the straightening algorithm to $\Upsilon^\prime$.   We thus have that $\breve c(\Upsilon^\prime)$ is a 
 linear combination of two chains   
 $\breve c(\Upsilon_1)$ and $\breve c(\Upsilon_2)$, where each 
 $\Upsilon_i \in \BT_n$ and $r(\Upsilon_i)  < r(\Upsilon^\prime) \le r(\Upsilon)$.  
 By induction, each $\breve c(\Upsilon_i)$  is a linear combination of chains 
 associated with blue-rooted bicolored combs. The same is thus true for each $\breve c(\Upsilon^\prime)$ and for 
 $\breve c(\Upsilon)$.  Hence 
 $\{\breve c(T,\sigma) : \,(T,\sigma) \in \comb_n, \, \col(\root(T))=\blue \}$ 
spans.  We conclude that this set is a basis by the step in the proof of 
Proposition~\ref{proposition:combbasiscardinality}
that shows that there are $(n-1)^{n-1}$ blue-rooted combs and Corollary~\ref{proposition:dimensionhat}.

{\em The Lyndon Basis:} 
From the EL-labeling of
Theorem
\ref{theorem:ellabelingposet} we have that all the maximal chains of $\widehat{\Pi_n^w}$ have last 
label
$(1,n+1)^0$. Then for a maximal chain to be ascent-free it must have a second to last label of
the form $(1,a)^1$ for $a \in [n]$.   By
Theorem~\ref{thm:ascfreeEL}, we see that the ascent-free chains correspond to red-rooted 
bicolored Lyndon trees.  It therefore follows from Theorem~\ref{elth} and Lemma~\ref{lemma:52} (with $\bar c$ 
replaced by $\breve c$) that  the second set is a basis for $\tilde H^{n-2}(\Pi^w_n \setminus\{\hat 0\})$.
\end{proof}

Since the comb basis  was shown to span $\tilde
H^{n-3}(\Pi_n^w\setminus \{\hat 0 \})$ by using only the relations of  
Theorem~\ref{proposition:binarybasishomology} and relation (\ref{eq:type4}) we can conclude that these are the 
only relations in a presentation of $\tilde
H^{n-3}(\Pi_n^w\setminus \{\hat 0 \})$.  We summarize with the following result.

\begin{theorem}
 The set $\{ \breve{c}(\Upsilon) : \Upsilon \in \BT_{n}\}$ is a generating set for $\tilde
H^{n-3}(\Pi_n^w\setminus \{\hat 0 \})$,
subject only to the relations of Theorem~\ref{proposition:binarybasishomology} (with $\bar c$ replaced by 
$\breve c$) and relation (\ref{eq:type4}).
\end{theorem}

\section{Whitney cohomology} \label{section:whitney}
Whitney cohomology (over the field ${\bf k}$) of a poset $P$ with a minimum element $\hat 0$ is
defined for each integer $r$ as follows
$$WH^r(P):= \oplus_{x \in P} \tilde H^{r-2}((\hat 0, x);{\bf k}).$$  Whitney (co)homology was
introduced in \cite{Bac1975} and   further studied in \cite{Sund1994, Wachs1999}.
It is shown in \cite{OrlikSolomon1980} that if $P$ is a geometric lattice then there is a vector
space isomorphism between  $\oplus_r WH^r(P)$ and the Orlik-Solomon algebra of $P$ that becomes a
graded $G$-module isomorphism when $G$ is a group acting on $P$. 
The symmetric group $\sym_n$ acts naturally on $WH^r(\Pi_n)$ and on the multilinear component
$\land^r \lie(n)$,  of the  $r$th exterior power of the free Lie algebra on $[n]$.
In \cite{BarceloBergeron1990} Barcelo and Bergeron, working with the Orlik-Solomon algebra,
establish 
the  following $\sym_n$-module isomorphism 
$$WH^{n-r}(\Pi_n) \simeq_{\sym_n} \land^r\lie(n) \otimes \sgn_n.$$
In \cite{Wachs1998} Wachs shows that an extension of her correspondence between generating sets of
$\tilde H^{n-3}(\overline{\Pi}_n)$ and $\lie(n) \otimes \sgn_n$ can be used to prove this result.

Let $\land^r \lie_2(n)$ be the multilinear component of the exterior algebra of the  free Lie
algebra on $[n]$ with two compatible brackets.  A {\em bicolored binary forest} is a sequence of bicolored
binary trees.  Given a bicolored binary forest $F$ with $n$ leaves and  $\sigma \in \sym_n$, let
$(F,\sigma)$ denote the {\it labeled}  bicolored binary forest  whose $i$th leaf from left to right
has label $\sigma(i)$.  Let $\mathcal {BF}_{n,r}$ be the set of labeled  bicolored binary forests
with $n$ leaves and $r$ trees.   If the $j$th labeled bicolored binary tree of 
$(F,\sigma)$ is $(T_j,\sigma_j)$ for each $j=1,\dots r$ then define
$$[F,\sigma] := [T_1,\sigma_1] \land \dots \land [T_r,\sigma_r],$$ where
now $\land$ denotes the wedge product operation in the exterior algebra.  The set $\{[F,\sigma] :
(F,\sigma) \in \mathcal {BF}_{n,r}\}$ is a generating set for $\land^r \lie_2(n)$.

The set $\mathcal {BF}_{n,r}$ also provides a natural generating set for $WH^{n-r}(\Pi_n^w)$.  For
$(F,\sigma) \in \mathcal {BF}_{n,r}$, let $ c(F,\sigma)$ be the unrefinable chain of $\Pi_n^w$ whose
rank $i$ partition is obtained from its rank $i-1$ partition by $\col_i$-merging the blocks $L_i$
and $R_i$, where $\col_i$ is the color of the $i$th postorder internal node $v_i$ of $F$, and $L_i$ 
and $R_i$ are  the respective sets of leaf labels in the left and right subtrees of $v_i$.

The symmetric group $\sym_n$ acts naturally on $\land^r \lie_2(n)$ and on $WH^r(\Pi_n^w)$ for each
$r$.
We have the following generalization of Theorem~\ref{theorem:liehomisomorphism} and \cite[Theorem
7.2]{Wachs1998}.  The proof is similar to that of Theorem~\ref{theorem:liehomisomorphism} and is
left to the reader.
\begin{theorem}\label{th:whitney} For each $r$, there is an $\sym_n$-module isomorphism
$$\phi:\land^r \lie_2(n) \to WH^{n-r}(\Pi_n^w) \otimes \sgn_n$$ determined by
$$\phi([F,\sigma]) = \sgn(\sigma) \sgn(F) \bar c(F,\sigma), \qquad (F,\sigma) \in \mathcal
{BF}_{n,r},$$
where if $F$ is the sequence $T_1,\dots,T_r$ of bicolored binary trees  then 
$$\sgn(F) :=  (-1)^{I(T_2)+I(T_4) + \dots + I(T_{2\lfloor r/2 \rfloor})} \sgn(T_1) \sgn(T_2) \dots
\sgn(T_r).$$
\end{theorem}

\begin{corollary}\label{cor:whitdim} For $0 \le r \le n-1$, $$\dim \land^{n-r} \lie_2(n) = \dim WH^{r}(\Pi_n^w) =
\binom{n-1}r n^r.$$
Moreover if $\land \lie_2(n)$ is the multilinear component of the exterior algebra of the free Lie
algebra on $n$ generators and $WH(\Pi_n^w) = \oplus_{r \ge 0} WH^r(\Pi_n^w)$ then
$$\dim \land \lie_2(n) = \dim WH(\Pi_n^w) = (n+1)^{n-1}.$$
\end{corollary}

\begin{proof} Since $\dim WH^{r}(\Pi_n^w)$ equals the signless $r$th Whitney number of the first
kind $|w_r(\Pi_n^w)|$, the result follows from Theorem~\ref{th:whitney}, equation (\ref{whiteq}),
and the binomial formula.
\end{proof}

For a result that is closely related to Corollary~\ref{cor:whitdim}, see 
\cite[Theorem~2]{BershteinDotsenkoKhoroshkin2007}.

\section{Related work}\label{section:futurework} 
In \cite{Dleon2013a} Gonz\'alez D'Le\'on considers a more general version of $\Pi_n^{w}$   and uses
it to study $\lie_k(n)$,
the multilinear component  of the free Lie algebra with $k$ compatible brackets, where $k$ is an
arbitrary positive integer. In particular, he
uses an EL-labeling of the generalized version of $\Pi_n^{w}$ to obtain a combinatorial description
of the dimension of $\lie_k(n)$. 
This answers a question posed by Liu \cite{Liu2010} on how to 
generalize $\lie(n)$ further and to
find the right combinatorial objects to compute the dimensions.
The comb basis and the Lyndon basis are also further generalized in this paper to
multicolored versions.

By Theorem~\ref{proposition:combbasiscohomology} and 
Corollary~\ref{corollary:lyndonbasiscohomology} we conclude that the set of bicolored combs 
 and bicolored
Lyndon trees are equinumerous (cf. Remark~\ref{rem:equi}).   In 
\cite{Dleon2013a} Gonz\'alez D'Le\'on presents bijections between
the multicolored combs, multicolored Lyndon trees and a certain class of permutations, which generalize  the
classical
bijections between the sets of combs, Lyndon trees and permutations in $\sym_{n-1}$.

It can be concluded from equation (\ref{equation:drake}) that the generating polynomial of rooted
trees enumerated by number of descents
$\sum_{i = 0}^{n-1} |\T_{n,i}| t^i$ has only negative real roots.  Since the 
polynomial is also
palindromic (or symmetric), this implies it can be written using nonnegative coefficients in the basis 
$\{t^i(1+t)^{n-1-2i}\}_{i=0}^{\lfloor \frac{n-1}{2} \rfloor}$, 
a property known as $\gamma$-positivity. In \cite{Dleon2013c} the $\gamma$-positivity
 property is discussed further and generalized. In particular, formulas and 
combinatorial interpretations of the $\gamma$-coefficients in terms of sets of normalized labeled binary trees are
provided.

In a forthcoming paper we will study  a more general weighted partition poset obtained by associating weights 
to the bonds of an arbitrary graph on $n$-vertices.

\section*{Acknowledgement}
The authors would like to thank the referee for  valuable comments and 
 for pointing out an error in an earlier proof of Proposition~\ref{proposition:combbasisspans}.

\appendix
\section{Homology and Cohomology of a Poset}
\label{section:homologyposets}
We give a brief review of poset (co)homology with group actions.  For further information see
\cite{Wachs2007}.  

Let $P$ be a
 finite poset  of length $\ell$. The reduced simplicial (co)homology of  $P$ is defined to be the
reduced simplicial (co)homology of its order complex $\Delta(P)$, where $\Delta(P)$ is the
simplicial complex whose faces are the chains of $P$.  We will review the definition here by dealing
directly with the chains of $P$, and not resorting to the order complex of $P$.
  
  Let ${\bf k}$ be an arbitrary field or the ring of integers $\ZZ$. The  (reduced) chain and
cochain complexes
\[
 \cdots
\stackrel{\xrightarrow{\partial_{r+1}}}{\xleftarrow[\hspace*{0.05in}\delta_{r}\hspace*{0.05in}]{}}
C_r(P) 
\stackrel{\xrightarrow{\hspace*{0.05in}\partial_{r}\hspace*{0.05in}}}{\xleftarrow[\delta_{r-1}]{}}C_
{r-1}(P)
\stackrel{\xrightarrow{\partial_{r-1}}}{\xleftarrow[\delta_{r-2}]{}}\cdots
\] are defined
by letting $C_r(P)$  be the ${\bf k}$-module generated by the chains of length $r$ in $P$, for each integer $r$, 
and letting the boundary maps $\partial_r:C_r(P)\rightarrow C_{r-1}(P)$ be defined on chains by 
\[
\partial_r(\alpha_0<\alpha_1<\cdots<\alpha_r)=\sum_{i=0}^{r}(-1)^{i}(\alpha_0<\cdots
<\hat{\alpha_i}<\cdots<\alpha_r),
\]
where $\hat{\alpha_i}$ means that the element $\alpha_i$ is omitted from the chain.  Note that  
$C_{-1}(P)$ is generated by the empty chain and $C_r(P) = (0)$ if $r< -1$ or $r>\ell$.

Let $\langle,\rangle$ be the bilinear form on $\bigoplus _{r=-1}^{\ell} C_r(P)$ for which the chains of $P$ form an orthonormal basis.   
This allows us to define the coboundary map $\delta_r:C_r(P)\rightarrow C_{r+1}(P)$ by
\[
\langle \delta_r(c),c' \rangle=\langle c,\partial_{r+1}(c') \rangle.
\]
Equivalently,
\begin{equation}\label{equation:cohomologyboundary}
\delta_r(\alpha_0<\cdots<\alpha_r)=\sum_{i=0}^{r+1}(-1)^{i}\sum_{\alpha \in
(\alpha_{i-1},\alpha_i)}(\alpha_0<\cdots <\alpha_{i-1}<\alpha<\alpha_i<\cdots<\alpha_r),
\end{equation}
for all chains $\alpha_0<\cdots<\alpha_r$, where  $\alpha_{-1}=\hat{0}$ and $\alpha_{r+1}=\hat{1}$ of the augmented poset $\hat P$ in which a minimum element $\hat 0$ and a maximum element $\hat 1$ have been adjoined to $P$. 

Let $r \in \ZZ$.  Define the cycle space
$Z_r(P):= \ker \partial_r$ and the boundary space
$B_r(P) := \im \partial_{r+1}$.  Homology of the poset $P$ in dimension $r$
is defined by
$$\tilde H_r(P) : = Z_r(P) / B_r(P).$$
Define the cocycle space
$Z^r(P):= \ker \delta_r$ and the coboundary space
$B^r(P) := \im \delta_{r-1}$.  Cohomology of the poset $P$ in dimension $r$
is defined by
$$\tilde H_r(P) : = Z^r(P) / B^r(P).$$

For $x \le y$ consider the open interval $(x,y)$ of $P$.  Note that if $y$ covers $x$ then $(x,y)$ is the empty poset whose only chain is the empty chain. Therefore $\tilde H_r((x,y))= \tilde H^r((x,y)) = 0$ unless $r=-1$, in which case $\tilde H_r((x,y)) = \tilde H^r((x,y)) = \bold k$.  If $y=x$ then we adapt the convention that $\tilde H_r((x,y))= \tilde H^r((x,y)) = 0$ unless $r=-2$, in which case $\tilde H_r((x,y)) = \tilde H^r((x,y)) = \bold k$.

\begin{proposition}\label{prop:free} Let $P$ be a finite poset of length $\ell$ whose order complex
has the homotopy type of a wedge of $m$ spheres of dimension $\ell-2$.  
 Then $\tilde H_{\ell-2}(P) $ and $\tilde H^{\ell-2}(P) $ are isomorphic free ${\bf k}$-modules of
rank $m$.
\end{proposition} 
 
 The following proposition gives a useful tool for identifying bases for top homology and top
cohomology.

\begin{proposition}[see {\cite[Theorem 1.5.1]{Wachs2007}}, {\cite[Proposition 6.4]{ShWa2007}}]
\label{proposition:dual} Let $P$ be a finite poset of length $\ell$ whose order complex has the
homotopy type of a wedge of $m$ spheres of dimension $\ell-2$.  
Let $\{\rho_1,\rho_2,...,\rho_m\} \subseteq Z_{\ell-2}(P)$ and $\{\gamma_1,\gamma_2,...,\gamma_m\}
\subseteq Z^{\ell-2}(P)$. If the  matrix $(<\rho_i,\gamma_j>)_{i,j \in [m]}$
is invertible over ${\bf k}$ then the sets $\{\rho_1,\rho_2,...,\rho_m\}$ and
$\{\gamma_1,\gamma_2,...,\gamma_m\}$
are bases for $\tilde H_{\ell-2}(P;{\bf k})$ and $\tilde H^{\ell-2}(P;{\bf k})$ respectively.
\end{proposition}

Let $G$ be a finite group. A $G$-poset is a poset $P$
together with a $G$-action on its elements that preserves the partial order; i.e., $x <
y \implies  gx < gy$ in $P$.  

Now assume that ${\bf k}$ is a field. Let $P$ be a $G$-poset and let $0 \le r \le \ell$. Since $g
\in G$ takes $r$-chains to $r$-chains, $g$ acts as a linear
map on the chain space $C_r(P)$ (over ${\bf k}$). It is easy to see that for all $g \in G$ and $c
\in C_r(P)$,
$$g \partial_r (c) = \partial_r(gc)\,\, \mbox{ and } \,\, g\delta_r(c) = \delta_r(gc).$$
Hence $g$ acts as a linear map on the vector spaces $\tilde H_r(P)$ and on $\tilde H^r(P)$. This
implies that
whenever $P$ is a $G$-poset, $\tilde H_r(P)$ and $\tilde H^r(P) $ are $G$-modules. The bilinear
form $\langle , \rangle$, induces a pairing between $\tilde H_r(P) $ and $H^r(P)$, which allows one
to view them as dual $G$-modules. For $G = \sym_n$ we have the $\sym_n$-module isomorphism
\begin{equation}\label{eq:isocohom} \tilde H_r(P) \simeq_{\sym_n} \tilde H^r(P)\end{equation}since
dual $\sym_n$-modules are isomorphic.

\bibliographystyle{plain}
\def\cprime{$'$}

\end{document}